\documentclass[10pt]{amsart}
\usepackage{amssymb, amscd, amsmath, amsthm, epsf, epsfig, latexsym, enumerate}

\newtheorem{theorem}{Theorem}
\newtheorem{lemma}[theorem]{Lemma}
\newtheorem{cor}[theorem]{Corollary}

\begin{document}

\title{$PD_4$-complexes and 2-dimensional duality  groups}

\author{Jonathan A. Hillman }
\address{School of Mathematics and Statistics\\
     University of Sydney, NSW 2006\\
      Australia }

\email{jonathan.hillman@sydney.edu.au}

\begin{abstract}
This paper is a synthesis and extension of three earlier papers on 
$PD_4$-complexes $X$ with fundamental group $\pi$ 
such that $c.d.\pi=2$ and $\pi$ has one end. 
Our goal is to show that the homotopy types of such complexes 
are determined by $\pi$, the Stiefel-Whitney classes and the equivariant intersection pairing on $\pi_2(X)$.
We achieve this under further conditions on $\pi$.
\end{abstract}

\keywords{cohomological dimension, homotopy intersection, $k$-invariant, 
$PD_4$-complex}

\subjclass{57P10}

\maketitle

It remains an open problem to give a homotopy classification of 
closed 4-manifolds or $PD_4$-complexes,  
in terms of standard invariants such as the fundamental
group, characteristic classes and intersection pairings.
Hambleton and Kreck showed that if $X$ is orientable and 
$H_2(X;\mathbb{Q})\not=0$ the homotopy type of $X$ is determined by its
Postnikov 2-stage $P_2(X)$ and the image of 
the fundamental class $[X]$ in $H_4(P_2(X);\mathbb{Z})$, 
and if $\pi$ is finite and of cohomological period dividing 4 this image
is in turn determined by the equivariant intersection pairing on $\pi_2(X)$ \cite{HK88}.
Baues and Bleile have extended the first part of this result to all $PD_4$-complexes:
two $PD_4$-complexes $X$ and $Y$ are homotopy equivalent if and only if 
there is a homotopy equivalence $h:P_2(X)\to{P_2(Y)}$ such that $h^*w_1(Y)=w_1(X)=w$, say, 
and which carries the image of $[X]$ in $H_4(P_2(X);\mathbb{Z}^w)$ 
to the image of $\pm[Y]$ in $H_4(P_2(Y);\mathbb{Z}^w)$.
They also give a homotopy classification of $PD_4$-complexes (up to 2-torsion) 
in terms of homotopy classes of chain complexes with a homotopy 
commutative diagonal and an additional quadratic structure \cite{BB}.
However, there is still the question of how to characterize
the classes in $H_4(P_2(X);\mathbb{Z}^w)$ which
correspond to $PD_4$-complexes.

We shall extend the work of \cite{HK88} to relate these classes 
in $H_4(P_2(X);\mathbb{Z}^w)$ to intersection pairings, 
for certain cases with $\pi=\pi_1(X)$ infinite.
The central idea is that of ``strongly minimal $PD_4$-complex", 
one for which the equivariant intersection pairing is identically 0.
(We shall in fact use the equivalent cohomological pairing.)
If there is a 2-connected degree-1 map $f:X\to{Z}$, 
with $Z$ strongly minimal,
and if the orientation character $w:\pi\to\mathbb{Z}^\times$ does not split
then the homotopy type of  $X$ is determined by the homotopy type of $Z$ 
and the equivariant intersection pairing.
Every $PD_4$-complex $X$  with fundamental group $\pi$
has such a ``strongly minimal model" $Z$ if and only if $c.d.\pi\leq2$.
(See Theorem \ref{om=sm}  below.)
This class of groups is both tractable and of direct interest 
to low-dimensional geometric topology, as it includes all surface groups, 
knot groups and the groups of many other bounded 3-manifolds. 
We expect that if $c.d.\pi\leq2$ the homotopy type of $Z$ is
determined by $\pi$,  $w$ and the Wu class $v_2(Z)$,
and that if $v_2(X)$ is induced from $\pi$ then the minimal model is unique.
(In the latter case,  the homotopy type of $X$ is determined by $\pi$, 
$w$, $v_2(X)$ and the equivariant intersection pairing.)
However, this is only known for $\pi$ a free group, a surface group,
a semidirect product $F(r)\rtimes\mathbb{Z}$ or a solvable Baumslag-Solitar group $\mathbb{Z}*_m$.

We shall now outline the paper in more detail.
The first two sections are algebraic.
In particular, Theorem \ref{wh-herm} (in \S2) establishes a connection between 
hermitean pairings and the Whitehead quadratic functor $\Gamma_W$.
Sections 3--8 consider the homotopy classification of $PD_4$-complexes,
and introduce several notions of minimality.
The first main result is Theorem \ref{thm06} in \S7, 
where it is shown that  two $PD_4$-complexes 
with the same strongly minimal model and $\pm$isometric intersection pairings
are homotopy equivalent,
provided $w:\pi\to\mathbb{Z}^\times$ does not split.
Sections 9 and 10 determine the strongly minimal $PD_4$-complexes
with $\pi_2=0$ and for which $\pi$ has finitely many ends.
Strongly minimal $PD_4$-complexes with $\pi$ a semidirect product
$\nu\rtimes\mathbb{Z}$ (with $\nu$ finitely presentable) 
are shown to be mapping tori in \S11.
When $\nu$ is a free group the homotopy type of such a mapping torus
is determined by $\pi$ and the Stiefel-Whitney classes, by Theorem \ref{freebyZ}.
The next five sections lead to the second main result,
Theorem \ref{min} (in \S16), 
which extends the result  of Theorem \ref{freebyZ} to the case when
$\pi$ has one end and  $c.d.\pi=2$ {\it provided\/} that the image of 
$\Pi\odot_\pi\Pi$ in $\mathbb{Z}^w\otimes_{\mathbb{Z}[\pi]}\Gamma_W(\Pi)$ is 2-torsion free, 
where $\Pi=\pi_2(X)\cong\overline{H^2(\pi;\mathbb{Z}[\pi])}$.
This theorem is modelled on the much simpler case analyzed in \S14,
in which $\pi$ is a $PD_2$-group.
Apart from the notion of minimality, the main technical points are
the connection between hermitean pairings and $\Gamma_W$,
the fact that a certain ``cup product" defines an isomorphism, 
and the 2-torsion condition.
In  \cite{Hi09}, we showed that the cup-product condition held for
surface groups, torus knot groups and solvable Baumslag-Solitar groups.
Here we show that it holds for all finitely presentable groups $\pi$ 
with one end and $c.d.\pi=2$ (Theorem \ref{cupthm}).
The 2-torsion condition is only known for $\pi$ a  $PD_2$-group 
or $\pi$ a solvable Baumslag-Solitar group (Theorem \ref{BStf}), 
and does not hold for all the cases covered by Theorem \ref{freebyZ}.
The final section considers the classification up to TOP $s$-cobordism 
or homeomorphism of closed 4-manifolds with groups as in Theorem \ref{min}.
In particular, it is shown that a remarkable 2-knot discovered by Fox 
is determined up to TOP isotopy and reflection by its knot group.

The theme of Hambleton, Kreck and Teichner \cite{HKT} is close to ours, 
although their methods are very different.
They use Kreck's modified surgery theory to classify up to $s$-cobordism 
closed orientable 4-manifolds with fundamental groups of geometric dimension 2 
(subject to some $K$- and $L$-theoretic hypotheses),
and they show also that every automorphism of the algebraic 2-type 
is realized by an $s$-cobordism, in many cases.
(They do not require that $\pi$ have one end,
which is a restriction imposed by our arguments.
However, when $\pi$ is a free group there is a simpler, 
more homological approach, which also uses the ideas of \S2 below
\cite{Hi04a}.)

This paper is a synthesis and extension of three papers 
\cite{Hi04b,Hi06,Hi09} which explored the role of minimality
in the classification of $PD_4$-complexes,
in particular, those with fundamental group $\pi$ 
such that $c.d.\pi=2$ and $\pi$ has one end. 
(Some aspects were considered much earlier \cite{Hi91,Hi94}.)
Apart from the benefits of revision, 
the main novelties are in showing that strongly minimal finite $PD_4$-complexes have minimal Euler characteristic (Corollary \ref{sm-chim-om}),
strong minimality is equivalent to order minimality if and only if $c.d.\pi\leq2$ (Theorem \ref{om=sm}),
verification that cup product defines an isomorphism for all 2-dimensional duality groups (Theorem \ref{cupthm}),
clarification of the role of the refined $v_2$-type,
and relaxation of some of the hypotheses.

\section{modules and group rings}

Let $\pi$ be a finitely presentable group and $w:\pi\to\mathbb{Z}^\times=\{\pm1\}$ be a homomorphism.
(This shall represent the orientation character for a $PD_n$-complex
with fundamental group $\pi$.
We shall at times view it as a class in $H^1(\pi;\mathbb{F}_2)$.)
Define an involution on $\mathbb{Z}[\pi]$ by $\bar g=w(g)g^{-1}$, for all $g\in\pi$.
Let $\mathbb{Z}$ and $\mathbb{Z}^w$ be the augmentation and $w$-twisted 
augmentation rings, and $\varepsilon:\mathbb{Z}[\pi]\to\mathbb{Z}$ and 
$\varepsilon_w:\mathbb{Z}[\pi]\to\mathbb{Z}^w$ be the augmentation 
and the $w$-twisted augmentation, defined by $\varepsilon(g)=1$
and $\varepsilon_w(g)=w(g)$, for all $g\in\pi$, respectively.
Let $I_w=\mathrm{Ker}(\varepsilon_w)$.

All modules considered here shall be left modules, unless otherwise noted.
However, if $L$ is a left ${\mathbb{Z}[\pi]}$-module 
the dual $Hom_{\mathbb{Z}[\pi]}(L,{\mathbb{Z}[\pi]})$
and the higher extension groups 
$Ext_{\mathbb{Z}[\pi]}^i(L,{\mathbb{Z}[\pi]})$
are naturally right modules.
If $R$ is a right ${\mathbb{Z}[\pi]}$-module let $\overline{R}$ be
the corresponding left ${\mathbb{Z}[\pi]}$-module with the conjugate structure 
given by $g.r=r.\bar g$, for all $g\in{\mathbb{Z}[\pi]}$ and $r\in R$.
Let $L^\dagger=\overline{Hom_{\mathbb{Z}[\pi]}(L,{\mathbb{Z}[\pi]})}$ 
and $E^iL=\overline{Ext_{\mathbb{Z}[\pi]}^i(L,{\mathbb{Z}[\pi]})}$, 
for $i\geq 0$ be the conjugate dual left modules. 
If $L$ is free, stably free or projective then so is $E^0L=L^\dagger$.
We shall consider $\mathbb{Z}$ and $\mathbb{Z}^w$ to be bimodules, 
with the same left and right $\pi$-structures.
(Note that $\overline{\mathbb{Z}}=\mathbb{Z}^w$.)

The modules $E^q\mathbb{Z}=\overline{H^q(\pi;\mathbb{Z}[\pi])}$ with $q\leq3$ 
shall recur throughout this paper. 
In particular, $E^0\mathbb{Z}\cong\mathbb{Z}^w$ if $\pi$ is finite and is 0
otherwise, while $E^1\mathbb{Z}$ reflects the number of ends of $\pi$.
It is 0 if $\pi$ is finite or has one end, 
infinite cyclic if $\pi$ has two ends (i.e., is virtually infinite cyclic)
and is free abelian of infinite rank otherwise.

\begin{lemma}
Let $M$ be a $\mathbb{Z}[\pi]$-module with a finite resolution
of length $n$ and such that $E^iM=0$ for $i<n$. Then 
$Aut(M)\cong{Aut(E^nM)}$.
\end{lemma}

\begin{proof}
Since $E^iM=0$ for $i<n$ the dual of a resolution of length $n$
for $M$ is a finite resolution for $E^nM$. 
Taking duals again recovers the original resolution, 
and so $E^nE^nM\cong{M}$.
If $f\in{Aut(M)}$ it extends to an endomorphism of the resolution
inducing an automorphism $E^nf$ of $E^nM$. 
Taking duals again gives $E^nE^nf=f$.
Thus $f\mapsto{E^nf}$ determines an isomorphism 
$Aut(M)\cong{Aut(E^nM)}$.
\end{proof}

A group $\pi$ is an {\it $n$-dimensional duality group\/} over $\mathbb{Z}$
if the augmentation $\mathbb{Z}[\pi]$ module $\mathbb{Z}$
has a finite projection resolution of length $n$, 
$H^i(\pi;\mathbb{Z}[\pi])=0$ for $i<n$ and the {\it dualizing module\/}
$\mathcal{D}=H^n(\pi;\mathbb{Z}[\pi])$ is torsion free as an abelian group.
(See Theorem VIII.10.1 of \cite{Br}.)
We then have $Aut(E^n\mathbb{Z})=\mathbb{Z}^\times$, by Lemma 1.
Finitely generated free groups are duality groups of dimension 1. 
If $\pi$ is finitely presentable and $c.d.\pi=2$ then 
$H^2(\pi;\mathbb{Z}[\pi])\not=0$,
and it is torsion free, by Proposition 13.7.1 of \cite{Ge}.
Hence $\pi$ is a 2-dimensional duality group if and only if it has one end.

In general,
$H^2(\pi;\mathbb{Z}[\pi])$ is $0$, $\mathbb{Z}$ or not finitely generated 
(\cite{Fa74} -- see Proposition 13.7.12 of \cite{Ge}).
In the latter case,
$H^2(\pi;\mathbb{Z}[\pi])$ must have infinite rank, 
by the main result of \cite{Bo}.
It remains open whether $H^2(\pi;\mathbb{Z}[\pi])$ must be free 
as an abelian group.

Let $F(n)$ be the free group with basis $\{x_1,\dots,x_n\}$.
The augmentation ideal of $\mathbb{Z}[F(n)]$
is freely generated by $\{x_1-1,\dots,x_n-1\}$ 
as a left $\mathbb{Z}[F(n)]$-module
and so we may write
\[
w-1=\Sigma_{1\leq{i}\leq{n}}\frac{\partial{w}}{\partial{x_i}}(x_i-1),
\]
for $w\in{F(n)}$.
Since $vw-1=v-1+v(w-1)$, for all $v,w\in{F(\mu)}$,
the Leibniz conditions 
\[
\frac{\partial{vw}}{\partial{x_i}}=
\frac{\partial{v}}{\partial{x_i}}+v\frac{\partial{w}}{\partial{x_i}}
\]
hold for all $v,w\in{F(\mu)}$ and $1\leq{i}\leq{n}$.
In particular, 
$\frac{\partial{1}}{\partial{x_i}}=0$ and 
$\frac{\partial{w^{-1}}}{\partial{x_i}}=
-w^{-1}\frac{\partial{w}}{\partial{x_i}}$,
for $1\leq{i}\leq{n}$.
We may extend these functions linearly to 
``derivations" of $\mathbb{Z}[F(n)]$.

Now let $\pi$ be a group with a finite presentation 
$\mathcal{P}=\langle{x_1,\dots,x_g}|{w_1,\dots,w_r}\rangle^\varphi,$
where $\varphi:F(g)\to\pi$ is an epimorphism with kernel 
the normal closure of $\{w_1,\dots,w_r\}$.
Let $def(\mathcal{P})=g-r$ be the deficiency  and $C(\mathcal{P})$ be 
the 2-complex corresponding to this presentation.
Then $\chi(C(\mathcal{P}))=1-def(\mathcal{P})$.
A choice of lifts of the $q$-cells of $C(\mathcal{P})$ to the universal cover 
$\widetilde{C(\mathcal{P})}$ determines
a basis for $C_q(\widetilde{C(\mathcal{P})})$ as a free left $\mathbb{Z}[\pi]$-module.
We view these as modules of column vectors.
The differentials are given by 
$\partial_1(c_1^{(i)})=(\varphi(x_i)-1)c_0$ and 
$\partial_2(c_2^{(j)})=\Sigma_{1\leq{i}\leq{g}}\varphi(\frac{\partial{w_j}}{\partial{x_i}})
c_1^{(i)}$.
(We extend $\varphi$ linearly to the group rings.)
The module of 0-cycles $Z_0(\widetilde{C(\mathcal{P})})$ is isomorphic to
$I(\pi)$, and so $I(\pi)$ has a $g\times{r}$ presentation matrix 
with $(i,j)$th entry $\varphi(\frac{\partial{w_j}}{\partial{x_i}})$.
(We shall refer to  $C_*(\widetilde{C(\mathcal{P})})$  as the {\it Fox-Lyndon resolution} of $\mathbb{Z}$ 
associated to $\mathcal{P}$.)

\begin{lemma}
\label{freeE^1}
Let $\pi=G*F(s)$, where $G=*_{i=1}^rG_i$ is the free product of $r\geq1$
one-ended groups $G_i$ and $s\geq0$. 
Then $E^1\mathbb{Z}\cong\mathbb{Z}[\pi]^{r+s-1}$.
\end{lemma}

\begin{proof}
If $s=0$ the result follows from the Mayer-Vietoris sequence for the free product, with coefficients $\mathbb{Z}[\pi]$.

In general, let $C_*(G)$ be a resolution of the augmentation module by free $\mathbb{Z}[G]$-modules
with $C_0(G)=\mathbb{Z}[G]$.
Then there is a corresponding resolution $C_*(\pi)$ with 
$C_q(\pi)\cong\mathbb{Z}[\pi]\otimes_{\mathbb{Z}[G]}C_q(G)$
if $q\not=1$ and $C_1(\pi)\cong\mathbb{Z}[\pi]\otimes_{\mathbb{Z}[G]}C_q(G)\oplus\mathbb{Z}[\pi]^s$.
Hence  there is a short exact sequence of chain complexes
\[
0\to\mathbb{Z}[\pi]\otimes_{\mathbb{Z}[G]}C_*(G)\to{C_*(\pi)}\to\mathbb{Z}[\pi]^s\to0,
\]
where the third term is concentrated in degree 1.
The exact sequence of cohomology  with coefficients $\mathbb{Z}[\pi]$
and the first step together 
give a short exact sequence
\[
0\to\mathbb{Z}[\pi]^s\to{E^1\mathbb{Z}=H^1(\pi;\mathbb{Z}[\pi])}\to
{H^1(G;\mathbb{Z}[G])}\otimes_{\mathbb{Z}[G]}\mathbb{Z}[\pi]
\cong\mathbb{Z}[\pi]^{r-1}\to0,
\]
from which the lemma follows easily.
\end{proof}

The hypothesis of this lemma holds if $\pi$ is torsion free but not free.
On the other hand, if $\pi$ is a nontrivial free group
then $E^1\mathbb{Z}$ has projective dimension 1 as a $\mathbb{Z}[\pi]$-module, 
and so the conclusion fails.

\section{the whitehead functor and hermitean pairings}

Let $A$ and $B$ be abelian groups.
A function $f:A\to B$ is {\it quadratic} if $f(-a)=f(a)$ for all $a\in A$ 
and if $f(a+b)-f(a)-f(b)$ defines a bilinear function from $A\times A$ to $B$.
The {\it Whitehead quadratic functor} $\Gamma_W$ assigns to each abelian group 
$A$ an abelian group $\Gamma_W(A)$ and a quadratic function 
$\gamma_A:A\to\Gamma_W(A)$ which is universal for quadratic functions 
with domain $A$.
The natural epimorphism from $A$ onto
$A/2A=\mathbb{F}_2\otimes{A}$ is quadratic, 
and so induces a canonical epimorphism $q_A$ from $\Gamma_W(A)$ to $A/2A$.
The kernel of this epimorphism is the image of the
symmetric square $A\odot{A}$ under the homomorphism $s$ from
$A\odot{A}$ to $\Gamma_W(A)$ given by 
$s(a\odot{b})=\gamma_A(a+b)-\gamma_A(a)-\gamma_A(b)$.
Thus there is an exact sequence
\begin{equation*}
\begin{CD}
A\odot{A}@>s>>\Gamma_W(A)@>q_A>>A/2A\to0.
\end{CD}
\end{equation*}
Moreover, $2\gamma_A(a)=s(a\odot{a})$, for all $a\in{A}$. 
(Topologically, if $\eta:S^3\to{S^2}$ is the Hopf map and $x\in\pi_2(X)$
then $2x\circ\eta=[x,x]$, the Whitehead product in $\pi_3(X)$.)
This sequence is short exact if $A$ is torsion free.
(See \S1.2 of \cite{Ba'}).

If $A$ and $B$ are abelian groups the inclusions into $A\oplus B$
induce a canonical splitting
$\Gamma_W(A\oplus B)\cong\Gamma_W(A)\oplus\Gamma_W(B)\oplus(A\otimes B)$.
Since $\Gamma(\mathbb{Z})\cong\mathbb{Z}$ it follows by a finite induction that
if $A\cong\mathbb{Z}^r$ then $\Gamma_W(\mathbb{Z}^r)$ is 
finitely generated and free, and that $s$ is injective.
If $A$ is any free abelian group,
every finitely generated subgroup of such a group 
lies in a finitely generated direct summand,
and so $\Gamma_W(A)$ is again free, and $s$ is injective.

A $w$-hermitean pairing on a  finitely generated $\mathbb{Z}[\pi]$-module
$M$ is a function $b:M\times{M}\to\mathbb{Z}[\pi]$ which is linear in the first variable 
and such that $b(n,m)=\overline{b(m,n)}$, for all $m,n\in{M}$.
The {\it adjoint homomorphism\/} $\widetilde{b}:M\to{M}^\dagger$
is given by $\widetilde{b}(n)(m)=b(m,n)$, for all $m,n\in{M}$.
The pairing $b$ is {\it nonsingular} if $\widetilde{b}$ is an isomorphism.

Let $Her_w(M)$ be the group of $w$-hermitean pairings on $M$.
Let $ev_M(m)(n,n')=\overline{n(m)}n'(m)$ for all $m\in M$ and $n,n'\in M^\dagger$.
Then $ev_M(m)(n,n')$ is quadratic in $m$ and $w$-hermitean in $n$ and
$n'$ and $ev_M(gm)=w(g)ev_M(m)$ for all $g\in\pi$ and $m\in M$.
Hence $ev_M$ determines a homomorphism 
\[
B_M:\mathbb{Z}^w\otimes_{\mathbb{Z}[\pi]}\Gamma_W(M)\to Her_w(M^\dagger).
\]
Let  $M\odot_\pi{M}=\mathbb{Z}^w\otimes_{\mathbb{Z}[\pi]}(M\odot_\mathbb{Z} M)$,
where $M\odot{M}$ has the diagonal $\pi$-action, 
given by $g(m\odot{n})=gm\odot{gn}$, for all $g\in\pi$ and $m,n\in M$.
 
\begin{theorem}
\label{wh-herm}
Let $\pi$ be a group, $w:\pi\to \mathbb{Z}^\times$ a homomorphism 
and $M$ a finitely generated projective $\mathbb{Z}[\pi]$-module.
If $\mathrm{Ker}(w)$ has no element of order $2$ then $B_M$ is surjective,
while if there is no element $g\in\pi$ of order $2$ such that $w(g)=-1$
then $B_M$ is injective.
\end{theorem}

\begin{proof}
Since $M$ is a free abelian group there is a short exact sequence 
\[
0\to{M}\odot{M}\to\Gamma_W(M)\to{M/2M}\to0,
\]
and $\Gamma_W(M)$ is free as an abelian group.
This is a sequence of $\mathbb{Z}[\pi]$-modules and homomorphisms. 
Since $M$ is projective, $\mathbb{Z}^w\otimes_{\mathbb{Z}[\pi]}M$
is also free as an abelian group.
Hence the sequence
\[
0\to{M\odot_\pi{M}}\to\mathbb{Z}^w\otimes_{\mathbb{Z}[\pi]}\Gamma_W(M)\to
\mathbb{Z}^w\otimes_{\mathbb{Z}[\pi]}M/2M=\mathbb{F}_2\otimes_{\mathbb{Z}[\pi]}M\to0
\]
is also exact, since
$Tor^{\mathbb{Z}[\pi]}_1(\mathbb{Z}^w,M/2M)=
\mathrm{Ker}(2:\mathbb{Z}^w\otimes_{\mathbb{Z}[\pi]}M\to
\mathbb{Z}^w\otimes_{\mathbb{Z}[\pi]}M)=0$.
Let $\eta_M:M\to\mathbb{Z}^w\otimes_{\mathbb{Z}[\pi]}\Gamma_W(M)$ be the 
composite of $\gamma_M$ with the reduction from $\Gamma_W(M)$ to
$\mathbb{Z}^w\otimes_{\mathbb{Z}[\pi]}\Gamma_W(M)$.
Then the composite of $\eta_M$ with the projection to
$\mathbb{F}_2\otimes_{\mathbb{Z}[\pi]}M$ is the canonical epimorphism.
Let $[m\odot{n}]$ be the image of $m\odot{n}$ in $M\odot_\pi{M}$.

Suppose first that $M$ is a free $\mathbb{Z}[\pi]$-module,
with basis $e_1,\dots,e_r$, and let
$e_1^*,\dots,e_r^*$ be the dual basis for $M^\dagger$,
defined by $e_i^*(e_i)=1$ and $e_i^*(e_j)=0$ if $i\not=j$.
Since $[m\odot{gn}]=[g(g^{-1}m\odot{n})]=[\bar gm\circ{n}]$ in $M\odot_\pi{M}$,
the typical element of $M\odot_\pi{M}$ may be expressed in the form
$\mu=\Sigma_{i\leq j}(r_{ij}e_i)\odot{e_j}$.
For such an element 
\[
B_M(\mu)(e^*_k,e^*_l)=r_{kl},\quad\mathrm{for}~k<l,
\]
and
\[
B_M(\mu)(e^*_k,e^*_l)=r_{kk}+\bar r_{kk},\quad\mathrm{for}~k=l.
\] 
In particular, $B_M(\mu)$ is {\it even}:
if $\varepsilon_2:\mathbb{Z}[\pi]\to\mathbb{F}_2$ is the composite of the augmentation with
reduction {\it mod} $(2)$ then $\varepsilon_2(B_M(\mu)(n,n))=0$ for all $n\in
M^\dagger$.
If $m\in M$ has nontrivial image in $\mathbb{F}_2\otimes_{\mathbb{Z}[\pi]} M$
then $\varepsilon_2(e_i^*(m))\not=0$ for some $i\leq r$.
Hence $B_M(\eta_M(m))$ is not even, 
and it follows easily that $\mathrm{Ker}(B_M)\leq{M\odot_\pi{M}}$.
If $B_M(\mu)=0$, for some $\mu=\Sigma_{i\leq j}(r_{ij}e_i)\odot{e_j}$,
then $r_{kl}=0$, if $k<l$, and $r_{ii}+\bar r_{ii}=0$, for all $i$.

If $\pi$ has no orientation reversing element of order 2 
and $B_M(\mu)=0$, where $\mu=\Sigma_{i\leq j}(r_{ij}e_i)\odot{e_j}$,
then $r_{ii}=\Sigma_{g\in F(i)}a_{ig}(g-\bar g)$, 
where $F(i)$ is a finite subset of $\pi$, for $1\leq i\leq r$.
Since $((g-\bar g)e_i)\odot{e_i}=0$ it follows easily that $\mu=\Sigma(r_{ii}e_i)\odot{e_i}=0$.
Hence $B_M$ is injective.

To show that $B_M$ is surjective when
$\mathrm{Ker}(w)$ has no element of order $2$
it shall suffice to assume that $M$ has rank 1 or 2, 
since $h$ is determined by the values $h_{ij}=h(e_i^*,e_j^*)$.
Let $\varepsilon_w[m,m']$ be the image of $m\odot{m'}$ in
$\mathbb{Z}^w\otimes_{\mathbb{Z}[\pi]}\Gamma_W(M)$.
Then 
\[
B_M(\varepsilon_w[m,m'])(n,n')=\overline{n(m)}n'(m')+\overline{n(m')}n'(m),
\]
for all $m,m'\in M$ and $n,n'\in M^\dagger$.
Suppose first that $M$ has rank 1.
Since $h_{11}=\bar h_{11}$ and $\mathrm{Ker}(w)$ has no element of order $2$
we may write $h_{11}=2b+\delta+\Sigma_{g\in F}(g+\bar g)$,
where $b=\bar b$, $\delta=1$ or 0 and $F$ is a finite subset of $\pi$.
Let 
\[
\mu=\varepsilon_w[(b+\delta+\Sigma_{g\in F}g)e_1,e_1]+\delta\eta_M(e_1).
\]
Then $B_M(\mu)(e_1^*,e_1^*)=h_{11}$.
If $M$ has rank 2 and $h_{11}=h_{22}=0$
let $\mu=\varepsilon_w[h_{12}e_1,e_2]$.
Then $B_M(\mu)(e_i^*,e_j^*)=h_{ij}$.
In each case $B_M(\mu)=h$, 
since each side of the equation is a $w$-hermitian pairing on $M^\dagger$.

Now suppose that $M$ is projective,
and that $P$ is a finitely generated projective complement to $M$, 
so that $M\oplus P\cong\mathbb{Z}[\pi]^r$ for some $r\geq0$.
The inclusion of $M$ into the direct sum induces a split monomorphism from 
$\Gamma_W(M)$ to $\Gamma_W(\mathbb{Z}[\pi]^r)$ which is clearly compatible with 
$B_M$ and $B_{\mathbb{Z}[\pi]^r}$.
We may extend an hermitian pairing $h$ on $M^\dagger$ 
to a pairing $h_1$ on $M^\dagger\oplus P^\dagger$ by setting 
$h_1(n,p)=h_1(p',p)=0$ for all $n\in M^\dagger$ and 
$p,p'\in P^\dagger$.
Clearly $h_1|_{M\times{M}}=h$ and so this extension determines a split
monomorphism from $Her_w(M^\dagger)$ to $Her_w((\mathbb{Z}[\pi]^r)^\dagger)$.
If $h_1=B_{\mathbb{Z}[\pi]^r}(\theta)$ then $h=B_M(\theta_M)$,
where $\theta_M$ is the image of $\theta$ under the homomorphism
induced by the projection from $M\oplus P$ onto $M$.
Thus if $B_{\mathbb{Z}[\pi]^r}$ is a monomorphism or an epimorphism 
so is $B_M$. 
\end{proof}

In particular, if $\pi$ has no 2-torsion then $B_M$ is an isomorphism, for any
projective $\mathbb{Z}[\pi]$-module $M$.
The restriction on 2-torsion is necessary, as can be seen by considering
the group $G=Z/2Z=\langle g\mid g^2\rangle$ with $w$ trivial and 
$h$ the pairing on $M=\mathbb{Z}[G]$ determined by $h(m,n)=mg\bar n$.

Let $E$ be another left $\mathbb{Z}[\pi]$-module.
The summand $M\otimes{E}$ of $\Gamma_W(M\oplus{E})$
has the diagonal left $\mathbb{Z}[\pi]$-module structure.
Let $d:M\to M^{\dagger\dagger}$ and 
$t:\mathbb{Z}\otimes_{\mathbb{Z}[\pi]}(M\otimes{E})\to
{Hom}_{\mathbb{Z}[\pi]}(M,E)$ 
be given by $d(m)(\mu)=\overline{\mu(m)}$ and $t(\mu\otimes{e})(m)=\mu(m)e$,
for all $m\in M$, $\mu\in M^\dagger$ and $e\in{E}$.
If $M$ is finitely generated and projective these functions
are isomorphisms (of left $\mathbb{Z}[\pi]$-modules
and abelian groups, respectively).
Let $\widetilde{B_M}(\gamma)$ be the adjoint of $B_M(1\otimes\gamma)$,
for all $\gamma\in\Gamma_W(M)$.

\begin{lemma}
\label{alphatheta}
Let $M$ be a finitely generated projective $\mathbb{Z}[\pi]$-module
and $\theta:M\to E$ be a $\mathbb{Z}[\pi]$-module homomorphism.
Let $\alpha_\theta(m,e)=(m,e+\theta(m))$ for all $(m,e)\in\Pi=M\oplus{E}$,
and let $d:M\to M^{\dagger\dagger}$ and 
$t:\mathbb{Z}\otimes_{\mathbb{Z}[\pi]}(M\otimes{E})\to
{Hom}_{\mathbb{Z}[\pi]}(M,E)$ be the isomorphisms defined above.
Then $\alpha_\theta$ is an automorphism of $\Pi$ and 
\[
\Gamma_W(\alpha_\theta)(\gamma)-\gamma\equiv
(d\otimes1)^{-1}[(\widetilde{B_M}(\gamma)\otimes1)(t^{-1}(\theta))]
\quad{mod}\quad \Gamma_W(E),
\]
for all $\gamma\in\Gamma_W(M)$.
\end{lemma}

\begin{proof}
The homomorphism $\alpha_\theta$ is clearly an automorphism of $\Pi$ 
which restricts to the identity on the summands $E$ and $M$,
and 
\[
\Gamma_W(\alpha_\theta)(\gamma_\Pi(m))
=\gamma_\Pi(m)+\gamma_\Pi(\theta(m))+{m\otimes\theta(m)},
\]
for all $m\in{M}.$
(See (1.2.7) on page 16 of \cite{Ba'}.)

Let $\beta_m=B_M(1\otimes\gamma_M(m))$, for $m\in{M}$.
Now the adjoint homomorphism $\widetilde{\beta_m}$
is given by $\widetilde{\beta_m}(\mu)=\overline{\mu(m)}d(m)$.
Since $t$ is surjective we have $\theta=t(\Sigma\mu_i\otimes e_i)$,
for some $\mu_i\in M^\dagger$ and $e_i\in E$.
Then 
$(\widetilde{\beta_m}\otimes1)(t^{-1}(\theta))=$
\[\Sigma\widetilde{\beta_m}(\mu_i)\otimes{e_i}=
\Sigma{d(m)}\otimes\mu_i(m){e_i}=
d(m)\otimes\theta(m)=(d\otimes1)(m\otimes\theta(m)).
\]
Since $\Gamma_W(\alpha_\theta)(\gamma_\Pi(m))-\gamma_\Pi(m)\equiv
(d\otimes1)^{-1}[(\widetilde{\beta_m}\otimes1)(t^{-1}(\theta))]$ 
{\it mod} $\Gamma_W(E)$, for all $m\in{M}$, 
and each side is quadratic in $m$, 
we have 
\[
\Gamma_W(\alpha_\theta)(\gamma)-\gamma\equiv
(d\otimes1)^{-1}[(\widetilde{B_M}(\gamma)\otimes1)(t^{-1}(\theta))]\quad {\mathit mod}
\quad\Gamma_W(E),
\]
for all $\gamma\in\Gamma_W(M).$
\end{proof}

\section{postnikov stages}

Let $X$ be a based, connected cell complex with fundamental group $\pi$,
and let $p_X:\widetilde{X}\to{X}$ be its universal covering projection.
Let $E_0(X)$ be the group of 
based homotopy classes of based self-homotopy equivalences of $X$,
and $E_\pi(X)$ be the subgroup which induces the identity on $\pi$.
If we fix a basepoint for $\widetilde{X}$ over the basepoint of $X$ 
then there are well-defined Hurewicz homomorphisms
$hwz_q:\pi_q(X)=\pi_q(\widetilde{X})\to{H_q(\widetilde{X};\mathbb{Z})}$, 
for all $q\geq2$.

Let $f_{X,k}:X\to{P_k(X)}$ be the $k^{th}$ stage of the Postnikov tower for $X$.
We may construct $P_k(X)$ by adjoining cells of dimension at least $k+2$
to kill the higher homotopy groups of $X$.
The map $f_{X,k}$ is then given by the inclusion of $X$ into $P_k(X)$, 
and is a $(k+1)$-connected map.
In particular, $P_1(X)\simeq{K=K(\pi,1)}$ and $c_X=f_{X,1}$ is the 
classifying map for the fundamental group $\pi=\pi_1(X)$.

If $M$ is a left $\mathbb{Z}[\pi]$-module let $L_\pi(M,n)$ be the 
{\it generalized Eilenberg-Mac Lane space\/} 
over $K=K(\pi,1)$ realizing the given action of $\pi$ on $M$.
Thus the classifying map for $L=L_\pi(M,n)$ is a principal $K(M,n)$-fibration 
with a section $\sigma:K\to{L}$.
Let $[X;Y]_K$ be the set of homotopy classes over $K=K(\pi,1)$ 
of maps $f:X\to{Y}$ such that $c_X=c_Yf$.
(These may also be considered as
$\pi$-equivariant homotopy classes of $\pi$-equivariant maps
from $\widetilde{K}$ to $\widetilde{L}$.)
We may view $L_\pi(M,n)$ as the $ex$-$K$ loop space 
$\overline\Omega{{L_\pi}(M,n+1)}$, 
with section $\sigma$ and projection $c_L$.
Let $\mu:L\times_KL\to {L}$ be the (fibrewise) loop multiplication.
Then $\mu(id_L,\sigma{c_L})=\mu(\sigma{c_L},id_L)=id_L$ in $[L;L]_K$.
Let $\iota_{M,n}\in{H^n}(L;M)$ be the characteristic element.
The function $\theta:[X,L]_K\to{H^n}(X;M)$ given by $\theta(f)=f^*\iota_{M,n}$ 
is a isomorphism with respect to the addition on $[X,L]_K$ determined by $\mu$.
Thus $\theta(id_L)=\iota_{M,n}$,
$\theta(\sigma{c_X})=0$ and $\theta(\mu(f,f'))=\theta(f)+\theta(f')$.
(See \S{V.2} of \cite{Ba0}.)

Let $k_1(X)\in{H}^3(\pi;\pi_2(X))$ be the first $k$-invariant
and $f_X=f_{X,2}$ be the second stage of the Postnikov tower for $X$.
The {\it algebraic $2$-type} $[\pi,\pi_2(X),k_1(X)]$ and the
Postnikov 2-stage determine each other.
More precisely,  $P_2(X)\simeq{P_2(Y)}$ if and only if there are isomorphisms
$\alpha:\pi\cong\pi_1(Y)$ and $\beta:\pi_2(X)\cong\pi_2(Y)$ such that
$\beta$ is $\alpha$-semilinear and $\alpha^*k_1(Y)=\beta_\#k_1(X)$ in $H^3(\pi;\pi_2(Y))$.
Moreover, 
\[
k_1(X)=0~\Leftrightarrow~c_{P_2(X)}~\mathrm{has~a~section}~
\Leftrightarrow~P_2(X)\simeq{L}_\pi(\pi_2(X),2).
\]
Let $L=L_\pi(M,2)$.
Then $E_\pi(L)$ is the group of units of $[L,L]_K$ with respect to composition.
We shall use the following special case of a result of Tsukiyama \cite{Tsu}; 
we give only the part that we need below.

\begin{lemma}
\label{tsuki}
There is an exact sequence
$0\to{H^2}(\pi;M)\to{E_\pi(L)}\to{Aut(M)}\to0.$
\end{lemma}

\begin{proof}
Let $\theta:[K,L]_K\to{H^2}(\pi;M)$ be the isomorphism given by
$\theta(s)=s^*\iota_{M,2}$, 
and let $\theta^{-1}(\phi)=s_\phi$ for $\phi\in{H^2}(\pi;M)$.
Then $s_\phi$ is a homotopy class of sections of $c_L$, 
$s_0=\sigma$ and $s_{\phi+\psi}=\mu(s_\phi,s_\psi)$,
while $\phi=s_\phi^*\iota_{M,2}$.
(Recall that $\mu:L\times_KL\to {L}$ is the fibrewise loop multiplication.)

Let $h_\phi=\mu(s_\phi{c_L},id_L)$.
Then $c_Lh_\phi=c_L$ and so $h_\phi\in[L;L]_K$.
Clearly $h_0=\mu(\sigma{c_L},id_L)=id_L$
and $h_\phi^*\iota_{M,2}=\iota_{M,2}+c_L^*\phi\in{H^2}(L;M)$.
We also see that
\begin{align*}
h_{\phi+\psi}&=\mu(\mu(s_\phi,s_\psi){c_L},id_L)\\
&=\mu(\mu(s_\phi{c_L},s_\psi{c_L}),id_L)\\
&=\mu(s_\phi{c_L},\mu(s_\psi{c_L},id_L))
\end{align*}
(by homotopy associativity of $\mu$) and so
\[h_{\phi+\psi}=\mu(s_\phi{c_L},h_\psi)=
\mu(s_\phi{c_L}h_\psi,h_\psi)=h_\phi{h_\psi}.\]
Therefore $h_\phi$ is a homotopy equivalence for all 
$\phi\in{H^2}(\pi;M)$, 
and $\phi\mapsto{h_\phi}$ defines a homomorphism from
${H^2}(\pi;M)$ to ${E_\pi(L)}$.

The lift of $h_\phi$ to the universal cover $\widetilde{L}$ is 
(non-equivariantly) homotopic to the identity, 
since the lift of $c_L$ is (non-equivariantly) homotopic to a constant map.
Therefore $h_\phi$ acts as the identity on $M=\pi_2(L)$.
\end{proof}

The homomorphism $h:\phi\mapsto{h_\phi}$ is in fact an isomorphism 
onto the kernel of the action of ${E_\pi(L)}$ on $M$ \cite{Tsu},
and the extension splits: 
$E_\pi(L)$ is isomorphic to a semidirect product $H^2(\pi;M)\rtimes{Aut}(M)$.
(See Corollary 8.2.7 of \cite{Ba}.)
More  generally, if $P=P_2(X)$, $\Pi=\pi_2(X)$ and $H$ is the subgroup of
$Aut_\pi(\Pi)\rtimes{Aut(\pi)}$ which fixes $k_1(X)\in{H^3}(\pi;\Pi)$ then
\[
E_0(P)\cong{H^2(\pi;\Pi)}\rtimes{H}
\]
(see part II of \cite{Ru92}).
Thus if $P=L_\pi(\Pi)$ every automorphism of $\pi$ 
lifts to a self-homotopy equivalence of $L$,
and $E_0(L)\cong{E_\pi(L)}\rtimes{Aut(\pi)}$.

Let $X^{[k]}$ be the $k$-skeleton of $X$, for all $k\geq0$, and let $\Pi=\pi_2(X)$.
The image of $\pi_3(X^{[2]})$ in $\pi_3(X^{[3]})$ is isomorphic to  $\Gamma_W(\Pi)$, 
and the inclusion of the 3-skeleton induces a homomorphism 
$\iota_X:\Gamma_W(\Pi)\to\pi_3(X)$.
The composite of $\iota_X$ with the natural map from $\Pi\odot\Pi$ to $\Gamma_W(\Pi)$ 
is the Whitehead product $[-,-]$,
and there is a natural {\it Whitehead exact sequence\/} of abelian groups
\begin{equation*}
\begin{CD}
\pi_4(X)@> hwz_4>>H_4(\widetilde{X};\mathbb{Z})@>b_X>>\Gamma_W(\Pi)
@>\iota_X>>
\pi_3(X)@> hwz_3>>H_3(\widetilde{X};\mathbb{Z})\to0,
\end{CD}
\end{equation*}
where $b_X$ is the {\it secondary boundary homomorphism}  \cite{Wh50}.
(See (2.1.17) of \cite{Ba'}.)
This is an exact sequence of  left $\mathbb{Z}[\pi]$-modules, by naturality.
(Note also that the Whitehead sequence for $K(\Pi,2)$ gives 
$H_4(\Pi,2;\mathbb{Z})\cong\Gamma_W(\Pi)$. )
 
The homology  spectral sequence for $P_3(\widetilde{X})$
as a fibration over $K(\Pi,2)$ with fibre $K(\pi_3(X),3)$
gives an exact sequence
\begin{equation*}
\begin{CD}
0\to{H_4(P_3(\widetilde{X});\mathbb{Z})}\to{H_4(\Pi,2;\mathbb{Z})}@>d^2_{4,0}>>
{H_3(\pi_3(X),3;\mathbb{Z})}\to
H_3(P_3(\widetilde{X});\mathbb{Z})\to0,
\end{CD}
\end{equation*}
in which $d^2_{4,0}$ is  the homology transgression.
Composing $d^2_{4,0}$ with the inverse of the Hurewicz isomorphism 
$hwz_3$ for $K(\pi_3(X),3)$ gives the image of 
the second $k$-invariant $k_2(\widetilde{X})\in{H^4(\Pi,2;\pi_3(X))}$ 
in $Hom(H_4(\Pi,2;\mathbb{Z}),\pi_3(X) )$ under the evaluation homomorphism,
by the interpretation of $k$-invariants 
in terms of transgression \cite{Na}.
In fact $d^2_{4,0}=hwz_3\iota_X$, by Theorem 2.5.10 of \cite{Ba'}.

\section{$PD_4$-complexes and intersection pairings}

Let $X$ be a based finitely dominated cell complex, 
with the natural left $\mathbb{Z}[\pi]$-module structure.
The equivariant cellular chain complex $C_*=C_*(X;\mathbb{Z}[\pi])$
of $\widetilde{X}$ is a complex of left $\mathbb{Z}[\pi]$-modules, 
and is $\mathbb{Z}[\pi]$-chain homotopy equivalent to 
a finitely generated complex of projective modules.
Let $C^q=Hom_{\mathbb{Z}[\pi]}(C_q,\mathbb{Z}[\pi])$, for all $q\geq0$,
and let $\Pi={H_2(\widetilde{X};\mathbb{Z})}=H_2(C_*)$.
Recall that the choice of a basepoint for $\widetilde{X}$
determines an isomorphism $\pi_2(X)\cong\Pi$.

Let $ev:\overline{H^2(X;\mathbb{Z}[\pi])}\to\Pi^\dagger$ 
be the evaluation homomorphism,
given by 
\[
ev([c])([z])={[c]\cap[z]}=c(z)\quad\forall~c\in{C^2}~\mathrm{and}~z\in{C_2}.
\]
This homomorphism sits in the {\it evaluation} exact sequence 
\begin{equation*}
\begin{CD}
0\to{E^2\mathbb{Z}}\to\overline{H^2(X;\mathbb{Z}[\pi])}@> 
ev >>\Pi^\dagger\to{E^3\mathbb{Z}}\to
\overline{H^3(X;\mathbb{Z}[\pi])}.
\end{CD}
\end{equation*}
(See Lemma 3.3 of \cite{Hi}.)
If $X$ is a $PD_4$-complex then $H^3(X;\mathbb{Z}[\pi])=H_1(\widetilde{X};\mathbb{Z})=0$,
and the evaluation sequence is a 4-term exact sequence.

We assume henceforth that $X$ is a $PD_4$-complex,
with orientation character $w=w_1(X)$.
Let $X^+$ be the orientable covering space associated to
$\pi^+=\mathrm{Ker}(w)$.
The complex $X$ is finitely dominated and is homotopy equivalent to 
$X_o\cup_\phi{e^4}$, 
where $X_o$ is a complex of dimension at most 3 and $\phi\in\pi_3(X_o)$
\cite{Wa}.
In particular, $\pi$ is finitely presentable.
In \cite{Hi04a} and \cite{Hi04b} cellular decompositions were used
to study the homotopy types of $PD_4$-complexes.
Here we shall rely more consistently on the dual Postnikov approach.

\begin{lemma}
\label{post3}
If $\pi$ is infinite the homotopy type of $X$ is determined by $P_3(X)$.
\end{lemma}

\begin{proof}
If $X$ and $Y$ are two such $PD_4$-complexes and $h:P_3(X)\to{P_3(Y)}$ 
is a homotopy equivalence then $hf_{X,3}$ is homotopic to a map $g:X\to{Y}$.
Since $\pi$ is infinite 
$H_4(\widetilde{X};\mathbb{Z})=H_4(\widetilde{Y};\mathbb{Z})=0$,
by Poincar\'e duality.
Since $\pi_i(g)$ is is an isomorphism for $i\leq3$ 
any lift ${\tilde{g}:\widetilde{X}\to\widetilde{Y}}$ 
is a homotopy equivalence, by the Hurewicz and Whitehead theorems, 
and so $g$ is a homotopy equivalence.
\end{proof}

In particular, if $\pi$ is torsion free but not free
then $H_3(X;\mathbb{Z}[\pi])\cong{E^1}\mathbb{Z}$ 
is a free $\mathbb{Z}[\pi]$-module, by Lemma \ref{freeE^1}, 
and so $\pi_3(X)\cong\Gamma_W(\Pi)\oplus{E^1\mathbb{Z}}$.
Hence the homotopy type of $X$ is determined by
$\pi,w,\Pi$ and the first two $k$-invariants.
Note that the first $k$-invariant $k_1(X)$ may be defined 
as the primary obstruction to constructing a left inverse 
to the classifying map $c_X$.
The Homotopy Addition Theorem (see Theorem IV.6.1 of \cite{Wh} or
Proposition 7.5.3 of \cite{Spn}) may be used to 
identify $k_1(X)$ with the class in $Ext^3_{\mathbb{Z}[\pi])}(\mathbb{Z},\Pi)$
of the iterated extension
\[
0\to\pi_2(X)\to{C_2/\partial{C_3}}\to{C_1}\to{C_0}\to\mathbb{Z}\to0.
\]

Let $H=\overline{H^2(X;\mathbb{Z}[\pi])}$.
A choice of generator $[X]$ for $H_4(X;\mathbb{Z}^w)\cong\mathbb{Z}$
determines a Poincar\'e duality isomorphism $D:{H}\to\Pi$
by $D(u)=u\cap[X]$, for all $u\in{H}$.
Moreover $H^3(X;\mathbb{Z}[\pi])=0$.
The {\it cohomology intersection pairing} $\lambda:H\times H\to\mathbb{Z}[\pi]$ 
is defined by $\lambda(u,v)=ev(v)(D(u))$, for all $u,v\in{H}$. 
This pairing is $w$-hermitian: $\lambda(gu,hv)=g\lambda(u,v)\bar{h}$ and
$\lambda(v,u)=\overline{\lambda(u,v)}$ for all $u,v\in{H}$ and $g,h\in\pi$.
If $X$ is a closed 4-manifold this pairing is equivalent under Poincar\'e 
duality to the equivariant intersection pairing on $\Pi$.
(See page 82 of \cite{Ra}.)
Replacing $[X]$ by $-[X]$ changes the sign of the pairing.
Since $\lambda(u,e)=0$ for all $u\in{H}$ and $e\in E=E^2\mathbb{Z}$ 
the pairing $\lambda$ induces a pairing 
\[
\lambda_X:H/E\times{H/E}\to\mathbb{Z}[\pi].
\]
The adjoint $\widetilde{\lambda_X}$ is a monomorphism,
since $\mathrm{Ker}(ev)=E$.
The $PD_4$-complex $X$ is {\it strongly minimal\/} if $\lambda_X=0$.

The next lemma relates nonsingularity of $\lambda_X$,
projectivity of $\Pi$ and $H/E$ and conditions on $E^2\mathbb{Z}$ 
and $E^3\mathbb{Z}$.

\begin{lemma}
\label{lambdaE}
Let $X$ be a $PD_4$-complex with fundamental group $\pi$, and let
$E=E^2\mathbb{Z}$, $H=\overline{H^2(X;\mathbb{Z}[\pi])}$ and $\Pi=\pi_2(X)$.
Then

\begin{enumerate}
\item $\lambda_X=0$ if and only if $H=E$, 
and then $E^3\mathbb{Z}\cong{E^\dagger}$;

\item if $\lambda_X$ is nonsingular and
$H/E$ is a projective $\mathbb{Z}[\pi]$-module then
$E^3\mathbb{Z}\cong{E^\dagger}$;

\item if $\lambda_X$ is nonsingular and $E^\dagger=0$ then $E^3\mathbb{Z}=0$;

\item if $E^3\mathbb{Z}=0$ then $\lambda_X$ is nonsingular;

\item if $E^3\mathbb{Z}=0$ and $\Pi$ is a projective $\mathbb{Z}[\pi]$-module 
then $E=0$;

\item if $\pi=G*F(s)$, where $G=*_{i=1}^rG_i$ is the free product of $r\geq1$
one-ended groups and $\Pi$ is a projective $\mathbb{Z}[\pi]$-module
then $c.d.\pi\leq4$, with equality if $\pi$ has one end.
\end{enumerate}
\end{lemma}

\begin{proof}
Let $p:\Pi\to\Pi/D(E)$ and ${q:H\to{H/E}}$ be the canonical epimorphisms.
Poincar\'e duality induces an isomorphism $\gamma:H/E\cong\Pi/D(E)$.
It is straightforward to verify that
$p^\dagger(\gamma^\dagger)^{-1}\widetilde{\lambda_X}q=ev$,
and (1) is clear.

If $\lambda_X$ is nonsingular then $\widetilde{\lambda_X}$ is an isomorphism,
and so $\mathrm{Coker}(p^\dagger)=\mathrm{Coker}(ev)$. 
If moreover $\Pi/D(E)\cong{H/E}$ is projective then 
$\Pi\cong(\Pi/D(E))\oplus{D}(E)$.
Hence $\Pi^\dagger\cong(\Pi/D(E))^\dagger\oplus{E}^\dagger$,
and so $E^\dagger\cong\mathrm{Coker}(p^\dagger)={E}^3\mathbb{Z}$.

If $\lambda_X$ is nonsingular 
and $E^\dagger=0$ then $\widetilde{\lambda_X}$
and $p^\dagger$ are isomorphisms, and so 
$ev=p^\dagger(\gamma^\dagger)^{-1}\widetilde{\lambda_X}q$ is an epimorphism.
Hence $E^3\mathbb{Z}=0$.

If $E^3\mathbb{Z}=0$ then $H/E=\Pi^\dagger$ and $ev=q$.
Since $q$ is an epimorphism it follows that 
$p^\dagger(\gamma^\dagger)^{-1}\widetilde{\lambda_X}=id_{\Pi^\dagger}$,
and so $p^\dagger$ is an epimorphism.
Since $p^\dagger$ is also a monomorphism it is an isomorphism.
Therefore $\widetilde{\lambda_X}=\gamma^\dagger(p^\dagger)^{-1}$
is also an isomorphism.

If $\Pi$ is projective then so is $\Pi^\dagger$.
If, moreover, $E^3\mathbb{Z}=0$ then $H\cong{E}\oplus\Pi^\dagger$.
Hence  $E$ is projective, since it is a direct summand of $H\cong\Pi$,
and so $E\cong{E^{\dagger\dagger}}=0$.

If $\pi$ is a free product of $r\geq1$ one-ended groups and $s$ copies of $\mathbb{Z}$ then $E^1\mathbb{Z}\cong\mathbb{Z}[\pi]^{r+s-1}$,
by Lemma \ref{freeE^1}.
If, moreover, $\Pi$ is projective then so are $C_3'=C_3\oplus\Pi$ and $C_4'=C_4\oplus{E^1\mathbb{Z}}$.
We may easily extend the differentials of $C_*$ to obtain a projective
resolution $C_*'$ of length 4 for $\mathbb{Z}$.
Hence $c.d.\pi\leq4$.
If $\pi$ has one end and $\Pi$ is projective then 
$H^4(\pi;\mathbb{Z}[\pi])=E^4\mathbb{Z}\cong{H^4(X;\mathbb{Z}[\pi])}\cong\mathbb{Z}$,
by the Universal Coefficient spectral sequence and Poincar\'e duality,
and so $c.d.\pi=4$.
\end{proof}

In particular, if $E^2\mathbb{Z}=0$ then 
$\lambda_X$ is nonsingular if and only if $E^3\mathbb{Z}=0$ also.

Can  the hypotheses in this lemma be simplified?
If $G=\mathbb{Z}^2*\mathbb{Z}^2$ and $\pi=G\times{G}$ 
then $E^1\mathbb{Z}=E^3\mathbb{Z}=0$,
but $E^2\mathbb{Z}\cong\mathbb{Z}[\pi]$, and so $E^\dagger\not=0$.
Projectivity of $\Pi^\dagger$ and $E^2\mathbb{Z}=0$ together do not imply that
$E^3\mathbb{Z}=0$.
For if $\pi$ is a $PD_3$-group and $w=w_1(\pi)$ then 
$E^s\mathbb{Z}=0$ for $s<3$ and $\Pi$ is stably isomorphic 
to the augmentation ideal of $\mathbb{Z}[\pi]$, by Theorem 3.13 of \cite{Hi}, 
and so $\Pi^\dagger$ is stably free.
However $E^3\mathbb{Z}\cong\mathbb{Z}\not=0$.

We shall say that a based map $f:X\to{Y}$ between $PD_4$-complexes 
is a {\it degree-1 map\/} and write $f_*[X]=\pm[Y]$ if 
$f^*w_1(Y)=w_1(X)=w$ and the lift of $f$ to a based map of universal covers
induces an isomorphism $H_4(X;\mathbb{Z}^w)\cong{H_4(Y;\mathbb{Z}^w)}$.
(Note that if we do not work with based maps the homomorphisms
induced by different lifts may differ by sign -- 
see \cite{Ta} for an investigation of the subtleties involved.)
The homomorphism $\pi_1(f)$ is then surjective,
and Poincar\'e duality in $X$ and $Y$ determine {\it umkehr} 
homomorphisms $f_!:H_*(Y;\mathbb{Z}[\pi_1(Y)])\to{H_*(X;f^*\mathbb{Z}[\pi_1(Y)])}$, 
which split the homomorphisms induced by $f$.
The umkehr homomorphisms are well-defined up to sign.
(See \S10.3 of \cite{Ra}.)
If $f:X\to Z$ is a 2-connected degree-1 map then cap product with $[X]$ 
induces an isomorphism from the ``surgery cokernel"
$K^2(f)=\overline{\mathrm{Cok}(H^2(f;\mathbb{Z}[\pi]))}$ to $K_2(f)$, and
the induced pairing $\lambda_f$ on $K^2(f)\times{K^2(f)}$ is nonsingular,
by Theorem 5.2 of \cite{Wa}.

We shall not usually specify a fundamental class $[X]$, 
and so we shall allow orientation-reversing homotopy equivalences 
of oriented $PD_4$-complexes, 
and isomorphisms of modules with pairings
which are isometries after a change of sign.
In particular, if $Y$ is a second $PD_4$-complex we write $\lambda_X\cong\lambda_Y$
if there is an isomorphism $\theta:\pi_1(X)\cong\pi_1(Y)$ such that
$w_1(X)=w_1(Y)\circ\theta$ and a $\mathbb{Z}[\pi]$-module isomorphism 
$\Theta:\pi_2(X)\cong\theta^*\pi_2(Y)$
inducing an isometry of cohomology intersection pairings
(after changing the sign of $[Y]$, if necessary).

In \cite{BB} it is shown that a $PD_4$-complex $X$ is determined by 
its algebraic 2-type (i.e., by $P_2(X)$) 
together with $w_1(X)$ and $f_{X*}[X]$. 
(The main step involves showing that if $h:P_2(X)\to{P_2(Y)}$ 
is a homotopy equivalence such that $h^*w_1(Y)=w_1(X)$ and
$h_*f_{X*}[X]=f_{Y*}[Y]$ (up to sign) then $h=P_2(g)$ for some map 
$g:X\to{Y}$ such that $H_4(g;\mathbb{Z}^w)$ is an isomorphism.)
Our goal is to show that under suitable conditions
$X$ is determined by the more accessible invariants encapsulated in
the sextuple $[\pi,w,v_2(X),\Pi,k_1(X),\lambda_X]$.
(This is the {\it quadratic $2$-type} of $X$,
as in \cite{HK88},
enhanced by the Wu classes; equivalently, by the Stiefel-Whitney classes.) 
If $\lambda_X\not=0$ then $\lambda_X$ determines $w$,
since $\lambda_X(gu,gv)=w(g)g\lambda_X(u,v)g^{-1}$ for all $u,v$ and $g$.

It shall be useful to distinguish three ``$v_2$-types" of $PD_4$-complexes:
{\makeatletter
\renewcommand{\theenumi}{\Roman{enumi}}
\renewcommand{\p@enumi}{\theenumi--}
\makeatother
\begin{enumerate}
\item $v_2(\widetilde{X})\not=0$ (i.e., 
$v_2(X)$ is not in the image of $H^2(\pi;\mathbb{F}_2)$ under $c_X^*$);
\item $v_2(X)=0$;
\item $v_2(X)\not=0$ but $v_2(\widetilde{X})=0$ (i.e.,
$v_2(X)$ is in $c_X^*(H^2(\pi;\mathbb{F}_2))\setminus\{0\}$).
\end{enumerate}}
(This trichotomy is due to Kreck, 
who formulated it in terms of Stiefel-Whitney
classes of the stable normal bundle of a closed 4-manifold.)
The {\it refined\/} $v_2$-type (II and III) is given by the orbit 
of $v_2$ in $H^2(\pi;\mathbb{F}_2)$ under the action of 
automorphisms of $\pi$ which fix the orientation character.

\section{minimal models}

A  {\it model} for a $PD_4$-complex $X$ is a 2-connected 
degree-1 map $f:X\to Z$ to a $PD_4$-complex $Z$.
(We shall also say that $Z$ is a model for $X$.)
The ``surgery kernel" $K_2(f)=\mathrm{Ker}(\pi_2(f))$ is a finitely generated 
projective $\mathbb{Z}[\pi]$-module, 
and is an orthogonal direct summand of $\pi_2(X)$ with respect to 
the intersection pairing, by Theorem 5.2 of \cite{Wa}.
If both complexes are finite then $K_2(f)$ is stably free.
The $PD_4$-complex $X$ is {\it order-minimal} if 
every such map is a homotopy equivalence, i.e., 
if $X$ is minimal with respect to the order determined by such maps.
It is {\it strongly minimal\/} if $\lambda_X=0$,
and is {\it $\chi$-minimal\/} if $\chi(X)\leq\chi(Y)$,
for $Y$ any $PD_4$-complex with $(\pi_1(Y),w_1(Y))\cong(\pi,w)$.
We then let $q(\pi,w)=\chi(X)$ be this minimal value.
(The definition of  ``strongly minimal" used here may be broader 
than the one used in \cite{Hi04b}, 
where we said that $Z$ was strongly minimal if $\pi_2(Z)^\dagger=0$.
The two definitions are equivalent if $(E^2\mathbb{Z})^\dagger=0$.)

Order minimality is the most natural property, 
and $\chi$-minimality perhaps the one most easily established.
Strongly minimal $PD_4$-complexes are clearly order-minimal.
We shall show that $\chi$-minimality interpolates between these notions, 
when the $L^2$ Euler characteristic formula applies.

\begin{theorem}
\label{L2beta2}
A $PD_4$-complex $X$ with fundamental group $\pi$ 
is strongly minimal if and only if
$\beta_2^{(2)}(X)=\beta_2^{(2)}(\pi)$.
\end{theorem}

\begin{proof}
The module $C^2(X;\mathbb{C}[\pi])$ may be identified with the group 
of cellular 2-cochains with compact support on $\widetilde{X}$, 
while the corresponding module $C^2_{(2)}(\widetilde{X})$ 
of $L^2$-cochains is the group of square-summable cellular 2-cochains  
on $\widetilde{X}$.
The compactly supported cochains are dense in the square-summable cochains.
For each $z\in\pi_2(X)$ the evaluation $ev_z:f\to{f(z)}$ is continuous 
as a linear map from $C^2_{(2)}(\widetilde{X})$ to $\mathbb{C}$.
(See the proof of Theorem 3.4 of \cite{Hi}.
If $X$ is strongly minimal then $ev_z(f)=0$ for all $f\in{C^2(X;\mathbb{C}[\pi])}$.
Hence  $ev_z=0$ for all $z\in\pi_2(M)$.
The $L^2$ analogue of the evaluation sequence 
(as in \S1.4 of \cite{Ec94}) then shows that $c_X$ induces 
an isomorphism on the unreduced $L^2$-cohomology modules, 
and so $\beta_2^{(2)}(X)=\beta_2^{(2)}(\pi)$.
The converse is part (3) of Theorem 3.4 of \cite{Hi}.
\end{proof}

The next two corollaries need a further hypothesis at present.
 
\begin{cor}
\label{sm-chim-om}
Suppose that either $X$ is finite or $\pi$ 
satisfies the Strong Bass Conjecture.
Then if $X$ is strongly minimal it is $\chi$-minimal, 
and if it is $\chi$-minimal it is order minimal.
\end{cor}

\begin{proof}
If $X$ is finite or $\pi$ satisfies the Strong Bass Conjecture
we may use the $L^2$-Euler characteristic formula
then $\chi(X)=\beta_2^{(2)}(X)-2\beta_1^{(2)}(X)$ \cite{Ec96}.
Since we may construct a $K(\pi,1)$ complex by adjoining cells 
of dimension $>2$ to $X$, we have 
$\beta_2^{(2)}(X)\geq\beta_2^{(2)}(\pi)$, in general.
Hence $X$ strongly minimal implies that $X$ is $\chi$-minimal, 
by the Theorem.

Suppose that $f:X\to Y$ is a 2-connected degree 1 map and $\chi(X)=\chi(Y)$.
Then $K_2(f)$ is a finitely generated projective $\mathbb{Z}[\pi]$-module 
and $\mathbb{Z}\otimes_{\mathbb{Z}[\pi]}K_2(f)=0$.
If $X$ is finite then $X$ is a stably free $\mathbb{Z}[\pi]$-module, 
so $K_2(f)=0$, by a result of Kaplansky  \cite{Ro84}.
This also holds if $\pi$ satisfies the Weak Bass Conjecture \cite{Ec86}.
In either case, $f$ is a homotopy equivalence, and
so $\chi$-minimality implies order minimality.
\end{proof}

In particular, every sequence of 2-connected degree 1 maps 
\[
X_1\to X_2\to X_3\to\dots
\] 
eventually becomes a sequence of homotopy equivalences.
If $f:X\to Z$ is a 2-connected degree-1 map and $Z$ is strongly minimal 
then $\lambda_f=\lambda_X$.
 
\begin{cor}
\label{SBC}
Suppose that either $X$ is finite or $\pi$ 
satisfies the Strong Bass Conjecture.
If $\beta_1^{(2)}(X)=\chi(X)=0$ 
then $X$ is strongly minimal.
\end{cor}

\begin{proof}
In this case the $L^2$ Euler characteristic formula gives 
$\beta_2^{(2)}(X)=0$.
Hence $\beta_2^{(2)}(X)=\beta_2^{(2)}(\pi)$.
\end{proof}

Strong minimality has the disadvantage of limited applicability.
However, the case of greatest interest to us is when $c.d.\pi\leq2$. 
The three notions of minimality are then equivalent,
and order minimality is equivalent to 
strong minimality if and only if $c.d.\pi\leq2$.
(See Theorems \ref{cd2} and \ref{om=sm} below,
and \cite{Hi04a} for  $\pi$ a free group.)

If $\pi\cong\mathbb{Z}^r$ and $X$ is $\chi$-minimal 
then $X$ is order minimal.
However, $X$ can only be strongly minimal if $r=1$, 2 or 4.
The 4-torus $\mathbb{R}^4/\mathbb{Z}^4$ is the unique strongly minimal 
$PD_4$-complex with fundamental group $\mathbb{Z}^4$, 
since $E^s\mathbb{Z}=0$ if $s\leq3$ for this group.
Hence $q(\mathbb{Z}^4)=0$.
Let $K$ be the 2-complex corresponding to the standard presentation 
of $\mathbb{Z}^4$ with four generators and six relators, and let
$N$ be a regular neighbourhood of an embedding of $K$ in $\mathbb{R}^5$.
Then $M=\partial{N}$ is an orientable 4-manifold with 
$\pi_1(M)\cong\mathbb{Z}^4$ and $\chi(M)=6$.
If a 2-connected degree 1 map ${f:M\to Y}$ is not a homotopy equivalence
then $\chi(Y)<\chi(M)$ and so $\beta_2(Y)<12$.
Since $c_Y^*H^2(\mathbb{Z}^4;\mathbb{Z})$ has rank 6 it follows easily 
from Poincar\'e duality in $Y$ that $c_Y^*H^2(\mathbb{Z}^4;\mathbb{Z})$ 
cannot be self-annihilating with respect to cup product, 
and so $c_Y$ has nonzero degree.
However $c_{M*}[M]=0$, since $c_M$ factors through $N$,
and so there can be no such map $f$.
Thus $M$ is order-minimal, but not $\chi$-minimal, and not strongly minimal.

If $Z$ is strongly minimal and $\pi\cong{G_1*G_2}$ does $Z$ decompose 
accordingly as a connected sum?
If so, the hypothesis that $\pi$ have one end would not be
needed in our consideration later of groups of cohomologicial dimension 2.
If $M$ is a closed 4-manifold and $\pi_1(M)\cong{G_1*G_2}$ then 
there is a simply-connected 4-manifold $N$ such that 
$M\#N\cong{P_1\#P_2}$,
where $\pi_1(P_i)\cong{G_i}$ for $i=1,2$.
(See Theorem 14.10 of \cite{Hi}.)
If $p_i:P_i\to{Z_i}$ are strongly minimal models then 
$p=p_1\#p_2:M\#N\to{Z_1\#Z_2}$ is a strongly minimal model for $M\#N$.
The image of $\pi_2(N)$ generates a projective direct summand of
$\pi_2(M\#N)$ on which the intersection pairing is nonsingular,
and so $p$ factors through $M$, by the construction of Theorem \ref{exist} below.
Thus $M$ has a strongly minimal model which is a connected sum.

A strongly minimal 4-manifold $M$ must be of type II or III, 
since $\alpha^*v_2(\widetilde{M})$ is the normal 
Stiefel-Whitney class $w_2(\nu_\alpha)$, 
for $\alpha$ an immersion of $S^2$ in $\widetilde{M}$ 
with normal bundle $\nu_\alpha$, and so $v_2(\widetilde{M})([\alpha])$
is the {\it mod}-2 self-intersection number of $[\alpha]\in\pi_2(M)$.
Is there a purely homotopy-theoretic argument showing that
all strongly minimal $PD_4$-complexes are of type II or III?
(This is so if  $c.d.\pi=2$, by Theorem \ref{cd2min} below.)

\begin{lemma}
\label{v2type}
Let $f:X\to{Z}$ be a $2$-connected degree-$1$ map of $PD_4$-complexes
with fundamental group $\pi$.
If $X$ is of type {\rm{II}} or {\rm{III}} then so is $Z$.
\end{lemma}

\begin{proof}
Since $f$ is 2-connected, 
$c_X=gc_Zf$, for some self homotopy equivalence $g$ of $K(\pi,1)$.
If $v_2(X)=c_X^*V$ for some $V\in{H^2(\pi;\mathbb{F}_2)}$ then
\[
f^*(v_2(Z)\cup\alpha)=f^*(\alpha^2)=v_2(X)\cup{f^*\alpha}=f^*(c_Z^*g^*V\cup\alpha),
\]
for all $\alpha\in{H^2(Z;\mathbb{F}_2)}$.
Hence $v_2(Z)=c_Z^*g^*V$, since $H^4(f;\mathbb{F}_2)$ is an isomorphism.
\end{proof}

The converse is false.
For instance, the blowup of a ruled surface is of type I, 
but its minimal models are of type II or III.
(See \S14 below.)

If $X$ has $v_2$-type I and $c.d.\pi=2$ 
is there a model $f:X\to{Z}$ with $v_2(Z)=0$?

\begin{lemma}
\label{virtual}
Let $Z$ be a $PD_4$-complex with fundamental group $\pi$,
and let $Z_\rho$ be the covering space associated to a subgroup $\rho$ of finite index in $\pi$.
Then $Z$ is strongly minimal if and only if $Z_\rho$ is strongly minimal.
\end{lemma}

\begin{proof}
Let $\Pi=\pi_2(Z)$.
Then $\pi_2(Z_\rho)\cong\Pi|_\rho$.
Moreover, $H^2(\pi;\mathbb{Z}[\pi])|_\rho\cong{H^2}(\rho;\mathbb{Z}[\rho])$
and $Hom_{\mathbb{Z}[\pi]}(\Pi,\mathbb{Z}[\pi])|_\rho\cong
{Hom}_{\mathbb{Z}[\rho]}(\Pi|_\rho,\mathbb{Z}[\rho])$,
as right $\mathbb{Z}[\rho]$-modules, since $[\pi:\rho]$ is finite.
The lemma follows from these observations.
\end{proof}

\section{existence of strongly minimal models}

In this section we shall obtain a criterion for the existence of a strongly
minimal model, as a consequence of the following theorem, which may be thought
of as a converse to the 4-dimensional case of Wall's Lemma 2.2 and Theorem 5.2.

\begin{theorem}
\label{exist}
Let $X$ be a $PD_4$-complex  with fundamental group $\pi$.
If $K$ is a finitely generated projective direct summand of $H^2(X;\mathbb{Z}[\pi])$ such that $\lambda_X$ induces 
a nonsingular pairing on $K\times{K}$ then there is a $PD_4$-complex $Z$ 
and a $2$-connected degree-$1$ map $f:X\to Z$ with $K_2(f)=D(K)$.
\end{theorem}

\begin{proof}
Suppose first that $K$ is stably free and choose maps $m_i:S^2\to X$ 
for $1\leq i\leq s$ representing generators of $D(K)$,
and such that the kernel of the corresponding epimorphism 
$m:\mathbb{Z}[\pi]^s\to{D(K)}$ is free of rank $t$.
Attach $s$ 3-cells to $X$ along the $m_i$ to obtain a cell complex $Y$
with $\pi_1(Y)\cong\pi$, $\pi_2(Y)\cong\Pi/D(K)$ and
$H_3(Y;\mathbb{Z}[\pi])\cong H_3(X;\mathbb{Z}[\pi])\oplus\mathbb{Z}[\pi]^t$.
Since the Hurewicz map is onto in degree 3 for 1-connected spaces
(such as $\widetilde Y$)
we may then attach $t$ 4-cells to $Y$ along maps whose Hurewicz images
form a basis for $H_3(Y,X;\mathbb{Z}[\pi])$ to obtain a cell complex $Z$ with 
$\pi_1(Z)\cong\pi$ and $\pi_2(Z)\cong\Pi/D(K)$.

If $K$ is not stably free then $K\oplus F\cong F$, 
where $F$ is free of countable rank,
and we first construct $Y$ by attaching countably many 2- and 3-cells to $X$,
and then attach countably many 4-cells to $Y$ to obtain $Z$ as before.

The inclusion $f:X\to Z$ is 2-connected and 
$\mathrm{Ker}(H_2(f;\mathbb{Z}[\pi]))=D(K)$.
Comparison of the equivariant chain complexes for $X$ and $Z$ shows that
$H_i(f;\mathbb{Z}[\pi])$ is an isomorphism for all $i\not=2$,
while $H^j(f;\mathbb{Z}[\pi])$ is an isomorphism for all $j\not=2$ or $3$,
and $H^2(f;\mathbb{Z}[\pi])$ is a monomorphism.
The connecting homomorphism in the long exact sequence for the cohomology of $(Z,X)$ 
with coefficients $\mathbb{Z}[\pi]$ induces an isomorphism from the summand 
$K\leq{H^2(X;\mathbb{Z}[\pi])}$ to $H^3(Z,X;\mathbb{Z}[\pi])=Hom_{\mathbb{Z}[\pi]}(D(K),\mathbb{Z}[\pi])$.
Therefore $H^3(Z;\mathbb{Z}[\pi])=0$.

Let $[Z]=f_*[X]\in H_4(Z;\mathbb{Z}^w)$.
Cap product with $[Z]$ gives isomorphisms
$\overline{H^j(Z;\mathbb{Z}[\pi]})\cong H_{4-j}(Z;\mathbb{Z}[\pi])$ for $j\not=2$, 
by the projection formula $f_*([X]\cap f^*\alpha)=[Z]\cap\alpha$.
This is also true when $j=2$,
for then $H^2(f;\mathbb{Z}[\pi])$ identifies $H^2(Z;\mathbb{Z}[\pi])$
with the orthogonal complement of $K$ in $H^2(X;\mathbb{Z}[\pi])$,
and $f_*([X]\cap-)$ carries this isomorphically to $H_2(Z;\mathbb{Z}[\pi])$.
Therefore $Z$ is a $PD_4$-complex with fundamental class $[Z]$, 
$f$ has degree 1 and $K_2(f)=D(K)$.
\end{proof}

This construction derives from \cite{Hi04b}, via \cite{Hi06}.
The main theorem of \cite{HRS} includes a similar result, 
for $X$ a closed orientable 4-manifold and $K$ a free module.

\begin{cor}
\label{existcor}
The $PD_4$-complex $X$ has a strongly minimal model if and only if
$H/E$ is a finitely generated projective $\mathbb{Z}[\pi]$-module
and $\lambda_X$ is nonsingular.
\end{cor}

\begin{proof}
If $f:X\to Z$ is a 2-connected degree-$1$ map then
$K^2(f)=\mathrm{Cok}(H^2(f;\mathbb{Z}[\pi]))$ 
is a finitely generated projective direct summand of $H^2(X;\mathbb{Z}[\pi])$,
by Lemma 2.2 of \cite{Wa}.
If $Z$ is strongly minimal the inclusion $E\to\overline{H^2(Z;\mathbb{Z}[\pi])}$
is an isomorphism,
and so  $H/E\cong\overline{K_2(f)}$.
Hence the conditions are necessary.
If they hold the construction of Theorem \ref{exist} 
gives a strongly minimal model for $X$.
\end{proof}


The above conditions hold if $\Pi^\dagger$ is a finitely generated projective
$\mathbb{Z}[\pi]$-module and $E^3\mathbb{Z}=0$.
In particular, they hold if $c.d.\pi\leq2$, 
by an elementary argument using Schanuel's Lemma and duality.
(See Theorem \ref{cd2} below).
On the other hand, if $c.d.\pi=3$ then no $PD_4$-complex with fundamental
group $\pi$ is strongly minimal.
For if $\lambda_X=0$ then $E^3\mathbb{Z}\cong(E^2\mathbb{Z})^\dagger$,
by Lemma \ref{lambdaE},
and this condition cannot hold, by the next lemma.

\begin{lemma}
\label{cd3}
Let $\pi$ be a finitely presentable group such that $c.d.\pi\leq3$ and
let $E=E^2\mathbb{Z}$.
If $E^3\mathbb{Z}\cong{E^\dagger}$ then $c.d.\pi\leq2$.
\end{lemma}

\begin{proof}
Let $P_*$ be a projective resolution of $\mathbb{Z}$, of length 3.
Then $\partial_3^\dagger:P_2^\dagger\to{P_3^\dagger}$ is a presentation 
for $E^3\mathbb{Z}$.
Hence $(E^3\mathbb{Z})^\dagger=\mathrm{Ker}(\partial_3^{\dagger\dagger})=
\mathrm{Ker}(\partial_3)=0$.
But then $E^3\mathbb{Z}\cong{E^\dagger}\cong{E}^{\dagger\dagger\dagger}=0$.
Hence $\partial_3$ is a split injection, and so $c.d.\pi\leq2$.
\end{proof}

Surgery on a factor of the 4-torus $\mathbb{R}^4/\mathbb{Z}^4$ 
gives a closed 4-manifold $M$ with $\pi\cong\mathbb{Z}^3$ 
and $\chi(M)=2$.
This 4-manifold is $\chi$-minimal, by Lemma 3.11 of \cite{Hi},
and is order minimal, by Corollary \ref{SBC},
but cannot be strongly minimal, 
since $c.d.\pi=3$.

The condition $E^3\mathbb{Z}\cong(E^2\mathbb{Z})^\dagger$ 
is far from characterizing the 
fundamental groups of strongly minimal $PD_4$-complexes.
In \S9--\S14 we shall determine such groups within certain subclasses.
In all cases considered,
$\pi$ has finitely many ends 
(i.e., $\pi$ is virtually cyclic or $E^1\mathbb{Z}=0$) 
and  $E^3\mathbb{Z}=0$.

\begin{lemma}
\label{sharp}
Let $f:X\to{Z}$ be a $2$-connected degree-$1$ map of $PD_4$-complexes
with fundamental group $\pi$.
Then $k_1(Z)=f_\#(k_1(X))$ and $k_1(X)=f_{!\#}k_1(Z)$, 
where $f_\#$ and $f_{!\#}$ are the change-of-coefficients homomorphisms 
induced by $\pi_2(f)$ and the umkehr homomorphism.
If $E^3\mathbb{Z}=0$ then these are mutually inverse isomorphisms.
\end{lemma}

\begin{proof} 
Since $K_2(f)$ is projective, $\pi_2(X)\cong\pi_2(Z)\oplus{K_2(f)}$,
with projection onto the first factor given by $\pi_2(f)$ and split by the umkehr map $f_!$.

Let $q:Q\to{Z}$ be the pullback of $P_3(f):P_3(X)\to{P_3(Z)}$ over
the inclusion of $Z$ into $P_3(Z)$. 
Then $q$ is a fibration with homotopy fibre $K(K_2(f),2)$ and $f=qg$, 
where $g:X\to{Q}$ and $P_3(g)$ is a homotopy equivalence.
Hence $\pi_2(g)$ is an isomorphism and $k_1(Q)=g_\#k_1(X)$.
This fibration is determined by a $k$-invariant in $H^3(Z;K_2(f))\cong{H_1(Z;K_2(f))}$,
which is 0 since $K_2(f)$ is projective.
Hence $k_1(Q)=g_\#f_{!\#}k_1(Z)$.
Therefore $k_1(X)=f_{!\#}k_1(Z)$, since $g_\#$ is an isomorphism,
and so $f_\#k_1(X)=f_\#f_{!\#}k_1(Z)=k_1(Z)$.

The second assertion follows easily from the fact that $\pi_2(f)$ 
is an epimorphism with kernel $K_2(f)$ 
a finitely generated projective direct summand of $\Pi=\pi_2(X)$ 
and the hypothesis $E^3\mathbb{Z}=0$, 
which implies that $H^3(\pi;K_2(f))=0$.
\end{proof}

In particular, if $X$ has a strongly minimal model then $k_1(X)$
derives from $H^3(\pi;E^2\mathbb{Z})$. 
Are there such examples with $k_1(X)\not=0$?
The simplest examples for testing that we have found are the groups $\pi=A^2_3*_CA^3_2$, 
where $A_n=\mathbb{Z}^n*\mathbb{Z}^n$ and $C$ is either trivial or $\mathbb{Z}^4$.
These groups have $c.d.\pi=6$.
Mayer-Vietoris arguments show that if $C=1$ then
$E^1\mathbb{Z}\cong{E^2}\mathbb{Z}\cong{E^3}\mathbb{Z}\cong\mathbb{Z}[\pi]$,
while if $C=\mathbb{Z}^4$ then $E^1\mathbb{Z}=0$ (i.e., $\pi$ has one end)
and ${E^2}\mathbb{Z}\cong{E^3}\mathbb{Z}\cong\mathbb{Z}[\pi]$.
In each case it follows that $H^3(\pi;E^2\mathbb{Z})\cong\mathbb{Z}[\pi]$.
These groups are right angled Artin groups.
Perhaps the ``smallest"  RAAG with similar cohomological properties
is the one given by the 1-skeleton of a minimal triangulation 
of $S^2\times{S^1}$, which has 10 generators and 40 relators
but is less easily described explicitly.
(This group has one end and $c.d.=4$.)

\section{reduction}

The main result of this section implies that when a $PD_4$-complex $X$ 
has a strongly minimal model $Z$ its homotopy type is determined 
by $Z$ and $\lambda_X$.

\begin{lemma}
\label{CPinfty}
Let $\beta_\xi=B_{\mathbb{Z}^n}(b_{(\mathbb{CP}^\infty)^n }(\xi))$, for
$\xi\in{H}_4((\mathbb{CP}^\infty)^n;\mathbb{Z})$, 
and let $G$ be a group.
Let $u=\Sigma{u_g}g$ and 
$v=\Sigma{v_h}h\in{H}^2((\mathbb{CP}^\infty)^n;\mathbb{Z}[G])\cong 
{H}^2((\mathbb{CP}^\infty)^n;\mathbb{Z})\otimes_\mathbb{Z}\mathbb{Z}[G]$.
Then ${v(u\cap{\xi})}=\Sigma_{g,h\in G}{\beta_\xi(u_g,v_h)}g\bar{h}$,
for all such $u,v$ and $\xi$.
\end{lemma}

\begin{proof}
As each side of the equation is linear in $\xi$ and 
$H_4((\mathbb{CP}^\infty)^n;\mathbb{Z})$ is generated 
by the images of homomorphisms 
induced by maps from $\mathbb{CP}^\infty$ or $(\mathbb{CP}^\infty)^2$, 
it suffices to assume $n=1$ or 2.
Since moreover each side of the equation is bilinear in $u$ and $v$ 
we may reduce to the case $G=1$. 
As these functions have integral values and 
$2(x\otimes y)=(x+y)\otimes(x+y)-x\otimes x-y\otimes y$ in
$H_4((\mathbb{CP}^\infty)^2;\mathbb{Z})$, for all $x,y\in\Pi\cong\mathbb{Z}^2$,
we may reduce further to the case $n=1$, which is easy.
\end{proof}

\begin{lemma}
\label{2dary}
Let $M$ be a finitely generated projective $\mathbb{Z}[\pi]$-module 
and $L=\!L_\pi(M,2)$.
The secondary boundary homomorphism $b_L$ determines an epimorphism
$b'$ from $H_4(L;\mathbb{Z}^w)$ to $\mathbb{Z}^w\otimes_{\mathbb{Z}[\pi]}\Gamma_W(M)$
such that  
\[B_M(b'(x))(u,v)=v(u\cap{x})\quad\mathrm{for}\,\mathrm{all}
\quad u,v\in M^\dagger\quad\mathrm{and}\quad x\in {H_4(L;\mathbb{Z}^w)}.\]
\end{lemma}

\begin{proof}
The homomorphism from $H_4(L;\mathbb{Z}^w)$ to ${H}_4(\pi;\mathbb{Z}^w)$
induced by $c_L$ is an a epimorphism, since $c_L$ has a section $\sigma$.
Since $\widetilde{L}\simeq{K(M,2)}$ the homomorphism $b_{\widetilde{L}}$ 
is an isomorphism and $H_3(\widetilde{L};\mathbb{Z})=0$,
while since $M$ is projective $H_p(\pi;M)=0$ for all $p>0$.
Therefore it follows from the Cartan-Leray spectral sequence of the covering
$\widetilde{L}\to{L}$ that the kernel of the epimorphism induced by $c_L$
is $\mathbb{Z}^w\otimes_{\mathbb{Z}[\pi]}H_4(\widetilde{L};\mathbb{Z})$.
Let $b'(x)=(1\otimes{b}_{\widetilde{L}})(x-\sigma_*c_{L*}(x))$ for all 
$x\in{H_4(L;\mathbb{Z}^w)}$.
Then $b'$ is an epimorphism onto
$\mathbb{Z}^w\otimes_{\mathbb{Z}[\pi]}\Gamma_W(M)$.

Let $x\in{H_4(L;\mathbb{Z}^w)}$ and
$u,v\in{M}^\dagger\cong{H}^2(L;\mathbb{Z}[\pi])$.
Since $M$ is the union of its finitely generated free abelian subgroups
and homology commutes with direct limits there is an $n>0$ and a map 
$k:(\mathbb{CP}^\infty)^n\to\widetilde{L}$ such that $b'(x)$ 
is the image of $k_*(\xi)$
for some $\xi\in H_4((\mathbb{CP}^\infty)^n;\mathbb{Z})$.
Then $B_M(b'(x))(u,v)=ev_M(k_*\xi)(u,v)$.

Suppose that $k^*u=\Sigma{u_g}g$ and $k^*v=\Sigma{v_h}h$ in
$H^2((\mathbb{CP}^\infty)^n;\mathbb{Z}[\pi])$.
Then we have $ev_M(k_*\xi)(u,v)=\Sigma_{g,h\in G}{\beta_\xi(u_g,v_h)}g\bar{h}$,
which is equal to $v(u\cap{k_*\xi})=k^*v(k^*u\cap\xi)$, 
by Lemma \ref{CPinfty}.
Now $x=k_*\xi+\sigma^*u\cap{c_{L*}}x$
and $u\cap\sigma_*c_{L*}x=\sigma_*(\sigma^*u\cap{c_{L*}}x)=0$,
since $H_2(\pi;\mathbb{Z}[\pi])=0$.
Hence $B_M(b'(x))(u,v)=v(u\cap{x})$,
for all $u,v\in M^\dagger$ and $x\in {H_4(L;\mathbb{Z}^w)}$.
\end{proof}

\begin{theorem}
\label{thm06}
Let $g_X:X\to Z$ and $g_Y:Y\to Z$ be $2$-connected degree-$1$ maps 
of $PD_4$-complexes with fundamental group $\pi$.
If $w=w_1(Z)$ is trivial on elements of order $2$ in $\pi$ then there is 
a homotopy equivalence $h:X\to Y$ such that $g_Yh=g_X$ if and only if 
$\lambda_{g_X}\cong\lambda_{g_Y}$ (after changing the sign of $[Y]$,
if necessary).
\end{theorem}

\begin{proof}
The condition $\lambda_{g_X}\cong\lambda_{g_Y}$ is clearly necessary.
Suppose that it holds.

Since $g_X$ and $g_Y$ induce isomorphisms on $\pi_1$,
we may assume that $c_X=c_Zg_X$ and $c_Y=c_Zg_Y$.
Since $g_X$ and $g_Y$ are 2-connected degree-1 maps,
there are canonical splittings $\pi_2(X)={K_2(g_X)}\oplus{N}$ and
$\pi_2(Y)={K_2(g_Y)}\oplus{N}$, where $N=\pi_2(Z)$,
and $K_2(g_X)$ and $K_2(g_Y)$ are projective.
The projections $\pi_2(g_X)$ and $\pi_2(g_Y)$ onto the second factors 
are split by the umkehr homomorphisms.
We may identify $K_2(g_X)^\dagger$ and $K_2(g_Y)^\dagger$ with direct summands of 
$H^2(X;\mathbb{Z}[\pi])$ and $H^2(X;\mathbb{Z}[\pi])$, respectively,
by Lemma 2.2 of \cite{Wa}.
The homomorphism $\theta$ induces an isomorphism $K_2(Y)\cong{M}=K_2(X)$ such that 
$\lambda_{g_Y}=\lambda_{g_X}$ as pairings on $M^\dagger\times{M^\dagger}$.
Hence $\pi_2(X)\cong\pi_2(Y)\cong{\Pi=M\oplus{N}}$.
We may also assume that $M\not=0$, 
for otherwise $g_X$ and $g_Y$ are homotopy equivalences.

Let $g:P=P_2(X)\to{P_2(Z)}$ be the 2-connected map induced by $g_X$.
Then $g$ is a fibration with fibre $K(M,2)$,
and the inclusion of $N$ as a direct summand of $\Pi$ 
determines a section $s$ for $g$.
Since $\pi_2(X)\cong\pi_2(Y)$,
and $k_1(X)=(g_{X!})_\#(k_1(Z))$ and  $k_1(Y)=(g_{Y!})_\#(k_1(Z))$,
by Lemma \ref{sharp}, we see that $P_2(Y)\simeq{P}$.
We may choose the homotopy equivalence so that composition with $g$ 
is homotopic to the map induced by $g_Y$.
(This uses our knowledge of $E_\pi(P)$, as recorded in \S3 above.)

The splitting $\Pi=M\oplus{N}$ also determines a projection 
$q:P\to L=L_\pi(M,2)$.
We may construct $L$ by adjoining 3-cells to $X$ to kill 
the kernel of projection from $\Pi$ onto $M$
and then adjoining higher dimensional cells to kill the higher homotopy.
Let $j:X\to L$ be the inclusion.
Then $B_M(b'(j_*[X]))(u,v)=v(u\cap{j_*[X]})$ for all $u,v\in{M}^\dagger$, 
by Lemma \ref{2dary}.
Using the projection formula and identifying $M^\dagger=H^2(L;\mathbb{Z}[\pi])$
with $K^2(X)$ we may equate this with $\lambda_{g_X}(u,v)$.
Hence $f_{X*}[X]$ and $f_{Y*}[Y]$ have the same image $\lambda_{g_X}= 
\lambda_{g_Y}$ in $Her_w(M^\dagger)$.

Since $P_2(Z)$ is a retract of $P$ comparison of the Cartan-Leray spectral
sequences for the classifying maps $c_P$ and $c_{P_2(Z)}$ shows that 
\[\mathrm{Cok}(H_4(s;\mathbb{Z}^w))\cong
\mathbb{Z}^w\otimes_{\mathbb{Z}[\pi]}(\Gamma_W(\Pi)/\Gamma_W(N)).\]
Since $\pi$ has no orientation reversing element of order 2 the homomorphism 
$B_M$ is injective, by Theorem \ref{wh-herm}, and therefore since
$\lambda_{g_X}=\lambda_{g_Y}$ the images of 
$f_{X*}[X]$ and $f_{Y*}[Y]$ in 
$\mathbb{Z}^w\otimes_{\mathbb{Z}[\pi]}(\Gamma_W(\Pi)/\Gamma_W(N))$
differ by an element of the subgroup 
${\mathbb{Z}^w\otimes_{\mathbb{Z}[\pi]}(M\otimes{N})}$.
Let $c\in{M\otimes{N}}$ represent this difference, 
and let $\gamma\in\Gamma_W(M)$ represent $b'(f_{X*}[X])$.
Since $B_M(1\otimes\gamma)=\lambda_{g_X}$
is nonsingular $\widetilde{B_M}(\gamma)$ is surjective, 
and so we may choose a homomorphism $\theta:M\to{N}$ 
such that 
$(\widetilde{B_M}(\gamma)\otimes1)(t^{-1}(\theta))=(d\otimes1)(c)$.
Hence
$\Gamma_W(\alpha_\theta)(\gamma)-\gamma\equiv{c}$ mod $\Gamma_W(N)$,
by Lemma \ref{alphatheta}.
Let $P(\theta)$ be the corresponding self homotopy equivalence of $P$.
Then
$gP(\theta)=g$ and $P(\theta)_*f_{Y*}[Y]=f_{X*}[X]$ {\it mod} 
$\mathbb{Z}^w\otimes_{\mathbb{Z}[\pi]}\Gamma_W(N)$.
Since $g_{X*}[X]=g_{Y*}[Y]$ in $H_4(Z;\mathbb{Z}^w)$ and hence
$(gf_X)_*[X]=(gf_Y)_*[Y]$ in $H_4(P_2(Z);\mathbb{Z}^w)$ it
follows that $P(\theta)_*f_{Y*}[Y]=f_{X*}[X]$ in $H_4(P;\mathbb{Z}^w)$.

There is then a map $h:X\to Y$ with $f_Yh=f_X$, 
by the argument of Lemma 1.3 of \cite{HK88}. 
Since the orientation characters of $X$ and $Y$ are compatible,
$h$ lifts to a map $h^+:X^+\to Y^+$.
Since $f_X$ and $f_Y$ are 3-connected $\pi_1(h^+)$, $\pi_2(h^+)$ and 
$H_2(h^+;\mathbb{Z})$ are isomorphisms.
Since $M$ is projective and nonzero, $\mathbb{Z}\otimes_{\mathrm{Ker}(w)}M$ 
is a nontrivial torsion free direct summand of $H_2(X^+;\mathbb{Z})$,
and so $h^+$ has degree 1, by Poincar\'e duality.
Hence $h^+$ is a homotopy equivalence, and therefore so is $h$.
\end{proof}

The original version of this result (Theorem 11 of \cite{Hi06}) 
assumed that $k_1(X)=k_1(Y)=0$.
This was relaxed to the condition that ``$k_1(X)=(g_{X!})_\#k_1(Z)$ and
$k_1(Y)=(g_{Y!})_\#k_1(Z)$" in v2 of the present article
(put on the arXiv on 8 October 2013).
The final step is due to Hegenbarth, Pamuk and Repov\v s,
who noted that Poincar\'e duality in $Z$ may be used to establish an equivalent condition \cite{HPR}.
(This observation has been used in the current version of Lemma \ref{sharp}
above.)

The argument for Theorem \ref{thm06} breaks down 
when $\pi=Z/2Z$ and $w$ is nontrivial, for then 
$B_M:\mathbb{Z}^w\otimes_{\mathbb{Z}[\pi]}\Gamma_W(M)\to{Her}_w(M^\dagger)$ 
is no longer injective,
and the intersection pairing is no longer a complete invariant \cite{HKT4}.
Thus the condition on 2-torsion is in general necessary. 

\begin{cor}
If $X$ has a strongly minimal model $Z$ and $\pi$ has no $2$-torsion 
then the homotopy type of $X$ is determined by $Z$ and $\lambda_X$.
\qed
\end{cor}

\begin{cor}
\cite{HRS}
If $g:X\to Z$ is a $2$-connected degree-$1$ map of $PD_4$-complexes such that 
$w_1(Z)$ is trivial on elements of order $2$ in $\pi_1(Z)$
then $X$ is homotopy equivalent to $M\#Z$ with $M$  $1$-connected 
if and only if $\lambda_g$ is extended from a nonsingular pairing 
over $\mathbb{Z}$.
\qed
\end{cor}

The result of \cite{HRS} assumes that $X$ is orientable, $\pi$ is infinite and
either $E^2\mathbb{Z}=0$ or $\pi$ acts trivially on $\pi_2(Z)$.
(In the latter case $Hom_{\mathbb{Z}[\pi]}(\pi_2(Z),\mathbb{Z}[\pi])=0$,
and so $Z$  is strongly minimal.)

\section{realization of pairings}

In this short section we shall show that if $Z$ is a strongly minimal
$PD_4$-complex and $\mathrm{Ker}(w)$ has no element of order $2$
every nonsingular $w$-hermitean pairing on a finitely generated
projective $\mathbb{Z}[\pi]$-module is realized as $\lambda_X$ for some 
$PD_4$-complex $X$ with minimal model $Z$.
This is an immediate consequence of the following stronger result.

\begin{theorem}
\label{constr}
Let $Z$ be a $PD_4$-complex with fundamental group $\pi$ and let $w=w_1(Z)$.
Assume that $\mathrm{Ker}(w)$ has no element of order $2$.
Let $N$ be a finitely generated projective $\mathbb{Z}[\pi]$-module
and $\Lambda$ be a nonsingular $w$-hermitean pairing on $N^\dagger$.
Then there is a $PD_4$-complex $X$ and a $2$-connected degree-$1$ map 
$f:X\to{Z}$ such that $\lambda_f\cong\Lambda$.
\end{theorem}

\begin{proof}
Suppose $N\oplus F_1\cong F_2$, where $F_1$ and $F_2$ are free
$\mathbb{Z}[\pi]$-modules with countable bases $I$ and $J$, respectively.
(These may be assumed finite if $N$ is stably free.)
We may assume $Z=Z_o\cup_\theta{e^4}$ is obtained by attaching a single
4-cell to a 3-complex $Z_o$, by Lemma 2.9 of \cite{Wa}.
Construct a 3-complex $X_o$ with $\pi_2(X_o)\cong\pi_2(Z_o)\oplus{N}$ 
by attaching $J$ 3-cells to $Z_o\vee(\vee^IS^2)$, 
along sums of translates under $\pi$ of 
the 2-spheres in $\vee^IS^2$, 
as in Theorem \ref{exist}.
Let $i:Z_o\to X_o$ be the natural inclusion.
Collapsing $\vee^IS^2$ gives $X_o/\vee^IS^2\simeq Z_o\vee(\vee^JS^3)$, 
and so there is a retraction $q:X_o\to Z_o$.
Let $p:\Pi=\pi_2(X_o)\to{N}$ be the projection with kernel
$\mathrm{Im}(\pi_2(i))$,
and let $j:X_o\to L=L_\pi(N,2)$ be the corresponding map.
Then $\pi_2(ji)=0$ and so $ji$ factors through $K(\pi,1)$.
The map $B_N:\mathbb{Z}^w\otimes_{\mathbb{Z}[\pi]}\Gamma_W(N)\to{Her}_w(N^\dagger)$
is an epimorphism, by Theorem \ref{wh-herm}.
Therefore we may choose $\psi\in\pi_3(X_o)$ so that $B_N([j(\psi)])=\Lambda$.

Let $\phi=\psi-iq\psi+i\theta$. 
Then $q\phi=\theta$ and $j(\phi)=j(\psi)$,
so $B_N([j(\phi)])=\Lambda$.
Let $X=X_o\cup_\phi D^4$.
The retraction $q$ extends to a map $f:X\to{Z}$.
Comparison of the exact sequences for these pairs shows that $f$ 
induces isomorphisms on homology and cohomology in degrees $\not=2$.
In particular, $H_4(X;\mathbb{Z}^w)\cong{H}_4(Z;\mathbb{Z}^w)$.
Let $[X]=f_*^{-1}[Z]$.
Then $f_*(f^*(\alpha)\cap[X])=\alpha\cap[Z]$
for all cohomology classes $\alpha$ on $Z$,
by the projection formula.
Therefore cap product with $[X]$ induces the Poincar\'e duality isomorphisms
for $Z$ in degrees other than 2.
As it induces an isomorphism 
$\overline{H^2(X;\mathbb{Z}[\pi])}\cong{H}_2(X;\mathbb{Z}[\pi])$, 
by the assumption on $\Lambda$, 
$X_\phi$ is a $PD_4$-complex with $\lambda_X\cong\Lambda$.
\end{proof}

\section{strongly minimal models with $\pi_2=0$}

A $PD_4$-complex $Z$ with $\pi_2(Z)=0$ is clearly strongly minimal.

\begin{lemma}
\label{Pi=0}
Let $X$ be a $PD_4$-complex with fundamental group $\pi$. Then 
\begin{enumerate}
\item$\Pi=0$ if and only if $X$ is strongly minimal and $E^2\mathbb{Z}=0$,
and then $E^3\mathbb{Z}=0$;

\item{if $\Pi=0$ and $\pi$ is infinite then the homotopy type of $X$ 
is determined by $\pi$, $w$ and $k_2(X)\in{H^4(\pi;E^1\mathbb{Z})}$.}
\end{enumerate}
\end{lemma}

\begin{proof}
Part (1) follows from part (1) of Lemma \ref{lambdaE}.
If $\Pi=0$ then $P_2(X)\simeq{K(\pi,1)}$ and $\pi_3(Z)\cong{E^1\mathbb{Z}}$, 
by Poincar\'e duality. 
Hence (2) follows from Lemma \ref{post3}.
\end{proof}
 
\begin{theorem}
Let $\pi$ be a finitely presentable group with no $2$-torsion and
such that $E^2\mathbb{Z}=E^3\mathbb{Z}=0$,
and let $w:\pi\to\mathbb{Z}^\times$ be a homomorphism.
Then two $PD_4$-complexes $X$ and $Y$ with fundamental group $\pi$,
$w_1(X)=c_X^*w$, $w_1(Y)=c_Y^*w$ and $\pi_2(X)$ and $\pi_2(Y)$ projective 
$\mathbb{Z}[\pi]$-modules are homotopy equivalent if and only if 
\begin{enumerate}
\item$c_{X*}[X]=\pm {g^*}c_{Y*}[Y]$ in $H_4(\pi;\mathbb{Z}^w)$,
for some $g\in{Aut(\pi)}$ with $wg=w$; 
and 
\item$\lambda_X\cong\lambda_Y$.
\end{enumerate}
\end{theorem}

\begin{proof}
The hypotheses imply that $X$ and $Y$ have strongly minimal models
$Z_X$ and $Z_Y$ with $\pi_2(Z_X)=\pi_2(Z_Y)=0$,
and hence $P_2(Z_X)\simeq P_2(Z_Y)\simeq K(\pi,1)$.
Moreover $H^3(\pi;\pi_2(X))=H^3(\pi;\pi_2(Y))=0$, since $E^3\mathbb{Z}=0$, 
and so the result follows by the argument of Theorem \ref{thm06}.
\end{proof}

In particular, $Z_X\simeq{Z_Y}$.
If $\pi$ also has one end then the minimal model is aspherical. 
See Theorem \ref{asph} below.

Connected sums of complexes with $\pi_2=0$ again have $\pi_2=0$,
and the fundamental groups of such connected sums usually have 
infinitely many ends.
(The sole nontrivial exception is $RP^4\#RP^4$.)
The arguments of \cite{Tu}  can be extended to this situation,
to show that if $\pi$ splits as a free product then 
$Z$ has a corresponding connected sum decomposition \cite{BBH}.
(In particular, if $\pi$ is torsion free then its free factors 
are one-ended or infinite cyclic,
and so the summands are either aspherical or copies of $S^1\times{S^3}$ or $S^1\tilde\times{S^3}$.)

In the next two sections we shall determine the groups $\pi$
with finitely many ends which are fundamental groups of strongly minimal $PD_4$-complexes $Z$ with $\pi_2(Z)=0$.
(Little is known about such complexes with $\pi$ indecomposable 
and having infinitely many ends.
It follows from the results of \cite{Cr} that  the centralizer 
of any element of finite order is finite or has two ends.)

\section{strongly minimal models with $\pi$ virtually free}

If $\pi$ is virtually free (in particular, if it is finite or two-ended) 
then $E^s\mathbb{Z}=0$ for all $s>1$,
and so a strongly minimal $PD_4$-complex $Z$ with fundamental group $\pi$
must have $\pi_2(Z)=0$,
by Lemma \ref{Pi=0}.
Thus if $\pi$ is finite $\widetilde{Z}\simeq{S^4}$,
and so $Z\simeq{S^4}$ or $\mathbb{RP}^4$.
(See Lemma 12.1 of \cite{Hi}.)
Every orientable $PD_n$-complex admits a degree-1 map to $S^n$.
It is well known that the (oriented) homotopy type of a 1-connected
$PD_4$-complex is determined by its intersection pairing and that every 
such pairing is realized by some 1-connected topological 4-manifold.
(See page 161 of \cite{FQ}).
Thus the only finite group we need to consider is $\pi=Z/2Z$.

\begin{theorem}
Let $X$ be a $PD_4$-complex with $\pi_1(X)=Z/2Z$ and let $w=w_1(X)$.
Then $\mathbb{RP}^4$ is a model for $X$ if and only if $w^4\not=0$.
\end{theorem}

\begin{proof}
The condition is clearly necessary.
Conversely, we may assume that $X=X_o\cup e^4$ is obtained by attaching 
a single 4-cell to a 3-complex $X_o$, by Lemma 2.9 of \cite{Wa}.
The map $c_X:X\to\mathbb{RP}^\infty=K(Z/2Z,1)$ 
factors through a map $f:X\to\mathbb{RP}^4$,
and $w=f^*w_1(\mathbb{RP}^4)$, since $w\not=0$.
The degree of $f$ is well-defined up to sign, 
and is odd since $w^4\not=0$.
We may arrange that $f$ is a degree-1 map, 
after modifying $f$ on a disc, if necessary.
(See \cite{Ol}.)
\end{proof}

In particular, $\pi_2(X)$ is projective if and only if $w^4\not=0$.
Can this be seen directly?
The two $\mathbb{RP}^2$-bundles over $S^2$ provide contrasting examples.
If $X={S^2\times\mathbb{RP}^2}$ then $w^3=0$ and 
$\Pi\cong\mathbb{Z}\oplus\mathbb{Z}^w$,
which has no nontrivial projective $\mathbb{Z}[Z/2Z]$-module summand.
Thus ${S^2\times\mathbb{RP}^2}$ is  order minimal but not strongly minimal.
On the other hand, if $X$ is the nontrivial bundle space then
$w^4\not=0$ and $\Pi\cong\mathbb{Z}[Z/2Z]$.

Nonorientable topological 4-manifolds with fundamental group $Z/2Z$ 
are classified up to homeomorphism in \cite{HKT4},
and it is shown there that the homotopy types are determined
by the Euler characteristic, $w^4$, the $v_2$-type and an Arf invariant
(for $v_2$-type III).
The authors remark that their methods show that $\lambda_X$ together with a
quadratic enhancement $q:\Pi\to{Z/4Z}$ due to \cite{KKR}
is also a complete invariant for the homotopy type of such a manifold.

If $\pi=\pi_1(Z)$ has two ends 
and $\pi_2(Z)=0$ then $\widetilde{Z}\simeq S^3$.
Since $\pi$ has two ends it is an extension of 
$\mathbb{Z}$ or the infinite dihedral group $D_\infty={Z/2Z*Z/2Z}$ 
by a finite normal subgroup $F$.
Since $F$ acts freely on $\widetilde{Z}$ it has cohomological period 
dividing 4 and acts trivially on $\pi_3(Z)\cong{H}_3(Z;\mathbb{Z}[\pi])$,
while the action $u:\pi\to\{\pm1\}=Aut(\pi_3(Z))$ induces the usual action of
$\pi/F$ on $H^4(F;\mathbb{Z})$.
The action $u$ and the orientation character $w_1(Z)$
determine each other, and every such group $\pi$ and action $u$ 
is realized by some $PD_4$-complex $Z$ with $\pi_2(Z)=0$.
The homotopy type of $Z$ is determined by $\pi$, $u$ 
and the first nontrivial $k$-invariant in $H^4(\pi;\mathbb{Z}^u)$.
(See Chapter 11 of \cite{Hi}.)

We shall use Farrell cohomology to show that any $PD_4$-complex $X$ with 
$\pi_1(X)\cong\pi$ satisfying corresponding conditions has a strongly minimal 
model.
We refer to the final chaper of \cite{Br} for more information on Farrell
cohomology.

It is convenient to use the following notation.
If $R$ is a noetherian ring and $M$ is a finitely generated $R$-module
let $\Omega^1M=\mathrm{Ker}(\phi)$, where $\phi:R^n\to M$ is any epimorphism,
and define $\Omega^kM$ for $k>1$ by iteration, so that
$\Omega^{n+1}M=\Omega^1\Omega^nM$.
We shall say that two finitely generated $R$-modules $M_1$ and $M_2$
are projectively equivalent ($M_1\simeq{M_2}$) if they are isomorphic up to
direct sums with a finitely generated projective module.
Then these ``syzygy modules" $\Omega^kM$ are finitely generated, 
and are well-defined up to projective equivalence, by Schanuel's Lemma.

\begin{theorem}
\label{2end}
Let $X$ be a $PD_4$-complex such that $\pi=\pi_1(X)$ has two ends.
Then $X$ has a strongly minimal model if and only if 
$\pi$ and the action $u$ of $\pi$ on $H_3(X;\mathbb{Z}[\pi])\cong\mathbb{Z}$ 
are realized by some $PD_4$-complex $Z$ with $\pi_2(Z)=0$.
\end{theorem}

\begin{proof}
If $\pi_2(Z)=0$ then $\widetilde{Z}\simeq S^3$, by 
Poincar\'e duality and the Hurewicz and Whitehead Theorems,
and the conditions on $\pi$ are necessary, 
by Theorem 11.1 and Lemma 11.3 of \cite{Hi}.

Conversely, since $\pi$ is virtually infinite cyclic the conditions imply that
the Farrell cohomology of $\pi$ has period dividing 4 \cite{Fa77}.
We may assume that the chain complex $C_*$ for $\widetilde{X}$
is a complex of finitely generated $\mathbb{Z}[\pi]$-modules. 
Then the modules $B_2$, $Z_2$ $Z_3$ and $\Pi$ are finitely generated,
since $\mathbb{Z}[\pi]$ is noetherian.
The chain complex $C_*$
gives rise to four exact sequences:
\begin{equation*}
\begin{CD}
0\to{Z}_2\to{C}_2\to{C}_1\to{C}_0\to\mathbb{Z}\to0,
\end{CD}
\end{equation*}
\begin{equation*}
\begin{CD}
0\to{Z}_3\to{C}_3\to{B}_2\to0,
\end{CD}
\end{equation*}
\begin{equation*}
\begin{CD}
0\to{B}_2\to{Z}_2\to\Pi\to0
\end{CD}
\end{equation*}
and
\begin{equation*}
\begin{CD}
0\to{C}_4\to{Z}_3\to\mathbb{Z}^u\to0.
\end{CD}
\end{equation*} 
It is clear that $Z_2\simeq\Omega^3\mathbb{Z}$ and $Z_3\simeq\Omega^1B_2$, 
while $\Omega^1Z_3\simeq\Omega^1(\mathbb{Z}^u)$.
The standard construction of a resolution of the middle term of a short exact 
sequence from resolutions of its extremes, applied to the third sequence, 
gives a projective equivalence $\Omega^1Z_2\simeq\Omega^1B_2\oplus\Omega^1\Pi$.
The corresponding sequences for a strongly minimal complex with
the same group $\pi$ and action $u$ give an equivalence
$\Omega^1(\mathbb{Z}^u)\simeq\Omega^1(\Omega^4\mathbb{Z})$.
(This is in turn equivalent to $\Omega^1\mathbb{Z}$, by periodicity.)
Together these equivalences give
\[
\Omega^5\mathbb{Z}\simeq\Omega^2Z_2\simeq\Omega^2B_2\oplus\Omega^2\Pi
\simeq\Omega^1Z_3\oplus\Omega^2\Pi\simeq\Omega^5\mathbb{Z}\oplus\Omega^2\Pi.
\]
Hence ${Ext}^q_{\mathbb{Z}[\pi]}(\Omega^5\mathbb{Z},N)
\cong{Ext}^q_{\mathbb{Z}[\pi]}(\Omega^5\mathbb{Z},N)\oplus
{Ext}^q_{\mathbb{Z}[\pi]}(\Omega^2\Pi,N)$, for all $q>v.c.d.\pi=1$, 
and any $\mathbb{Z}[\pi]$-module $N$.
If $N$ is finitely generated so is 
${Ext}^q_{\mathbb{Z}[\pi]}(\Omega^1\mathbb{Z},N)$, and so
${Ext}^{q+2}_{\mathbb{Z}[\pi]}(\Pi,N)=
{Ext}^q_{\mathbb{Z}[\pi]}(\Omega^2\Pi,N)=0$, for all $q>1$.
Since $\Pi$ is finitely generated ${Ext}^r_{\mathbb{Z}[\pi]}(\Pi,-)$ 
commutes with direct limits and so is 0, for all $r>3$.
Therefore $\Pi$ has finite projective dimension, by
Theorem X.5.3 of \cite{Br}.
There is a Universal Coefficient Spectral Sequence  
$$E_2^{pq}=Ext_{\mathbb{Z}[\pi]}^q (H_p (X;\mathbb{Z}[\pi]),\mathbb{Z}[\pi])
\Rightarrow H^{p+q}(X;\mathbb{Z}[\pi]).$$
Here $E_2^{pq}=0$ unless $p=0$, 2 or 3,
and $E_2^{0q}=E_2^{3q}=0$ if $q>1$,
since $\pi$ is virtually infinite cyclic and 
$\Omega^1(\mathbb{Z}^u)\simeq\Omega^1\mathbb{Z}$.
It follows easily from this spectral sequence and Poincar\'e duality that 
$Ext^s_{\mathbb{Z}[\pi]}(\Pi,\mathbb{Z}[\pi])=0$ for all $s\geq1$.
Since $\Pi$ also has finite projective dimension it is projective.
Hence $X$ has a strongly minimal model, by Theorem \ref{exist}.
\end{proof}

Thus, for instance, an orientable $PD_4$-complex with fundamental group 
$D_\infty$ does not have a strongly minimal model.

We shall summarize here the results of \cite{Hi04a} on the case 
when $\pi\cong{F(r)}$, for some $r\geq1$.
All epimorphisms $w:F(r)\to\mathbb{Z}^\times$ are equivalent up to
composition with an automorphism of $F(r)$.
The ring $\mathbb{Z}[F(r)]$ is a coherent domain of global dimension 2,
for which all projectives are free.
There are just two homotopy types of $\chi$-minimal $PD_4$-complexes
$Z$ with $\pi_1(Z)\cong{F(r)}$,
namely $\#^r(S^1\times{S^3})$ (if $w=0$) and 
$(S^1\tilde\times{S^3})\#(\#^{r-1}(S^1\times{S^3}))$ (if $w\not=0$).
(These are strongly minimal, and so the notions of minimality coincide 
in this case.)
If $X$ is any $PD_4$-complex with $\pi_1(X)\cong{F(r)}$ 
then $\pi_2(X)$ is a free $\mathbb{Z}[F(r)]$-module,
and there is a degree-1 map from $X$ to the minimal model 
with compatible $w$.
Every $w$-hermitean pairing on $\mathbb{Z}[F(r)]^s$ is realizable
by some such $PD_4$-complex, and two such complexes $X$ and $Y$
realizing $(F(r),w)$ are homotopy equivalent if and only if $\lambda_X$ and $\lambda_Y$ are isometric.

The key observation is that if $X$ is a $PD_4$-complex 
with $\pi_1(X)\cong{F(r)}$ then its 3-skeleton is standard:
if $\beta_2(X)=\beta$ then $X\simeq{X_\psi}=X_o\cup_\psi{e^4}$,
where $X_o=\vee^r(S^1\vee{S^3})\vee(\vee^\beta{S^2})$ and $\psi\in\pi_3(X_o)$.
(This is an easy homological argument, relying on Schanuel's Lemma and
the theorems of Hurewicz and Whitehead.)
The main results then follow on exploring how the group $E(X_o)$ acts
on the attaching map $\psi$.
This group is ``large" and its action is easily analyzed.
Most of these results (excepting for the determination of the minimal models)
can also be proven by adapting the arguments of this paper.

Finitely generated virtually free groups provide a potentially 
broader class of examples.
These groups are fundamental groups of finite graphs of finite groups.
The arguments of \cite{Cr} may be adapted to show that if $Z$ 
is a strongly minimal $PD_4$-complex such that $\pi=\pi_1(Z)$ is virtually free 
(so $\pi_2(Z)=0$)
and if $\pi$ has no dihedral subgroup of order $>2$ then it is a free product of groups with two ends \cite{BBH}.
However, not much is known about criteria for 2-connected degree-1 maps
to a specific minimal model.

\section{Strongly minimal models with $\pi$ one-ended}

We begin this section with a general result on the case when $\pi$ has one end.

\begin{lemma}
\label{nolf}
Let $G$ be a group.
If $T$ is a locally-finite normal subgroup of $G$ then $T$ 
acts trivially on $H^j(G;\mathbb{Z}[G])$, for all $j\geq0$.
\end{lemma}

\begin{proof} 
If $T$ is finite then  $H^j(G;\mathbb{Z}[G])\cong{H^j}(G/T;\mathbb{Z}[G/T]))$,
for all $j$, and the result is clear.
Thus we may assume that $T$ and $G$ are infinite.
Hence $H^0(G;\mathbb{Z}[G])=0$, and $T$ acts trivially.
We may write $T=\cup_{n\geq1}{T_n}$ as a strictly increasing union of finite subgroups. 
Then there are short exact sequences \cite{Je}
\[
0\to{\varprojlim}^1{H^{s-1}(T_n;\mathbb{Z}[\pi])}\to{H^s(T;\mathbb{Z}[\pi])}\to
\varprojlim{H^s(T_n;\mathbb{Z}[\pi])}\to0.
\]
Hence $H^s(T;\mathbb{Z}[\pi])=0$ if $s\not=1$ and 
$H^1(T;\mathbb{Z}[\pi])={\varprojlim}^1H^0(T_n;\mathbb{Z}[\pi])$,
and so the LHS spectral sequence collapses to give
$H^j(G;\mathbb{Z}[G])\cong{H^{j-1}}(G/T;H^1(T;\mathbb{Z}[G]))$,
for all $j\geq1$.
Let $g\in{T}$. 
We may assume that $g\in{T_n}$ for all $n$, and so $g$ acts trivially on $H^0(T_n;\mathbb{Z}{G})$,
for all $j$ and $n$.
But then $g$ acts trivially on ${\varprojlim}^1H^0(T_n;\mathbb{Z}[\pi])$,
by the functoriality of the construction.
Hence every element of $T$ acts trivially on $H^{j-1}(G/T;H^1(T;\mathbb{Z}[G]))$,
for all $j\geq1$.
\end{proof}

\begin{theorem}
\label{nofnsgp}
Let $X$ be an orientable, strongly minimal $PD_4$-complex .
If $\pi=\pi_1(X)$ has one end 
then $\pi$ has no non-trivial locally-finite normal subgroup.
\end{theorem}

\begin{proof}
Suppose that $\pi$ has a nontrivial locally-finite normal subgroup $T$.
Since $\pi$ has one end,  
$H_s(X;\mathbb{Z}[\pi])=0$ for $s\not=0$ or 2.
Since $X$ is strongly minimal,
$\Pi=H_2(X;\mathbb{Z}[\pi])\cong\overline{H^2(\pi;\mathbb{Z}[\pi])}$.
Hence $T$ acts trivially on $\Pi$, 
since it acts trivially on $H^2(\pi;\mathbb{Z}[\pi])$, 
by Lemma \ref{nolf}, and $X$ is orientable.

Let $g\in{T}$ have prime order $p$, and let $C=\langle{g}\rangle\cong{Z/pZ}$.
Then $C$ acts freely on $\widetilde{X}$,
which has homology only in degrees 0 and 2.
On considering the homology spectral sequence for the classifying map
$c_{\widetilde{X}/C}:\widetilde{X}/C\to{K(C,1)}$, 
we see that $H_{i+3}(C;\mathbb{Z})\cong H_i(C;\Pi)$, for all $i\geq2$.
(See Lemma 2.10 of \cite{Hi}.)
Since $C$ has cohomological period 2 and acts trivially on $\Pi$, 
there is an exact sequence 
\[
0\to{Z/pZ}\to\Pi\to\Pi\to0.
\]
On the other hand, since $\pi$ is finitely presentable,
$\Pi\cong{H^2(\pi;\mathbb{Z}[\pi])}$ is torsion-free,
by Proposition 13.7.1 of \cite{Ge}.
Hence $T$ has no such element $g$ and so $\pi$ has no such finite normal subgroup.
\end{proof}

As an immediate consequence, if $X$ is strongly minimal,
but not orientable, 
and $\pi$ has one end, then either $\pi$ has no nontrivial 
locally-finite normal subgroup or $\pi\cong\pi^+\times{Z/2Z^-}$, 
and $\pi^+$ has no nontrivial
locally-finite normal subgroup.

If $\pi$ has one end and $E^2\mathbb{Z}=0$ then 
any strongly minimal $PD_4$-complex with fundamental group $\pi$ is aspherical.
Hence $\pi$ is a $PD_4$-group and $K(\pi,1)$
is the unique strongly minimal model.
The next theorem gives several equivalent conditions for 
a $PD_4$-complex with such a group to have a strongly minimal model.

\begin{theorem}
\label{asph}
Let $X$ be a $PD_4$-complex with fundamental group $\pi$ 
such that $\pi$ has one end and $E^2\mathbb{Z}=0$.
Then the following are equivalent:
\begin{enumerate}
\item $X$ has a strongly minimal model;

\item $\pi$ is a $PD_4$-group and $\Pi=\pi_2(X)$ is projective:

\item $\pi$ is a $PD_4$-group,
$w_1(X)=c_X^*w_1(\pi)$ and $c_X$ is a degree-$1$ map;

\item $\pi$ is a $PD_4$-group,
$w_1(X)=c_X^*w_1(\pi)$ and $k_1(X)=0$.

\end{enumerate}
\end{theorem}

\begin{proof}
The equivalence $(1)\Leftrightarrow(2)$ follows from Corollary \ref{existcor}.

If $Z$ is strongly minimal and $E^1\mathbb{Z}=E^2\mathbb{Z}=0$
then $\pi_2(Z)=0$ and $\pi_3(Z)=E^1\mathbb{Z}=0$.
Hence $Z$ is aspherical, 
so $\pi$ is a $PD_4$-group and $Z\simeq{K}= K(\pi,1)$.
Any 2-connected map $f:X\to{K}$ is homotopic to $c_X$ 
(up to composition with a self homotopy equivalence of $K$).
Thus $w_1(X)=c_X^*w_1(\pi)$ and $c_X$ is a degree-1 map.
Conversely, if (3) holds then $K=K(\pi,1)$ is the unique strongly minimal
$PD_4$-complex with fundamental group $\pi$, and $c_X$
is a  2-connected degree-1 map.
Thus $(3)\Leftrightarrow(1)$. 

If (2) or (3) holds then $\Pi=\mathrm{Ker}(\pi_2(c_X))$ is projective,
Since $\pi$ is a $PD_4$-group, 
$H^3(\pi;M)=0$ for any projective module $M$,
and so $k_1(X)=0$.
Conversely, if (4) holds the map $c_P:P=P_2(X)\to{K}$ has a section $s$, 
since $k_1(X)=0$.
We may assume that $K=K_o\cup{e^4}$ and $X=X_o\cup{e^4}$,
where $K_o$ and $X_o$ are 3-complexes.
The restriction $s|_{K_o}$ factors through $X_o$, by cellular approximation,
since $P=X_o\cup\{ cells~of~dim\geq4\}$.
Thus $K_o$ is a retract of $X_o$.
The map $c_X$ induces a commuting diagram of homomorphisms 
between the long exact sequences of the pairs $(X,X_o)$ and $(K,K_o)$,
with coefficients $\mathbb{Z}[\pi]$.
Hence the induced map from $H_4(X,X_0;\mathbb{Z}[\pi])$ to
$H_4(K,K_0;\mathbb{Z}[\pi])$ is an isomorphism.
This homomorphism is one side of a commuting square whose 
opposite side is the corresponding map with coefficients $\mathbb{Z}^w$.
The other sides are epimorphisms, induced by the $w$-twisted augmentation.
It now follows easily that $c_X$ has degree 1.
Thus $(3)\Leftrightarrow(4)$.
\end{proof}

If $\pi$ has one end and $\Pi$ is projective then $c.d.\pi=4$ and
$H^4(\pi;\mathbb{Z}[\pi])\cong\mathbb{Z}$, by part (6) of Lemma \ref{lambdaE}.
Must $\pi$ be a $PD_4$-group?
This is so if also $E^3\mathbb{Z}=0$, 
for then $X$ has a strongly minimal model, by Lemma \ref{lambdaE}
and Theorem \ref{exist}, which must be aspherical.
If $X$ is strongly minimal and $\pi$ is virtually an $r$-dimensional duality group
then $r=1,2$ or 4, and in the latter case $\pi$ is a $PD_4$-group.

The next result follows immediately from Corollary 19 and Theorem \ref{asph}. 

\begin{cor}
Let $X$ and $Y$ be $PD_4$-complexes with fundamental group $\pi$ a $PD_4$-group,
and such that $\pi_2(X)$ and $\pi_2(Y)$ are projective $\mathbb{Z}[\pi]$-modules, 
$w_1(X)=c_X^*w$ and $w_1(Y)=c_Y^*w$, where  $w=w_1(\pi)$.
Then $X$ and $Y$ are homotopy equivalent if and only if $\lambda_X\cong\lambda_Y$. 
\qed
\end{cor}

This corollary and the equivalence of (3) and (4) in the Theorem
are from \cite{CH}.
(It is assumed there that $X$ and $\pi$ are orientable.)
Theorems \ref{asph} and \ref{thm06} give an alternative proof 
of the main result of \cite{CH},
namely that a $PD_4$-complex $X$ with fundamental group $\pi$ a $PD_4$-group
and $w_1(X)=w_1(\pi)$ is homotopy equivalent to $M\# K(\pi,1)$ 
with $M$ $1$-connected if and only if $k_1(X)=0$
and $\lambda_X$ is extended from a nonsingular pairing over $\mathbb{Z}$.

\section{semidirect products and mapping tori}

In this section we shall determine which semidirect products $\nu\rtimes_\alpha\mathbb{Z}$ 
with $\nu$ finitely presentable are fundamental groups of strongly minimal $PD_4$-complexes.

\begin{theorem}
\label{sdp}
Let $\nu$ be a finitely presentable group and let $X$ be a $PD_4$-complex
with fundamental group $\pi\cong\nu\rtimes_\alpha\mathbb{Z}$,
for some automorphism $\alpha$ of $\nu$.
Then the following are equivalent:
\begin{enumerate}
\item $X$ is the mapping torus of a self homotopy equivalence
of a $PD_3$-complex $N$ with fundamental group $\nu$;

\item $X$ is strongly minimal;

\item $\chi(X)=0$.
\end{enumerate}
In general, $X$ has a strongly minimal model if and only if $\Pi^\dagger$ is projective.
\end{theorem}

\begin{proof}
Let $X_\nu$ be the covering space of $X$ corresponding to $\nu$. 
Then $X_\nu$ is the homotopy fibre of a map from $X$ to $S^1$
which corresponds to the projection of $\pi$ onto $\mathbb{Z}$,
and $H_q(X_\nu;k)=0$ for $q>3$ and all coefficients $k$.
The LHS spectral sequence gives an isomorphism
$H^2(\pi;\mathbb{Z}[\pi])|_\nu\cong{H^1(\nu;\mathbb{Z}[\nu])}$
of right $\mathbb{Z}[\nu]$-modules.
Since $\nu$ is finitely presentable it is accessible, 
and hence $H^1(\nu;\mathbb{Z}[\nu])$ is finitely generated as a right 
$\mathbb{Z}[\nu]$-module.
(See Theorems VI.6.3 and IV.7.5 of \cite{DD}.)

Suppose first that $X$ is the mapping torus of 
a self homotopy equivalence of a $PD_3$-complex $N$.
Since $\pi_2(X)|_\nu=\pi_2(N)\cong\overline{H^1(\nu;\mathbb{Z}[\nu])}$
is finitely generated as a left $\mathbb{Z}[\nu]$-module,
$Hom_{\mathbb{Z}[\pi]}(\pi_2(X),\mathbb{Z}[\pi])=0$, 
and so $X$ is strongly minimal.

If $X$ is strongly minimal then
$\pi_2(X)\cong\overline{H^2(X;\mathbb{Z}[\pi])}=
\overline{H^2(\pi;\mathbb{Z}[\pi])}$.
Since $\beta_q(\nu;\mathbb{F}_2)<\infty$ for $q\leq2$ 
and $\pi_2(X_\nu)=\pi_2(X)|_\nu$ is finitely generated 
as a left $\mathbb{Z}[\nu]$-module
$\beta_q(X_{\nu};\mathbb{F}_2)$ is finite for $q\leq2$.
Poincar\'e duality in $X$ gives an isomorphism
$H_3(X_\nu;\mathbb{F}_2)\cong H^1(X;\mathbb{F}_2[\pi/\nu])=\mathbb{F}_2$.
Hence $\beta_q(X_{\nu};\mathbb{F}_2)$ is finite for all $q$,
and so $\chi(X)=0$, 
by a Wang sequence argument applied to the fibration
$X_\nu\to{X}\to{S^1}$.

If $\chi(X)=0$ then $X$ is a mapping torus of a self homotopy equivalence
of a $PD_3$-complex $N$ with $\pi_1(N)=\nu$. 
(See Chapter 4 of \cite{Hi}.)

The indecomposable factors $G_i$ of $\nu=*G_i$ are either $PD_3$-groups 
or virtually free \cite{Cr}, and in either case $H^2(G_i;\mathbb{Z}[G_i])=0$.
Therefore $H^2(\nu;\mathbb{Z}[\nu])=0$ and so $E^3\mathbb{Z}=0$.
The final assertion now follows from the evaluation sequence,
Lemma \ref{lambdaE} and Theorem \ref{exist}.
\end{proof}

The condition that $\nu$ be the fundamental group of a $PD_3$-complex
is quite restrictive.
Mapping tori of self homotopy equivalences of $PD_3$-complexes 
are always strongly minimal,
but other $PD_4$-complexes with such groups 
may have no strongly minimal model.
(See \S5 above for an example with $\pi=\mathbb{Z}^4$.)

If $\nu$ is finite then $\pi$ has two ends, 
and if $\nu$ has one end then $\pi$ is a $PD_4$-group.
If $\nu$ is torsion free and has two ends it is $\mathbb{Z}$,
and so $\pi\cong\mathbb{Z}^2$ or $\mathbb{Z}\rtimes_{-1}\mathbb{Z}$.
More generally, when $\nu$ is a finitely generated free group $F(s)$
(with $s>0$) then $\pi$ has one end and $c.d.\pi=2$.
This broader class of groups is the focus of the rest of this paper.

\section{groups of cohomological dimension $2$}

When $c.d.\pi=2$, we may drop the qualification ``strongly",
by the following theorem. 
(This is also so if $\pi$ is a free group.
The arguments below may be adapted to the latter case, 
which is well understood \cite{Hi04a}.)

\begin{theorem}
\label{cd2}
Let $X$ be a $PD_4$-complex with $\pi_1(X)\cong\pi$ such that $c.d.\pi=2$,
and let $w=w_1(X)$.
Then 
\begin{enumerate}
\item 
$C_*(X;\mathbb{Z}[\pi])$ is $\mathbb{Z}[\pi]$-chain homotopy equivalent to 
$D_*\oplus{P[2]}\oplus D_{4-*}^\dagger$, 
where $D_*$ is a projective resolution of $\mathbb{Z}$,
$P[2]$ is a finitely generated projective module $P$ concentrated 
in degree $2$ and $D_{4-*}^\dagger$ is the conjugate dual of $D_*$, 
shifted to terminate in degree $2$;

\item $\Pi=\pi_2(X)\cong{P}\oplus{E^2\mathbb{Z}}$;

\item $\chi(X)\geq2\chi(\pi)$, with equality if and only if $P=0$;

\item $(E^2\mathbb{Z})^\dagger=0$;

\item $\pi_3(X)\cong\Gamma_W(\Pi)\oplus{E^1}\mathbb{Z}$.
\end{enumerate}
Moreover, $P_2(X)\simeq{L}=L_\pi(\Pi,2)$, and
so the homotopy type of $X$ is determined by $\pi$, $w$, $\Pi$,
and the orbit of $k_2(X)\in{H^4(L;\pi_3(X))}$ 
under the actions of  $Aut_\pi(\pi_3(X))$ and $E_0(L)$.

Every nonsingular $w$-hermitean pairing on a finitely generated 
projective $\mathbb{Z}[\pi]$-module is realized by some such $PD_4$-complex.
\end{theorem}

\begin{proof}
Let $C_*=C_*(X;\mathbb{Z}[\pi])$,
and let $D_*$ be the chain complex with $D_0=C_0$, $D_1=C_1$,
$D_2=\mathrm{Im}(\partial_2^C)$ and $D_q=0$ for $q>2$.
Then 
\[0\to {D_2}\to{D_1}\to{D_0}\to\mathbb{Z}\to0\]
is a resolution of the augmentation module.
Since $c.d.\pi\leq2$ and $D_0$ and $D_1$ are free modules $D_2$ is projective,
by Schanuel's Lemma.
Therefore the epimorphism from $C_2$ to $D_2$ splits,
and so $C_*$ is a direct sum $C_*\cong D_*\oplus(C/D)_*$.
Since $X$ is a $PD_4$-complex $C_*$ is chain homotopy equivalent to
the conjugate dual  $C_{4-*}^\dagger$.
The first two assertions follow easily.

On taking homology with simple coefficients $\mathbb{Q}$, we see that 
$\chi(X)=2\chi(\pi)+\mathrm{dim}_\mathbb{Q}\mathbb{Q}\otimes_\pi{P}$.
Hence $\chi(X)\geq2\chi(\pi)$.
Since $\pi$ satisfies the Weak Bass conjecture \cite{Ec86} and $P$ is projective
$P=0$ if and only if $\mathrm{dim}_\mathbb{Q}\mathbb{Q}\otimes_\pi{P}=0$.

Let $\delta:D_2\to D_1$ be the inclusion.
Then $E^2\mathbb{Z}=\mathrm{Cok}(\delta^\dagger)$
and so $(E^2\mathbb{Z})^\dagger=\mathrm{Ker}(\delta^{\dagger\dagger})$.
But $\delta^{\dagger\dagger}=\delta$ is injective, and so
$(E^2\mathbb{Z})^\dagger=0$.

The indecomposable free factors of $\pi$ are 
either one-ended or infinite cyclic, 
and at least one factor has one end, since $c.d.\pi>1$.
Thus $H_3(\widetilde{X};\mathbb{Z})\cong{E^1\mathbb{Z}}$ is a free
$\mathbb{Z}[\pi]$-module, by Lemma \ref{freeE^1}.
Hence $\pi_3(X)\cong\Gamma_W(\Pi)\oplus{E^1\mathbb{Z}}$. 

Since $c.d.\pi=2$ the first $k$-invariant of $X$ is trivial,
and so $P_2(X)\simeq{L}=L_\pi(\Pi,2)$.
Hence the next assertion follows from Lemma \ref{post3}.

The realization result follows from Theorem \ref{constr}.
\end{proof}

It follows immediately from (2), (3) and Theorem \ref{exist}
that  ``$\chi$-minimal", ``order-minimal" and ``strongly minimal" 
are equivalent, when $c.d.\pi=2$.
We shall henceforth use just ``minimal" for such complexes.

It remains unknown  whether every finitely presentable group $\pi$
with $c.d.\pi=2$ has a finite 2-dimensional $K(\pi,1)$-complex.
We shall write $g.d.\pi=2$ if this is so.

\begin{cor}
Let $X$ and $Y$ be $PD_4$-complexes with fundamental group $\pi$ 
such that $c.d.\pi=2$, and $w_1(X)=c_X^*w$ and $w_1(Y)=c_Y^*w$ for some
homomorphism $w:\pi\to\mathbb{Z}^\times$.
Then $X$ and $Y$ are homotopy equivalent if and only if they
have the same minimal model $Z$ and $\lambda_X\cong\lambda_Y$.
\qed
\end{cor}

The minimal model may not be uniquely determined! 
See \S14 below.

\begin{theorem}
\label{cd2min}
Let $Z$ be a minimal $PD_4$-complex with fundamental group $\pi$ such that 
${c.d.\pi=2}$, and let $w=w_1(Z)$, $L=L_\pi(E^2\mathbb{Z},2)$ and 
$\pi_3=\Gamma_W(E^2\mathbb{Z})\oplus{E^1}\mathbb{Z}$.
Then 
\begin{enumerate}
\item{the homotopy type} of $Z$ is determined by $\pi$, $w$ and the orbit of 
$k_2(Z)\in{H^4}(L;\pi_3)$ under the actions of 
$Aut_\pi(\Gamma_W(E^2\mathbb{Z})\oplus{E^1}\mathbb{Z})$ and $E_0(L)$;
\item{if} $\widehat{Z}$ is another such complex then $P_2(\widehat{Z})\simeq{P_2(Z)}$
if and only if there is an isomorphism $f:\pi_1(\widehat{Z})\cong\pi$ such that $w_1(\widehat{Z})=f^*w$;
\item{the $v_2$-type} of $Z$ is $\mathrm{II}$ or $\mathrm{III}$,
i.e., $v_2(Z)=c_Z^*V$ for some $V\in{H^2}(\pi;\mathbb{F}_2)$;
\item{if} $Z$ is orientable then it has signature $\sigma(Z)=0$;
\item{for every} $v\in{H^2}(\pi;\mathbb{F}_2)$ there is a 
minimal $PD_4$-complex $Z$ with $\pi_1(Z)\cong\pi$, $w_1(Z)=c_Z^*w$ and $v_2(Z)=c_Z^*v$.
\end{enumerate}
\end{theorem}

\begin{proof}
The first assertion follows from Theorem \ref{cd2}, since $P_2(Z)\simeq{L}$.

If $f:\pi_1(\widehat{Z})\cong\pi$ is an isomorphism such that $w_1(\widehat{Z})=f^*w$
then $\pi_2(\widehat{Z})\cong\Pi$ and so $P_2(\widehat{Z})\simeq{P_2(Z)}$.
Conversely,  
$Ext^2_{\mathbb{Z}[\pi]}(\Pi,\mathbb{Z}[\pi])=\mathbb{Z}^w$, 
so  $\pi$ and $\Pi$ determine $w$.

Let $H=c_Z^*H^2(\pi;\mathbb{F}_2)$.
Then $\dim{H^2}(Z;\mathbb{F}_2)=2\dim{H}$, since $\chi(Z)=2\chi(\pi)$,
and $H\cup{H}=0$, since $c.d.\pi=2$.
In particular, $v_2(Z)\cup{h}=h\cup{h}=0$ for all $h\in{H}$.
Therefore $v_2(Z)\in{H}$, by the nonsingularity of Poincar\'e duality.
If $Z$ is orientable a similar argument with coefficients $\mathbb{Q}$
shows that $H^2(Z;\mathbb{Q})$ has a self-orthogonal summand of rank
$\beta_2(\pi)=\frac12\beta_2(Z)$, and so $\sigma(Z)=0$.

We may use a presentation
$\mathcal{P}=\langle {x_1,\dots,x_g}|{w_1,\dots,w_r}\rangle$ for $\pi$
as a pattern for constructing a 5-dimensional handlebody
$D^5\cup{g}h^1\cup{r}h^2\simeq{C(\mathcal{P})}$, 
but we refine the construction by taking non-orientable 1-handles 
for generators $x$ with $w(x)\not=0$ and
using $w_2=v+w^2$ to twist the framings of the 2-handles corresponding to the relators. 
Let $M$ be the boundary of the resulting 5-manifold.
Then $\pi_1(M)\cong\pi$, $w_1(M)=c_M^*w$ and $v_2(M)=c_M^*v$.
Since $E^3\mathbb{Z}=0$ the pairing $\lambda_M$ is nonsingular,
by part (4) of Lemma \ref{lambdaE}.
Hence $M$ has a strongly minimal model $Z$, by Corollary \ref{existcor}.
Since $c_M$ factors through $c_Z$ via a 2-connected degree-1 map,
$Z$ has the required properties.
\end{proof}

The argument for realizing $v$ is taken from \cite{HKT},
where it is shown that if $C(\mathcal{P})$ is aspherical then
the manifold $M$ is itself minimal.

How does $k_2(X)$ determine $v_2(X)$ (and conversely)?
This seems to be a crucial question.
We expect that the orbit of the $k$-invariant 
is detected by the refined $v_2$-type,
but have only proven this in some cases.
(See Theorems \ref{pd2gp} and \ref{min} below.)

Since $f_{X,3}$ is 4-connected $H^4(f_{X,3};\mathbb{F}_2)$ is injective,
and so it is an isomorphism if also $\beta_2(X;\mathbb{F}_2)>0$,
by the nondegeneracy of Poincar\'e duality.
Thus  the ring $H^*(X;\mathbb{F}_2)$ and hence $v_2(X)$ 
should be directly computable from $H^*(P_3(X);\mathbb{F}_2)$.

If $X$ is of $v_2$-type II or III then any minimal model for $X$ must have compatible $v_2$-type, by Lemma \ref{v2type}.
What happens if $v_2(\widetilde{X})\not=0$? 
Does $X$ have a minimal model $Z$ with $v_2(Z)=0$?
(If $\pi$ is a $PD_2$-group then $X$ has minimal models of each type,
by Theorem \ref{s2bdle} below.)

We show next that the class of groups considered here is the largest  
for which every $PD_4$-complex with such a fundamental group 
has a strongly minimal model.

\begin{theorem}
\label{om=sm}
Let $\pi$ be a finitely presentable group and $w:\pi\to\mathbb{Z}^\times$
be a homomorphism.
Then the following are equivalent:
\begin{enumerate}
\item{every} $PD_4$-complex with fundamental group $\pi$ and orientation character $w$ has a strongly minimal model;

\item{every} order minimal $PD_4$-complex with fundamental group $\pi$ and orientation character $w$ is strongly minimal;

\item$c.d.\pi\leq2$.

\end{enumerate}
\end{theorem}

\begin{proof}
The equivalence $(1)\Leftrightarrow(2)$ is clear.

Suppose that (1) holds, 
and let $K$ be a finite $2$-complex with $\pi_1(K)=\pi$.
Then $K$ has a 4-dimensional thickening $N$ which is a handlebody with
only 0-, 1- and 2-handles, and with $w_1(N)=c_N^*w$.
(Cf. the final paragraph of Theorem 30.)
Let $M=D(N)$ be the closed $4$-manifold obtained by doubling $N$, 
and let $j:N\to{M}$ be one of the canonical inclusions.
Then $(\pi_1(M),w_1(M))\cong(\pi,w)$, 
and collapsing the double gives a retraction $r:M\to{N}$.
We may assume that $c_M={c_Nr}$.

Since $N$ is a retract of $M=D(N)$, we have 
\[
H^2(M;\mathbb{Z}[\pi])\cong
{H^2(N;\mathbb{Z}[\pi])}\oplus{H^2(M,N;\mathbb{Z}[\pi])}.
\]
Let $E=E^2\mathbb{Z}$, and $H=\overline{H^2(M;\mathbb{Z}[\pi])}$.
Since $c_M\sim{c_Nr}$, we have
\[
H/E\cong
(\overline{H^2(N;\mathbb{Z}[\pi])}/E)\oplus\overline{H^2(M,N;\mathbb{Z}[\pi])}.
\]
Since $M$ has a strongly minimal model $H/E$ is projective, 
by Corollary \ref{existcor}.
Hence so is the direct summand $\overline{H^2(M,N;\mathbb{Z}[\pi])}$.
This summand is $\overline{H^2(M,N;\mathbb{Z}[\pi])}\cong{H_2(N;\mathbb{Z}[\pi])}$,
by Poincar\'e-Lefshetz duality.

Now $H_2(N;\mathbb{Z}[\pi])\cong{P=H_2(K;\mathbb{Z}[\pi])}$, 
since $K\simeq{N}$.
Hence the augmentation $\mathbb{Z}[\pi]$-module
$\mathbb{Z}$ has a projective resolution of length 3, 
given by $C_*(K;\mathbb{Z}[\pi])$ in degrees $\leq2$
and by the module $P$ in degree 3, with differential $\partial_3$ 
given by the natural inclusion of $P$ as the submodule of 2-cycles.
Thus $c.d.\pi\leq3$.
We also have $E^3\mathbb{Z}\cong{E^\dagger}$, 
since there is a strongly minimal $PD_4$-complex realizing the pair $(\pi,w)$.
Therefore $c.d.\pi\leq2$, by Lemma \ref{cd3}.

The converse implication $(3)\Rightarrow(1)$ follows from Theorem \ref{cd2min}.
\end{proof}

The group $\pi$ is a $PD_2$-group if and only if $E^2\mathbb{Z}$ 
is infinite cyclic \cite{Bo}.
The minimal $PD_4$-complexes are then the total spaces 
of $S^2$-bundles over aspherical closed surfaces, 
by Theorem 5.10 of \cite{Hi}.
We shall review this case in \S14 below.

Otherwise $E^2\mathbb{Z}$ is not finitely generated.
If $\pi\cong\nu\rtimes\mathbb{Z}$, with $\nu$ finitely presentable, 
then $\nu\cong{F(s)}$ for some $s>0$ and $\pi$ has one end.

\begin{theorem}
\label{freebyZ}
Let $\pi=F(s)\rtimes_\alpha\mathbb{Z}$, where $s>0$, 
and let $w:\pi\to\mathbb{Z}^\times$ be a homomorphism.
Then the minimal $PD_4$-complexes $X$ with fundamental group $\pi$
and $w_1(X)=c_X^*w$ are homotopy equivalent to mapping tori,
and their homotopy types may be distinguished by their refined $v_2$-types.
\end{theorem}

\begin{proof}
A $PD_3$-complex $N$ with fundamental group $F(s)$ is homotopy equivalent to
$\#^s(S^2\times{S^1})$ (if it is orientable) or $\#^s(S^2\tilde\times{S^1})$ (otherwise).
There is a natural representation of $Aut(F(s))$ by isotopy classes of based
homeomorphisms of $N$, 
and the group of based self homotopy equivalences $E_0(N)$ 
is a semidirect product $D\rtimes{Aut(F(s))}$,
where $D$ is generated by Dehn twists about nonseparating 2-spheres. 
We may identify $D$ with $(\mathbb{Z}/2\mathbb{Z})^s=H^1(F(s);\mathbb{F}_2)$, 
and then $E_0(N)=(\mathbb{Z}/2\mathbb{Z})^s\rtimes{Aut(F(s))}$,
with the natural action of $Aut(F(s))$ \cite{He77}.

Thus a minimal $PD_4$-complex $X$ with $\pi_1(X)\cong\pi$
is homotopy equivalent to the mapping torus $M(f)$ 
of a based self-homeomorphism $f$ of such an $N$, 
with $w_1(N)=w|_{F(s)}$, and $f$ has image $(d,\alpha)$ in $E_0(N)$.
Let $\delta(f)$ be the image of $d$ in
$H^2(\pi;\mathbb{F}_2)=H^1(F(s);\mathbb{F}_2)/(\alpha-1)H^1(F(s);\mathbb{F}_2)$.
If $g$ is another based self-homeomorphism of $N$
with image $(d',\alpha)$ and $\delta(g)=\delta(f)$
then $d-d'=(\alpha-1)(e)$ for some $e\in{D}$.
Hence $(d,\alpha)$ and $(d',\alpha)$ are conjugate,
and so $M(g)\simeq{M(f)}$.

All minimal $PD_4$-complexes $X$ with $\pi_1(X)=\pi$ 
and $w_1(X)=w$ have the same Postnikov 2-stage $L=P_2(X)$, 
all have $v_2$-type II or III,
and  there is such a $PD_4$-complex $X$ with $v_2(X)=V$, 
for every  $V\in{H^2}(\pi;\mathbb{F}_2)$, by Theorem \ref{cd2}
and its corollary.
Hence the refined $v_2$-type is a complete invariant. 
\end{proof}

If $\beta_1(\pi)>1$ then $N$ may not be determined by $M(f)$.
For instance if $N=S^2\tilde\times{S^1}$ then $M(id_N)=N\times{S^1}$ 
is also the mapping torus of an orientation reversing self homeomorphism 
of $S^2\times{S^1}$.
It is a remarkable fact that if $\pi=F(s)\rtimes_\alpha\mathbb{Z}$,
$s>1$ and $\beta_1(\pi)\geq2$ then $\pi$ is such a semidirect product 
for infinitely many distinct values of $s$ \cite{Bu}.
However this does not affect our present considerations. 

The refined $v_2$-type is also a complete invariant of the homotopy type
of a minimal $PD_4$-complex when $\pi$ is a $PD_2$-group.
This case is treated in \S15 below.
The argument given there is generalized in Theorem \ref{min}
to other 2-dimensional duality groups,
subject to a technical algebraic condition.
This condition holds if $w=1$ and $\pi$ is 
an ascending HNN extension $\mathbb{Z}*_m$, 
by Theorem \ref{BStf},
while if $m$ is even there is an unique minimal model, by Corollary \ref{beta=0}.

\section{realizing $k$-invariants}

For the rest of this paper we shall assume that $\pi$ is a finitely presentable,
$2$-dimensional duality group (i.e., $\pi$ has one end and $c.d.\pi=2$).
The homotopy type of a minimal $PD_4$-complex  $X$ with $\pi_1(X)=\pi$ 
is determined by $\pi$, $w$ and the orbit of $k_2(X)$ under the actions 
of $E_0(L)$ and $Aut(\Gamma_W(\Pi))$, by Corollary 26.
We would like to find more explicit and accessible invariants that 
characterize such orbits.
We would also like to know which $k$-invariants give rise to $PD_4$-complexes.
Note first  that $H_3(\widetilde{X};\mathbb{Z})=H_4(\widetilde{X};\mathbb{Z})=0$, 
 since $\pi$ has one end.

\begin{theorem}
\label{Postcd2}
Let $\pi$ be a finitely presentable, $2$-dimensional duality  group,
and let $w:\pi\to\mathbb{Z}^\times$ be a homomorphism.
Let $\Pi=E^2\mathbb{Z}$ and let $k\in{H^4}(L;\Gamma_W(\Pi))$.
Then 
\begin{enumerate}
\item
there is a $4$-complex $Y$ with $\pi_1(Y)\cong\pi$, $\pi_2(Y)\cong\Pi$,
$\pi_3(Y)\cong\Gamma_W(\Pi)$,
$k_2(Y)=k$ and $H_3(\widetilde{Y};\mathbb{Z})=H_4(\widetilde{Y};\mathbb{Z})=0$ 
if and only if the homomorphism from $H_4(K(\Pi,2);\mathbb{Z})$ to $\Gamma_W(\Pi)$
determined by $p_L^*k$ is an isomorphism;

\item
any such complex $Y$ is finitely dominated, and
we may assume that $Y$ is a finite complex if $\pi$ is of type $FF$;

\item $\overline{H^2(Y;\mathbb{Z}[\pi])}\cong\Pi$;

\item $H_4(Y;\mathbb{Z}^w)\cong\mathbb{Z}$ and cap product with a generator induces 
isomorphisms $\overline{H^p(Y;\mathbb{Z}[\pi])}\cong{H_{4-p}}(Y;\mathbb{Z}[\pi])$, 
for $p\not=2$.
\end{enumerate}
\end{theorem}

\begin{proof}
If $Y$ is such a 4-complex then $p_L^*k$ is an isomorphism,
by the exactness of the Whitehead sequence.

Suppose, conversely, that $p_L^*k$ is an isomorphism.
Let $P(k)$ denote the Postnikov 3-stage determined by 
$k\in{H^4}(L;\Gamma_W(\Pi))$, and let $P=P(k)^{[4]}$.
Let $C_*=C_*(\widetilde{P})$ be the equivariant 
cellular chain complex for $\widetilde{P}$, and
let $B_q\leq{Z_q}\leq{C_q}$ be the submodules
of $q$-boundaries and $q$-cycles, respectively.
Clearly $H_1(C_*)=0$ and $H_2(C_*)\cong\Pi$,
while $H_3(C_*)=0$, since $p_L^*k$ is an isomorphism.
Hence there are exact sequences
\[0\to{B_1}\to{C_1}\to{C_0}\to\mathbb{Z}\to0,\]
\[0\to{B_3}\to{C_3}\to{Z_2}\to\Pi\to0\]
and 
\[0\to{H_4(C_*)}=Z_4\to{C_4}\to{B_3}\to0.\]
Schanuel's Lemma implies that $B_1$ is projective, since $c.d.\pi=2$.
Hence $C_2\cong{B_1}\oplus{Z_2}$ and so $Z_2$ is also projective.
It then follows that $B_3$ is also projective,
and so $C_4\cong{B_3}\oplus{Z_4}$.
Thus $H_4(C_*)=Z_4$ is a projective direct summand of $C_4$.

After replacing $P$ by $P\vee{W}$,
where $W$ is a wedge of copies of $S^4$,
if necessary, we may assume that $Z_4=H_4(P;\mathbb{Z}[\pi])$ is free.
Since $\Gamma_W(\Pi)\cong\pi_3(P)$ the Hurewicz homomorphism 
from $\pi_4(P)$ to $H_4(P;\mathbb{Z}[\pi])$ is onto,
by the exactness of the Whitehead sequence.
We may then attach 5-cells along maps
representing a basis for $Z_4$ to obtain a countable 5-complex $Q$
with 3-skeleton $Q^{[3]}=P(k)^{[3]}$ and
with $H_q(\widetilde{Q};\mathbb{Z})=0$ for $q\geq3$.
The inclusion of $P$ into $P(k)$ extends to a 4-connected map 
from $Q$ to $P(k)$.

Let $D_*$ be  the finite projective resolution of $\mathbb{Z}$ 
determined by a finite presentation for $\pi$.
Dualizing gives a finite projective resolution $E_*=D_{2-*}^\dagger$ 
for $\Pi=E^2\mathbb{Z}$.
Then $C_*(\widetilde{Q})$ is chain homotopy equivalent to $D_*\oplus{E_*[2]}$,
which is a finite projective chain complex.
It follows from the finiteness conditions of Wall that $Q$ 
is homotopy equivalent to a finitely dominated complex $Y$ 
of dimension $\leq4$ \cite{Wa66}.
(The splitting reflects the fact that $c_Y$ is a retraction, 
since $k_1(Y)=0$.)
The homotopy type of $Y$ is uniquely determined by the data, as in Lemma \ref{post3}.

If $\pi$ is of type $FF$ then $B_1$ is stably free, by Schanuel's Lemma.
Hence $Z_2$ is also stably free.
Since dualizing a finite free resolution of $\mathbb{Z}$
gives a finite free resolution of $\Pi=E^2\mathbb{Z}$ we see in turn that
$B_3$ must be stably free, and so $C_*(\widetilde{Y})$ is chain homotopy
equivalent to a finite free complex.
Hence $Y$ is homotopy equivalent to a finite 4-complex \cite{Wa66}.

Condition (3) follows immediately from the 4-term evaluation sequence, 
since $\Pi^\dagger=E^2\mathbb{Z}^\dagger=0$, by part (4) of Theorem \ref{cd2}.

We see easily that $\overline{H^4(Y;\mathbb{Z}[\pi])}=E^2\Pi\cong\mathbb{Z}$
and $H^4(Y;\mathbb{Z}^w)\cong{Ext^2(\Pi;\mathbb{Z}^w)}\cong\mathbb{Z}$.
The homomorphism $\varepsilon_{w\#}:H^4(Y;\mathbb{Z}[\pi])\to{H^4}(Y;\mathbb{Z}^w)$
induced by $\varepsilon_w$ is surjective,
since $Y$ is 4-dimensional, 
and therefore is an isomorphism.
We also have $H_4(Y;\mathbb{Z}^w)\cong{Tor_2(\mathbb{Z}^w;\Pi)}\cong
\mathbb{Z}^w\otimes_\pi\mathbb{Z}[\pi]\cong\mathbb{Z}$.
Let $[Y]$ be a generator of $H_4(Y;\mathbb{Z}^w)$.
Then evaluation on $[Y]$ induces an isomorphism from
$\overline{H^4(Y;\mathbb{Z}[\pi])}$
to $H_0(Y;\mathbb{Z}[\pi])$. 
Hence $-\cap[Y]$ induces isomorphisms  from $\overline{H^p(Y;\mathbb{Z}[\pi])}$
to $H_{4-p}(Y;\mathbb{Z}[\pi])$ for all $p\not=2$,
since $H^p(Y;\mathbb{Z}[\pi])=H_{4-p}(Y;\mathbb{Z}[\pi])=0$ if $p\not=2$ or 4.
\end{proof}


Cap product with $[Y]$ in degree 2 is determined by an integer,
since $\overline{H^2(Y;\mathbb{Z}[\pi])}\cong{E^2\mathbb{Z}}=\Pi\cong
{H_2(Y;\mathbb{Z}[\pi])}$, so $Hom_{\mathbb{Z}[\pi]}(\overline{H^2(Y;\mathbb{Z}[\pi])},H_2(Y;\mathbb{Z}[\pi]))\cong
{End(E^2\mathbb{Z})}\cong\mathbb{Z}$.
The 4-complex $Y$ is a $PD_4$-complex if and only if this integer is $\pm1$.
The obvious question is: what is this integer?
Is it always $\pm1$?
The complex $C_*$ is  chain homotopy equivalent to its dual,
but is the chain homotopy equivalence given by slant product with $[Y]$?

If $\pi$ is either a semidirect product $F(s)\rtimes\mathbb{Z}$
or the fundamental group of a Haken $3$-manifold $M$
then $\widetilde{K}_0(\mathbb{Z}[\pi])=0$,
i.e., projective $\mathbb{Z}[\pi]$-modules are stably free \cite{Wd78}.
(This is not yet known for all torsion free one relator groups.)
In such cases finitely dominated complexes are homotopy finite.

\section{$PD_2$-groups}

The case of most natural interest is when $\pi$ is a $PD_2$-group, 
i.e., is the fundamental group of an aspherical closed surface $F$. 
If $Z$ is the minimal model for such a $PD_4$-complex $X$
then $\Pi=\pi_2(Z)$ and $\Gamma_W(\Pi)$ are infinite cyclic, 
and $Z$ is homotopy equivalent to the total space of a $S^2$-bundle over 
a closed aspherical surface.
(The action $u:\pi\to{Aut}(\Pi)$ is given by $u(g)=w_1(\pi)(g)w(g)$
for all $g\in\pi$, by Lemma 10.3 of \cite{Hi},
while the induced action on $\Gamma_W(\Pi)$ is trivial.)
There are two minimal models for each pair $(\pi,w)$,
distinguished by their $v_2$-type.
This follows easily from the fact that the inclusion of $O(3)$ into
the monoid of self-homotopy equivalences $E(S^2)$ 
induces a bijection on components and an isomorphism on fundamental groups.
(See Lemma 5.9 of \cite{Hi}.)
It is instructive to consider this case 
from the point of view of $k$-invariants also, 
as we shall extend the argument of this section 
to other groups in Theorem \ref{min} below.

Suppose first that $\pi$ acts trivially on $\Pi$. 
Then $L\simeq{K\times{\mathbb{CP}^\infty}}$.
Fix generators $t$, $x$, $\eta$ and $z$ for 
$H^2(\pi;\mathbb{Z})$, $\Pi$, $\Gamma_W(\Pi)$ and 
$H^2(\mathbb{CP}^\infty;\mathbb{Z})=Hom(\Pi,\mathbb{Z})$, 
respectively, such that $z(x)=1$ and $2\eta=[x,x]$.
(These groups are all infinite cyclic,
but we should be careful to distinguish the generators,
as the Whitehead product pairing of $\Pi$ with itself into 
$\Gamma_W(\Pi)$ is not the pairing given by multiplication.)
Let $t,z$ denote also the generators of ${H^2}(L;\mathbb{Z})$ induced 
by the projections to $K$ and $\mathbb{CP}^\infty$, respectively.
Then $H^2(\pi;\Pi)$ is generated by $t\otimes{x}$,
while $H^4(L;\Gamma_W(\Pi))$ is generated by $tz\otimes\eta$ and
$z^2\otimes\eta$.
(Note that $t$ has order 2 if $w_1(\pi)\not=0$.)

\begin{lemma}
The $k$-invariant $k_2(S^2)$ generates $H^4(CP^\infty;\mathbb{Z})$.
\end{lemma}

\begin{proof}
Let $h:CP^\infty\to{K}(\mathbb{Z},4)$ be the fibration with homotopy fibre 
$P_3(S^2)$ corresponding to $k_2(S^2)$.
Since $P_3(S^2)$ may be obtained by adjoining cells of dimension $\geq5$ to 
$S^2$ we see that $H^4(P_3(S^2);\mathbb{Z})=0$.
It follows from the spectral sequence of the fibration that
$h^*$ maps $H^4({K}(\mathbb{Z},4);\mathbb{Z})$ onto 
$H^4(CP^\infty;\mathbb{Z})$,
and so $k_2(S^2)=h^*\iota_{\mathbb{Z},4}$ generates 
$H^4(CP^\infty;\mathbb{Z})$.
\end{proof}

Since $\widetilde{Z}\simeq{S^2}$, the image of $k_2(Z)$ in
$H^4(\widetilde{L};\mathbb{Z})\cong\mathbb{Z}$ generates this group.
Hence $k_2(Z)=\pm(z^2\otimes\eta+mtz\otimes\eta)$ for some $m\in\mathbb{Z}$.
The action of $[K,L]_K=[K,\mathbb{CP}^\infty]\cong{H^2(\pi;\mathbb{Z})}$
on ${H^2}(L;\mathbb{Z})$ is generated by $t\mapsto{t}$ and $z\mapsto{z+t}$.  
The action on $H^4(L;\Gamma_W(\Pi))$ is then given by
$tz\otimes\eta\mapsto{tz}\otimes\eta$ and 
$z^2\otimes\eta\mapsto{z^2\otimes\eta+2tz\otimes\eta}$. 
There are thus two possible $E_0(L)$-orbits of $k$-invariants,
and each is in fact realized by the total space of
an $S^2$-bundle over the surface $K$.

If the action $u$ is nontrivial these calculations 
go through essentially unchanged with 
coefficients $\mathbb{F}_2$ instead of $\mathbb{Z}$.
There are again two possible $E_\pi(L)$-orbits of $k$-invariants,
and each is realized by an $S^2$-bundle space.

In all cases the orbits of $k$-invariants correspond to
the elements of $H^2(\pi;\mathbb{F}_2)=\mathbb{Z}/2\mathbb{Z}$.
In fact the $k$-invariant may be detected by the Wu class.
Let $[c]_2$ denote the image of a cohomology class under reduction 
{\it mod\/} (2).
Since $k_2(Z)$ has image 0 in $H^4(Z;\Pi)$ 
it follows that $[z]_2^2\equiv{m[tz]_2}$ in $H^4(Z;\mathbb{F}_2)$.
This holds also if $\pi$ is nonorientable or the action $u$ is nontrivial,
and so $v_2(Z)=m[z]_2$ and the orbit of $k_2(Z)$ determine each other.

If $X$ is not minimal and $v_2(\widetilde{X})\not=0$ 
then the minimal model $Z$ is not uniquely determined by $X$. 
Nevertheless we have the following results. 

\begin{theorem} 
\label{s2bdle}
Let $E$ be the total space of an $S^2$-bundle over 
an aspherical closed surface $F$, 
and let $X$ be a $PD_4$-complex  with fundamental group $\pi\cong\pi_1(F)$.
Let $\tau$ be the image of the generator of $H^2(\pi;\mathbb{F}_2)$ 
in $H^2(X;\mathbb{F}_2)$.
Then there is a $2$-connected degree-$1$ map $h:X\to E$ such that $c_E=c_Xh$
if and only if 
\begin{enumerate}
\item $(c_X^*)^{-1}w_1(X)=(c_E^*)^{-1}w_1(E)$; and
\item 
$\xi\cup\tau\not=0$ for some $\xi\in H^2(X;\mathbb{F}_2)$
such that $\xi^2=0$ if $v_2(E)=0$ and $\xi^2\not=0$ if $v_2(E)\not=0$.
\end{enumerate}
\end{theorem}

\begin{proof}
See Theorem 10.17 of the current version of \cite{Hi}.
\end{proof}

This is consistent with Lemma \ref{v2type}, 
for if $v_2(X)=0$ then $\xi^2=0$ and $v_2(E)=0$,
while if $v_2(X)=\tau$ then $\xi^2\not=0$,  
and thus $v_2(E)\not=0$ also.
 
If  $w_1(X)=c_X^*w$, where $w=w_1(\pi)$, 
and $v_2(X)=0$ then $E$ must be $F\times{S^2}$,
and we may construct a degree-1 map as follows. 
Let $\Omega$ generate $H^2(\pi;\mathbb{Z}^w)$
and let $x\in H^2(X;\mathbb{Z})$ be such that $(x\cup{c_X^*\Omega})\cap[X]=1$.
Then $[x]_2^2=0$, since $v_2(X)=0$.
Therefore if $F$ is nonorientable $x^2=0$ in $H^4(X;\mathbb{Z})=Z/2Z$;
if $F$ is orientable then $x^2=2k(x\cup{c_X^*\Omega})$ for some $k$,
and we may replace $x$ by $x'=x-kc_X^*\Omega$ 
to obtain a class with square 0.
Such a class may be realized by a map $d:X\to S^2$,
by Theorem 8.4.11 of \cite{Spn}, 
and we may set $h=(c_X,d):X\to F\times{S^2}$.

If $v_2(X)\not=0$ or $\tau$ then there is a $\xi\in{H^2}(X;\mathbb{F}_2)$ 
such that $\xi\cup\tau\not=0$ but $\xi^2=0$.
There is also a class $\zeta$ such that
$\zeta\cup(\tau-v_2(X))=0$ but $\zeta\cup\tau\not=0$. 
Hence $\zeta^2=\zeta\cup\tau\not=0$.
Thus $X$ has minimal models of each $v_2$-type. 

In particular, if $C$ is a smooth projective complex curve of genus $\geq1$ 
and $X=(C\times{\mathbb{CP}^1})\#\overline{\mathbb{CP}^2}$ 
is a blowup of the ruled surface $C\times{\mathbb{CP}^1}=C\times{S^2}$ 
then each of the two orientable $S^2$-bundles 
over $C$ is a minimal model for $X$.
In this case they are also minimal models in the sense 
of complex surface theory. 
(See Chapter VI of \cite{BPV}.)
Many of the other minimal complex surfaces in the Enriques-Kodaira 
classification are aspherical, and hence strongly minimal in our sense.
However 1-connected complex surfaces are never minimal in our sense, 
since $S^4$ is the unique minimal 1-connected $PD_4$-complex
and $S^4$ has no complex structure, by a classical result of Wu.
(See Proposition IV.7.3 of \cite{BPV}.)

\begin{theorem}
\label{pd2gp}
The homotopy type of a $PD_4$-complex $X$  with fundamental group $\pi$ a $PD_2$-group 
is determined by $\pi$, $w_1(X)$, $\lambda_X$ and the $v_2$-type.
\end{theorem}

\begin{proof}
Let $v=w_1(\pi)$, $u=w_1(X)+c_X^*v$, 
and  let $\Omega$ generate $H^2(\pi;\mathbb{Z}^v)$. 
Then $[\Omega]_2$ generates $H^2(\pi;\mathbb{F}_2)$, 
and $\tau=c_X^*[\Omega]_2\not=0$.
If $v_2(X)=m\tau$ and ${p:X\to{Z}}$ is a $2$-connected degree-$1$ map 
then $v_2(Z)=mc_Z^*[\Omega]_2$,
and so there is an unique minimal model for $X$.
Otherwise $\tau\not=v_2(X)$, and so there are 
elements $y,z\in{H^2}(X;\mathbb{F}_2)$ such that $y\cup\tau\not=y^2$ and 
$z\cup\tau\not=0$.
If $y\cup\tau=0$ and $z^2\not=0$ then $(y+z)\cup\tau\not=0$ and $(y+z)^2=0$.
Taking $\xi=y,z$ or $y+z$ appropriately, we have $\xi\cup\tau\not=0$ and $\xi^2=0$.
Hence $X$ has a minimal model $Z$ with $v_2(Z)=0$, by Theorem \ref{s2bdle}.
In all cases the theorem now follows from Theorem \ref{thm06}.
\end{proof}

If $Z$ is strongly minimal and $E^2\mathbb{Z}$ is finitely generated but not 0 
then $E^2\mathbb{Z}$ is infinite cyclic \cite{Bo} and the kernel $\kappa$ of 
the natural action of $\pi$ on $\pi_2(Z)\cong\mathbb{Z}$ is a $PD_2$-group, 
by Theorem 10.1 of \cite{Hi}.
Thus $\pi$ is either a $PD_2$-group or a semidirect product $\kappa\rtimes(Z/2Z)$. 
(In particular, $\pi$ has one end).

\section{ cup products}

In Theorem \ref{min} below we shall use a ``cup-product" argument to
relate cohomology in degrees 2 and 4.
Let $G$ be a group and let $\Gamma=\mathbb{Z}[G]$.
Let $C_*$ and $D_*$ be chain complexes of left $\Gamma$-modules
and $\mathcal{A}$ and $\mathcal{B}$ left $\Gamma$-modules.
Using the diagonal homomorphism from $G$ to $G\times{G}$
we may define {\it internal products\/}
\[H^p(Hom_\Gamma(C_*,\mathcal{A}))\otimes
{H^q}(Hom_\Gamma(D_*,\mathcal{B}))\to
{H^{p+q}}(Hom_\Gamma(C_*\otimes{D_*},
\mathcal{A}\otimes\mathcal{B}))\]
where the tensor products of $\Gamma$-modules
taken over $\mathbb{Z}$ have the diagonal $G$-action.
(See Chapter XI.\S4 of \cite{CE}.)
If $C_*$ and $D_*$ are resolutions of $\mathcal{C}$ and $\mathcal{D}$,
respectively, we get pairings
\[ Ext^p_\Gamma(\mathcal{C},\mathcal{A})\otimes
{Ext^q_\Gamma}(\mathcal{D},\mathcal{B})\to
{Ext^{p+q}_\Gamma}
(\mathcal{C}\otimes\mathcal{D},\mathcal{A}\otimes\mathcal{B}).\]
When $\mathcal{A}=\mathcal{B}=\mathcal{D}$, $\mathcal{C}=\mathbb{Z}$
and $q=0$ we get pairings
\[H^p(G;\mathcal{A})\otimes{End}_G(\mathcal{A})\to
{Ext}^p_{\mathbb{Z}[G]}(\mathcal{A},\mathcal{A}\otimes\mathcal{A}).\]
If instead $C_*=D_*=C_*(\widetilde{S})$ for some space $S$ 
with $\pi_1(S)\cong{G}$ composing with an equivariant diagonal approximation 
gives pairings
\[ H^p(S;\mathcal{A})\otimes{H^q}(S;\mathcal{B})\to
{H^{p+q}}(S;\mathcal{A}\otimes\mathcal{B}).\]
These pairings are compatible with the universal coefficient spectral sequences
$Ext^q_\Gamma(H_p(C_*),\mathcal{A})\Rightarrow 
{H^{p+q}}(C^*;\mathcal{A})=H^{p+q}(Hom_\Gamma(C_*,\mathcal{A}))$, etc.
We shall call these pairings ``cup products",
and use the symbol $\cup$ to express their values.

We wish to show that if $\pi$ is a finitely presentable, $2$-dimensional duality 
group  then cup product  with $id_{\Pi}$ gives an isomorphism
\[
c^2_{\pi,w}:H^2(\pi;\Pi)\to{Ext^2_{\mathbb{Z}[\pi]}(\Pi,\Pi\otimes\Pi)}.
\]
The next lemma shows that these groups are isomorphic; 
we state it in greater generality than we need, 
in order to clarify the hypotheses on the group.

\begin{lemma}
\label{cup}
Let $G$ be a group for which the augmentation (left) module $\mathbb{Z}$ 
has a finite projective resolution $P_*$ of length $n$,
and such that $H^j(G;\Gamma)=0$ for $j<n$. 
Let $\mathcal{D}=H^n(G;\Gamma)$, $w:G\to\mathbb{Z}^\times$ be a homomorphism
and $\mathcal{A}$ be a left $\Gamma$-module.
Then there are natural isomorphisms
\begin{enumerate}
\item 
$\alpha_{\mathcal{A}}:\mathcal{D}\otimes_\Gamma\mathcal{A}\to
H^n(G;\mathcal{A})$; and

\item 
$e_{\mathcal{A}}:Ext^n_\Gamma(\overline{\mathcal{D}},\mathcal{A})\to
\mathbb{Z}^w\otimes_\Gamma\mathcal{A}=\mathcal{A}/I_w\mathcal{A}$.
\end{enumerate}
Hence there is an isomorphism $\theta_{\mathcal{A}}=
\alpha_{\mathcal{A}}e_{\overline{\mathcal{D}}\otimes\mathcal{A}}:
Ext^n_\Gamma(\overline{\mathcal{D}},\overline{\mathcal{D}}\otimes\mathcal{A})
\to{H^n(G;\mathcal{A})}$.
\end{lemma}

\begin{proof}
If $P$ is a finitely generated projective left $\Gamma$-module 
then $Q=Hom_\Gamma(P,\Gamma)$ is a finitely generated {\it right\/} module.
There is a natural isomorphism
$P\cong{Hom}_\Gamma(\overline{Q},\Gamma)$, given by 
$p\mapsto(:f\mapsto\overline{f(p)})$, for all $p\in{P}$ and $f\in\overline{Q}$.
There are also bifunctorial natural isomorphisms of abelian groups
$A_{P\mathcal{A}}:Hom_\Gamma(P,\Gamma)\otimes_\Gamma\mathcal{A}\to
{Hom_\Gamma}(P,\mathcal{A})$ 
given by $A_{P\mathcal{A}}(q\otimes_\Gamma{a})(p)=q(p)a$ for all $a\in\mathcal{A}$,
$p\in{P}$ and $q\in{Hom}_\Gamma(P,\Gamma)$.

We may assume that $P_0=\Gamma$.
Let $Q_j=Hom_\Gamma(P_{n-j},\Gamma)$
and $\partial_i^Q=Hom_\Gamma(\partial^P_{n-j},\Gamma)$.
This gives a resolution $Q_*$ for $\mathcal{D}$ with $Q_n=\Gamma$.
The isomorphisms $A_{P_*\mathcal{A}}$ and $A_{\overline{Q_*}\mathcal{A}}$
induce isomorphisms of chain complexes
$Q_*\otimes_\Gamma\mathcal{A}\to{Hom}_\Gamma(P_{n-*},\mathcal{A})$,
and $\overline{P_*}\otimes_\Gamma\mathcal{A}\to
{Hom}_\Gamma(\overline{Q_{n-*}},\mathcal{A})$,
respectively, from which the first two isomorphisms follow. 
The final assertion follows since 
$\mathbb{Z}^w\otimes_\Gamma(\overline{\mathcal{D}}\otimes\mathcal{A})\cong
\mathcal{D}\otimes_\Gamma\mathcal{A}$.
\end{proof}

If $G$ is finitely presentable and $n=2$ then $G$ is a
2-dimensional duality group.
It is not known whether all the groups considered in the lemma
are duality groups.

\begin{lemma}
\label{shap}
If $G$ satisfies the hypotheses of Lemma \ref{cup}
and $H$ is a subgroup of finite index in $G$ 
then cup product with $id_{\overline{\mathcal{D}}}$ 
is an isomorphism for $(G,w)$ if and only if it is so for $(H,w|_H)$.
\end{lemma}

\begin{proof}
If $\mathcal{A}$ is a left $\mathbb{Z}[G]$-module 
then 
$H^n(G;\mathcal{A})\cong{H^n}(H;\mathcal{A}|_H)$,
by Shapiro's Lemma.
Thus if $G$ satisfies the hypotheses of Lemma \ref{cup} 
the corresponding module for $H$ is $\overline{\mathcal{D}}|_H$.
Further applications of Shapiro's Lemma then give the result.
\end{proof}

In particular, it shall suffice to consider the orientable cases.

Let $\eta:Q_0\to\mathcal{D}$ be the canonical epimorphism,
and let $[\xi]\in H^n(G;\overline{\mathcal{D}})$
be the image of $\xi\in{Hom}_\Gamma(P_n,\overline{\mathcal{D}})$.
Then $\xi\otimes\eta:P_n\otimes\overline{Q_0}\to
\overline{\mathcal{D}}\otimes\overline{\mathcal{D}}$ represents
$[\xi]\cup{id_\mathcal{D}}$ in $Ext^n_\Gamma(\overline{\mathcal{D}},
\overline{\mathcal{D}}\otimes\overline{\mathcal{D}})$.
If $\xi=A_{P_n\overline{\mathcal{D}}}(q\otimes_\Gamma\delta)$ then 
$\alpha_{\overline{\mathcal{D}}}(\eta(q)\otimes_\Gamma\delta)=[\xi]$.
There is a chain homotopy equivalence 
$j_*:\overline{Q}_*\to{P_*\otimes\overline{Q}_*}$,
since $P_*$ is a resolution of $\mathbb{Z}$.
Given such a chain homotopy equivalence, 
$e_{\overline{\mathcal{D}}\otimes\overline{\mathcal{D}}}
([\xi]\cup{id_{\overline{\mathcal{D}}}})$ is the image of
$(\xi\otimes\eta)(j_n(1^*))$,
where $1^*$ is the canonical generator of $\overline{Q}_n$, 
defined by $1^*(1)=1$.

\begin{theorem}
\label{cupthm}
Let $G$ be a finitely presentable, $2$-dimensional duality  group,
and let $w:G\to\mathbb{Z}^\times$ be a homomorphism.
Then $c_{G,w}^2$ is an isomorphism.
\end{theorem}

\begin{proof}
Note first that $G$ satisfies the hypothesis of Lemma \ref{cup}, with $n=2$.
Let $\mathcal{P}=\langle{X}\mid{R}\rangle^\varphi$ 
be a finite presentation for $G$.
(We shall suppress the defining epimorphism $\varphi:F(X)\to{G}$ 
where possible.)
After introducing new generators $x'$ and relators $x'x$, if necessary,
we may assume that each relator is a product of distinct generators,
with all the exponents  positive.
The new presentation $\mathcal{P}'$ 
has the same deficiency as $\mathcal{P}$.
We may also assume that $w=1$, by Lemma \ref{shap}.

The Fox-Lyndon resolution associated to $\mathcal{P}$ gives an  exact sequence
\[0\to{P_3}=\pi_2(C(\mathcal{P}))\to{P_2}=\Gamma\langle{p_r^2;r\in{R}}\rangle\to
{P_1}=\Gamma\langle{p_x^1;x\in{X}}\rangle\to{P_0}=\Gamma\to\mathbb{Z}\to0\]
in which $\partial{p_x^1}=x-1$ and 
$\partial{p_r^2}=\Sigma_{x\in{X}}r_xp_x^1$,
where $r_{x} =\frac{\partial{r}}{\partial{x}}$, for $r\in{R}$ and $x\in{X}$.
Moreover, $P_3$ is projective and $\partial_3$ is a split monomorphism,
since $c.d.G=2$.

Suppose first that the 2-complex $C(\mathcal{P})$ associated to the presentation is aspherical.
(This assumption is not affected by our normalization of the presentations,
for if $C(\mathcal{P})$ is aspherical then $G$ is efficient, 
and $\chi(C(\mathcal{P}'))=def(\mathcal{P}')=\chi(C(\mathcal{P}))$.
Hence $C(\mathcal{P}')$ is also aspherical, by Theorem 2.8 of \cite{Hi}.)
Then $P_3=0$ and the above sequence is a free resolution of $\mathbb{Z}$.
Let $Q_j=Hom_\Gamma(P_{2-j},\Gamma)$
and $\partial_i^Q=Hom_\Gamma(\partial^P_{2-j},\Gamma)$.
Then $\overline{Q}_0=P_2^\dagger$ and
$\overline{Q}_1=P_1^\dagger$ have dual bases $\{q_x^0\}$ and $\{q_r^1\}$,
respectively.
(Thus $q_x^1(p_y^1)=1$ if $x=y$ and 0 otherwise,
and $q_r^0(p_s^2)=1$ if $r=s$ and 0 otherwise.)
Then $\partial 1^*=\Sigma_{x\in{X}}(x^{-1}-1)q_x^1$
and $\partial{q_x^1}=\Sigma_{r\in{R}}\overline{r_x}q_r^0$.
After our normalization of the presentation, 
each $r_x$ is either $0$ or in ${F(X)}$, for all $r\in{R}$ and $x\in{X}$,
and so  $r_x-1=\partial(\Sigma_{y\in{X}}\frac{\partial{r_x}}{\partial{y}}p_y^1)$.

Define homomorphisms $j_i:\overline{Q}_i\to(P_*\otimes\overline{Q}_*)_i$,
for $i=0,1,2$,
 by setting 
\[j_0(q_r^0)=1\otimes{q_r^0}\quad\mathrm{ for }\quad{r\in{R}},
\]
\[j_1(q_x^1)=1\otimes{q_x^1}-
\Sigma_{r,y}\overline{r_x}
(\frac{\partial{r_x}}{\partial{y}}p_y^1\otimes{q_r^0})
\quad\mathrm{ for }\quad{x\in{X}},\quad\mathrm{and}
\]
\[j_2(1^*)=1\otimes1^*-\Sigma_{x\in{X}}x^{-1}(p_x^1\otimes{q_x^1})-
\Sigma_{r\in{R}}(p_r^2\otimes{q_r^0}).\]
Then
\[
\partial{j_1}(q_x^1)-j_0(\partial{q_x^1})=
\Sigma_{r\in{R}}(1\otimes\overline{r_x}q_r^0)-
\Sigma_{r,y}\overline{r_x}(\frac{\partial{r_x}}{\partial{y}}(y-1)\otimes{q_r^0})-
\Sigma_{r\in{R}}\overline{r_x}(1\otimes{q_r^0})
\]
\[=\Sigma_{r\in{R}}[(1\otimes\overline{r_x}q_r^0)-
\overline{r_x}((r_x-1)\otimes{q_r^0}) -\overline{r_x}(1\otimes{q_r^0)}]=0,
\]
and so $\partial{j_1}=j_0\partial$.
Similarly, 
\[\partial{j_2}(1^*)-j_1(\partial1^*)=
\Sigma_x[1\otimes(x^{-1}-1)q_x^1
-x^{-1}((x-1)\otimes{q_x^1})
+\Sigma_r(x^{-1}(p_x^1\otimes\overline{r_x}q_r^0)-
r_xp_x^1\otimes{q_r^0})]
\]
\[
-\Sigma_x(x^{-1}-1)[1\otimes{q_x^1}-\Sigma_{r,y}\overline{r_x}
(\frac{\partial{r_x}}{\partial{y}}p_y^1\otimes{q_r^0})]
\]
\[=\Sigma_{r,x}
[x^{-1}(p_x^1\otimes\overline{r_x}q_r^0)-r_xp_x^1\otimes{q_r^0}+
\Sigma_y(x^{-1}-1)\overline{r_x}(\frac{\partial{r_x}}{\partial{y}}p_y^1\otimes{q_r^0})].
\]
It shall clearly suffice to show that the summand corresponding to each relator $r$ is 0.
After our normalization of the presentation,
we may assume that $r=x_1\dots{x_m}$ for some distinct $x_1,\dots,x_m\in{X}$.
Let $r_i=r_{x_i}$, for  $1\leq{i}\leq{m}$.
Then $r_i=x_1\dots{x_{i-1}}$, for $1\leq{i}\leq{m}$,
so $r_ix_i=r_{i+1}$ if $i<m$ and $r_mx_m=r=1$ in $G$.
Moreover, $\frac{\partial{r}_i}{\partial{y}}=r_{j}$ if $y=x_j$, for some $1\leq{j}<i$,
and is $0$ otherwise.
Let $S_{i,j}=r_i^{-1}(r_jp_{x_j}^1\otimes{q_r^0})$, for $1\leq{j}\leq{i}\leq{m}$.
Then $x_m^{-1}S_{m,j}=S_{1,j}$, for all $j\leq{m}$,
and so  the summand corresponding to the relator $r$ in
$\partial{j_2}(1^*)-j_1(\partial1^*)$ is
\[
\Sigma_{i\leq{m}}(x_i^{-1}S_{i,i}-S_{1,i}+\Sigma_{j<i}(x_i^{-1}S_{i,j}-S_{i,j}))
\]
\[
=\Sigma_{i<m}(S_{i+1,i}-S_{1,i})
+\Sigma_{i\leq{m}}\Sigma_{j<i}(S_{i+1,j}-S_{i,j})).
\]
This sum collapses to  $0$,  and so $\partial{j_2}=j_1\partial$.
Thus $j_*$ is a chain homomorphism.
Since  $\overline{Q}_*$ and $P_*\otimes\overline{Q}_*$ are resolutions of $\mathbb{Z}$
and $j_*$ induces the identity on $\mathbb{Z}$, it is a chain homotopy equivalence.

We then have
\[
(A_{P_2\overline{\mathcal{D}}}(q_s^0\otimes_\Gamma\delta)\otimes\eta)
(j_*(1^*))=-\Sigma_{r\in{R}}(q_s^0(p_r^2)\delta\otimes_\Gamma\eta(q_r^0)),
\]
which has image $-\delta\otimes_\Gamma\eta(q_s^0)$ in
${\mathcal{D}}\otimes_\Gamma\overline{\mathcal{D}}$.
Therefore 
$[\xi]\cup{id_{\overline{\mathcal{D}}}}=
-\theta_{\overline{\mathcal{D}}}(\tau([\xi]))$ 
for $\xi\in{H^2}(G;\overline{\mathcal{D}})$,
where $\tau$ is the ($\mathbb{Z}$-linear) involution of 
$H^2(G;\overline{\mathcal{D}})$ given by
$\tau(\alpha_{\overline{\mathcal{D}}}(\rho\otimes_\Gamma\alpha))=
\alpha_{\overline{\mathcal{D}}}(\alpha\otimes_\Gamma\rho))$,
and so $c_{G,w}^2$ is an isomorphism.

If $C(\mathcal{P})$ is not aspherical
we modify the definition of the dual complex $Q_*$
by setting $Q_1=Hom_\Gamma(P_1,\Gamma)\oplus{Hom_\Gamma(P_3,\Gamma)}$
and extending the differential by $s^\dagger$,
where $s\partial_3=id_{P_3}$.
Let $f:P_3^\dagger\to\Gamma^s$ be a split monomorphism,
with left inverse $g:\Gamma^s\to{P_3^\dagger}$.
Fix a basis $\{e_1,\dots,e_s\}$ for $\Gamma^s$, and define a
homomorphism $h:\Gamma\to\Gamma\otimes\Gamma^s$ by
$h(e_i)=1\otimes{e_i}$.
Then we may extend $j_1$ by setting $j_1=(1\otimes{g})hf$ on $P_3^\dagger$.
\end{proof}

In \cite{Hi09} we gave closed formulae for $j_2(1^*)$ for some simple 
(un-normalized) presentations of groups of particular interest.
We should have also given the appropriate form of $j_1$ explicitly,
for there we used the relators to simplify the derivatives $r_x$,
which in general are sums of monomials $\Sigma_k\pm{r_{xk}}$,
and such simplifications affect the second derivatives $\frac{\partial{r_{xk}}}{\partial{y}}$.
It is safer to calculate such derivatives in $\mathbb{Z}[F(X)]$ 
{\it before\/} using the relators to simplify their images in $\Gamma$.

Similar formulae show that $c^1_{F,w}$ is an isomorphism
for $F$ free of finite rank $r\geq1$.

\section{orbits of the $k$-invariant}

In this section we shall attempt to extend the
argument sketched in \S15 above for the case of $PD_2$-groups
to other finitely presentable, 2-dimensional duality groups.
The hypothesis on 2-torsion in Theorem \ref{min} below
seems necessary for our argument, 
but does not hold in some cases where the result is
known by other means. 

\begin{lemma}
\label{symmod2}
Let $\pi$ be a finitely presentable group such that $c.d.\pi=2$,
and let $w:\pi\to\mathbb{Z}^\times$ be a homomorphism.
Let $\Pi=E^2\mathbb{Z}$. Then there is an exact sequence 
\[
\Pi\odot_\pi\Pi\to\mathbb{Z}^w\otimes_{\mathbb{Z}[\pi]}\Gamma_W(\Pi)
\to{H^2}(\pi;\mathbb{F}_2)\to0.
\]
If $\Pi\odot_\pi\Pi$ is $2$-torsion free this sequence is short exact.
If, moreover, for every $x\in\Pi$ either $x\in(2,I_w)\Pi$ 
or $x\odot{x}\not\in(2,I_w)(\Pi\odot\Pi)$ 
then $\mathbb{Z}^w\otimes_{\mathbb{Z}[\Pi]}\Gamma_W(\pi)$ is $2$-torsion free.
\end{lemma}

\begin{proof}
Since $\Pi$ is torsion free as an abelian group.
 it is a direct limit of free abelian groups,
and so the natural map from $\Pi\odot\Pi$ to $\Gamma_W(\Pi)$ is injective.
Applying $\mathbb{Z}^w\otimes_{\mathbb{Z}[\pi]}-$ to the exact sequence
\begin{equation*}
\begin{CD}
0\to\Pi\odot\Pi@>s>>\Gamma_W(\Pi)@>q_\Pi>>\Pi/2\Pi\to0.
\end{CD}
\end{equation*}
gives the above sequence,
since $\mathbb{Z}^w\otimes_{\mathbb{Z}[\pi]}\Pi/2\Pi\cong\Pi/(2,I_w)\Pi\cong
{H^2}(\pi;\mathbb{F}_2)$.
The kernel on the left in this sequence
is the image of the 2-torsion group
$Tor_1^{\mathbb{Z}[\pi]}(\mathbb{Z}^w,\Pi/2\Pi)$.

If $\Pi\odot_\pi\Pi$ is $2$-torsion free this sequence is short exact, 
and nontrivial 2-torsion 
in $\mathbb{Z}^w\otimes_{\mathbb{Z}[\pi]}\Gamma_W(\Pi)$ 
has nontrivial image in  $\Pi/(2,I_w)\Pi$.
If there is such torsion there are $x, y_i,z_i\in\Pi$ such that
$x\not\in(2,I_w)\Pi$ but $2[\gamma_\Pi(x)+s(\Sigma{y_i}\odot{z_i})]=0$ in $\Pi\odot_\pi\Pi$.
Since $2\gamma_\Pi(x) =s(x\odot{x})$ in $\Gamma_W(\Pi)$, 
we then have 
$s(x\odot{x})\equiv2(-s(\Sigma{y_i}\odot{z_i}))$ {\it mod} $I_w(\Pi\odot\Pi)$,
and so $x\odot{x}\in(2,I_w)(\Pi\odot\Pi)$.
\end{proof}

The final condition in the lemma depends only on the image of $x$ in $\Pi/(2,I_w)\Pi$.

Let $X$ be a $PD_4$-complex with $\pi_1(X)=\pi$ and $\pi_2(X)=\Pi$,
and let $L=L_\pi(\Pi,2)$. 
Then $\widetilde{L}\simeq{K(\Pi,2)}$, and so
it follows from the Whitehead sequence that 
$H_3(\widetilde{L};\mathbb{Z})=0$ and 
$H_4(\widetilde{L};\mathbb{Z})\cong\Gamma_W(\Pi)$.
Let $\mathcal{A}$ be a left $\mathbb{Z}[\pi]$-module.
Since $\pi$ is a 2-dimensional duality group with dualizing module $\overline\Pi$,
Lemma 41 gives canonical isomorphisms 
\[H^2(\pi;\mathcal{A})=Ext^2_{\mathbb{Z}[\pi]}(\mathbb{Z},\mathcal{A})\cong 
\overline\Pi\otimes_{\mathbb{Z}[\pi]}\mathcal{A}\quad\mathrm{and}\quad
{Ext^2_{\mathbb{Z}[\pi]}(\Pi,\mathcal{A})=\mathbb{Z}^w\otimes_{\mathbb{Z}[\pi]}\mathcal{A}}.
\]
The spectral sequence for the universal covering 
$p_L:\widetilde{L}\to{L}$ gives exact sequences
\[
0\to{H^2(\pi;\mathcal{A})}\to{H^2(L;\mathcal{A})}\to
{Hom_{\mathbb{Z}[\pi]}(\Pi,\mathcal{A})=
H^0(\pi;H^2(\widetilde{L};\mathcal{A}))
}\to0
\]
(split by the homomorphism $H^2(\sigma;\mathcal{A})$ induced by a section
$\sigma$ for $c_L$), and 
\begin{equation*}
\begin{CD}
0\to\mathbb{Z}^w\otimes_{\mathbb{Z}[\pi]}\mathcal{A}\to{H^4(L;\mathcal{A})}\!
@>p_L^*>>\!Hom_{\mathbb{Z}[\pi]}(\Gamma_W(\Pi),\mathcal{A})
=H^0(\pi;H^4(\widetilde{L};\mathcal{A}))\to0,
\end{CD}
\end{equation*}
since $c.d.\pi\leq2$.
The right hand homomorphisms are induced by $p_L$, in each case.
The spectral sequence for the universal covering $p_X:\widetilde{X}\to{X}$ 
gives isomorphisms 
$Ext^2_{\mathbb{Z}[\pi]}(\Pi,\mathcal{A}))\cong{H^4(X;\mathcal{A})}$,
and so $f_{X,2}$ induces (non-canonical?) splittings
of the second of these sequences.

In the next theorem and subsequent comments
$p_L^*$ is used variously for homomorphisms determined by
$H^4(p_L;\Gamma_W(\Pi))$, $H^2(p_L;\Pi)$ and $H^4(p_L;\Pi/2\Pi)$.

\begin{theorem}
\label{min}
Let $\pi$ be a finitely presentable, $2$-dimensional duality group,
and let $w:\pi\to\mathbb{Z}^\times$ be a homomorphism.
Let $\Pi=E^2\mathbb{Z}$.
Assume that the image of $\Pi\odot_\pi\Pi$ in ${\mathbb{Z}^w\otimes_{\mathbb{Z}[\pi]}\Gamma_W(\Pi)}$ 
is $2$-torsion free.
Then the homotopy type of a minimal $PD_4$-complex $Z$
with $(\pi_1(Z),w_1(Z))\cong(\pi,w)$ is determined by its refined $v_2$-type.
\end{theorem}

\begin{proof}
Let $Z$ be a minimal $PD_4$-complex with $\pi_1(Z)\cong\pi$ 
and $w_1(Z)=c_Z^*w$.
Then $\pi_2(Z)\cong\Pi$ and $\pi_3(Z)\cong\Gamma_W(\Pi)$, since $\pi$ has one end, 
and the homotopy type of $Z$ is determined by $k=k_2(Z)\in{H^4(L;\Gamma_W(\Pi))}$,
where $\Pi=E^2\mathbb{Z}$ and $L=P_2(Z)=L_\pi(\Pi)$.
This class is only well defined up to the actions of $Aut(\Gamma_W(\Pi))$ 
and $E_0(L)$.
Since $p_L^*k=k_2(\widetilde{Z})$ is an automorphism
(considered as an endomorphism of $\Gamma_W(\Pi)$), 
by part (1) of Theorem \ref{Postcd2}, 
we may assume that $p_L^*k=id_{\Gamma_W(\Pi)}$,
after applying an automorphism of $\Gamma_W(\Pi)$.
Now $E_0(L)\cong{E_\pi(L)\rtimes{Aut(\pi)}}$
and $E_\pi(L)\cong{H^2(\pi;\Pi)\rtimes{Aut(\Pi)}}$.
(See \S3 above).
We shall consider the action of $Aut(\pi)$ in the final paragraph of the proof.
Since  $Aut(\Pi)=\{\pm1\}$ acts trivially on $\Gamma_W(\Pi)$,
the main task is to consider the action of $H^2(\pi;\Pi)$ on $k$.
We shall show that this action is closely related to 
the cup product homomorphism $c^2_{\pi,w}$.
Note also that since $Z$ is minimal,  $v_2(Z)=c_Z^*v$
for some $v\in{H}^2(\pi;\mathbb{F}_2)$, by Theorem \ref{cd2min},
and $E_\pi(L)$ fixes classes induced from $K=K(\pi,1)$, such as $c_L^*v$.

Let $\phi\in{H^2}(\pi;\Pi)$ and let $s_\phi\in[K,L]_K$ and $h_\phi\in[L,L]_K$ 
be as defined in Lemma \ref{tsuki}.
Let $M=L_\pi(\Pi,3)$.
Then $[M,M]_K=H^3(M;\Pi)\cong{End(\Pi)}$, since $c.d.\pi=2$.
Let $\overline\Omega:{[M,M]_K}\to[L,L]_K$ be the loop map.
Let $g\in[M,M]_K$ have image $[g]=\pi_3(g)\in{End(\Pi)}$
and let $f=\overline\Omega{g}$.
Then $\omega([g])=f^*\iota_{\Pi,2}$ defines a homomorphism 
$\omega:End(\Pi)\to{H^2}(L;\Pi)$ such that $p_L^*\omega([g])=[g]$ 
for all $[g]\in{End(\Pi)}$.
Moreover $f\mu=\mu(f,f)$,
since $f=\overline\Omega{g}$, and so $fh_\phi=\mu(fs_\phi{c_L},f)$.
Hence 
\[h_\phi^*\xi=\xi+c_L^*s_\phi^*\xi
\]
 for $\xi=\omega([g])=f^*\iota_{\Pi,2}$.
Naturality of the isomorphisms 
$H^2(X;\mathcal{A})\cong[X,L_\pi(\mathcal{A},2)]_K$ for $X$ a space over $K$
and $\mathcal{A}$ a left $\mathbb{Z}[\pi]$-module implies that
\[s_\phi^*\omega([g])=[g]_\#s_\phi^*\iota_{\Pi,2}=[g]_\#\phi\]
for all $\phi\in{H^2}(\pi;\Pi)$ and $g\in[M,M]_K$.
(See Chapter 5.\S4 of \cite{Ba0}.)

Using our present hypotheses, the exact sequences above give  sequences
\begin{equation}
\begin{CD}
0\to{H^2(\pi;\Pi)}@>c_L^*>>{H^2(L;\Pi)}@>p_L^*>>{End(\Pi)}\to0
\end{CD}
\end{equation}
(split by $\omega$ and the homomorphism $H^2(\sigma;\Pi)$ induced by 
a section $\sigma$ for $c_L$), 
and 
\begin{equation}
\begin{CD}
0\to\mathbb{Z}^w\otimes_{\mathbb{Z}[\pi]}\Gamma_W(\Pi)
\to{H^4(L;\Gamma_W(\Pi))}@>p_L^*>>
End(\Gamma_W(\Pi))\to0.
\end{CD}
\end{equation}
We shall identify the modules on the left with their images,
to simplify the notation.

If $u\in{H^2}(\pi;\Pi)$ then $h_\phi^*(u)=u$,
since $c_Lh_\phi=c_L$.
The induced automorphism of the quotient $End(\Pi)=H^0(\pi;(H^2(\widetilde{L};\Pi))$
is also the identity, 
since the lifts of $h_\phi$ are (non-equivariantly) homotopic 
to the identity in $\widetilde{L}$.
Hence there is a homomorphism 
\[
\delta_\phi:End(\Pi)\to{H^2}(\pi;\Pi)
\]
such that $h_\phi^*(\xi)=\xi+c_L^*\delta_\phi(p_L^*\xi)$ for all $\xi\in{H^2}(L;\Pi)$.
Since $p_L^*c_L^*=0$ and $h_{\phi+\psi}=h_\phi{h_\psi}$ 
it follows that $\delta_\phi$ is additive as a function of $\phi$.
Since $\pi$ is a 2-dimensional duality group,
${H^2(\pi;\Pi)}\cong\overline{\Pi}\otimes_{\mathbb{Z}[\pi]}\Pi$, 
and so 
$\phi=\rho\otimes_\pi\alpha$ for some $\rho\in\overline{\Pi}$ and $\alpha\in\Pi$. 
If $g\in[M,M]_K$ then 
\begin{equation}
\delta_\phi([g])=\delta_\phi(p_L^*\omega([g]))=s_\phi^*\omega([g])=\rho\otimes_\pi[g](\alpha).
\end{equation}
In particular, $\delta_\phi(id_\Pi)=\phi$.

Similarly, the automorphism of $H^4(L;\Gamma_W(\Pi))$ induced by $h_\phi$ 
fixes the subgroup $G=\mathbb{Z}^w\otimes_{\mathbb{Z}[\pi]}\Gamma_W(\Pi)$,
and induces the identity on the quotient $End(\Gamma_W(\Pi))=H^0(\pi;H^4(\widetilde{L};\Gamma_W(\Pi)))$.
Then there is a homomorphism 
\[
f_\phi:H^4(L;\Gamma_W(\Pi))\to{G}
\]
such that $h_\phi^*(u)=u+f_\phi(u)$
for all $u\in{H^4(L;\Gamma_W(\Pi))}$, and such that ${f_\phi|_G=0}$.
Moreover, $f_\phi$ is additive as a function of $\phi$,
so we may define $\widehat{f}:H^2(\pi;\Pi)\to{G}$ by
\[
\widehat{f}(\phi)=f_\phi(k),\quad\mathrm{for~all}~\phi\in{H^2(\pi;\Pi)}.
\]

When $S=L$, $\mathcal{A}=\mathcal{B}=\Pi$, and $p=q=2$ 
the construction of \S15 gives a cup product pairing of $H^2(L;\Pi)$ 
with itself with values in ${H^4(L;\Pi\otimes\Pi)}$.
Since $c.d.\pi=2$ this pairing is trivial on the image 
of $H^2(\pi;\Pi)\otimes{H^2(\pi;\Pi)}$.
The maps $c_L$ and $\sigma$ induce a splitting
$H^2(L;\Pi)\cong{H^2(\pi;\Pi)}\oplus{End(\Pi)}$,
and this pairing restricts to the cup product pairing of 
$H^2(\pi;\Pi)$ with $End(\Pi)$ with values in 
${Ext^2_{\mathbb{Z}[\pi]}}(\Pi,\Pi\otimes\Pi)$.
We may also compose with the natural homomorphisms from 
$\Pi\otimes\Pi$ to $\Pi\odot\Pi$ and $\Gamma_W(\Pi)$ to get pairings
with values in $H^4(L;\Pi\odot\Pi)$ and $H^4(L;\Gamma_W(\Pi))$.

Since $h_\phi^*(\xi\cup\xi')=h_\phi^*\xi\cup{h_\phi^*}\xi'$
we have also
\begin{equation}
f_\phi(\xi\cup\xi')=
\delta_\phi(p_L^*\xi')\cup\xi+\delta_\phi(p_L^*\xi)\cup\xi',
\end{equation}
for all $\xi,\xi'\in{H^2}(L;\Pi)$.
On passing to $\widetilde{L}\simeq{K(\Pi,2)}$ we find that
\begin{equation}
p_L^*(\xi\cup\xi')(\gamma_\Pi(x))=p_L^*\xi(x)\odot{p_L^*}\xi'(x),
\end{equation}
for all $\xi,\xi'\in{H^2(L;\Pi)}$ and $x\in\Pi$.
(To see this, note that the inclusion of $x$ determines 
a map from $\mathbb{CP}^\infty$ to $K(\Pi,2)$,
since $[\mathbb{CP}^\infty,K(\Pi,2)]=Hom(\mathbb{Z},\Pi)$.
Hence we may use naturality of cup products to
reduce to the case when $K(\Pi,2)=\mathbb{CP}^\infty$
and $x$ is a generator of $\Pi=\mathbb{Z}$.)

Let $P$ be the image of $\Pi\odot_\pi\Pi$ in $G$.
Since $c^2_{\pi,w}:H^2(\pi;\Pi)\to
{Ext^2_{\mathbb{Z}[\pi]}}(\Pi,\Pi\otimes\Pi)$
is an isomorphism,
by Theorem \ref{cupthm}, 
the induced map $\widehat{c}:H^2(\pi;\Pi)\to{P}$ is an epimorphism.
Let $e=\widehat{f}-\widehat{c}$.

If $\Xi=\lambda\cup\lambda$ with $p_L^*\lambda=id_\Pi$ then
$p_L^*(\Xi)(\gamma_\Pi(x))=x\odot{x}=2\gamma_\Pi(x)$, for all $x\in\Pi$,
by Equation (5).
Similarly,
$f_\phi(\Xi)=2(\phi\cup\lambda)=2\phi\cup{id_\Pi}$,
by Equation (4) and by the triviality of the cup product on the image 
of $H^2(\pi;\Pi)\otimes{H^2(\pi;\Pi)}$.
Hence
\[
p_L^*(\Xi)=2id_{\Gamma_W(\Pi)}\quad\mathrm{ and}\quad 
f_\phi(\Xi)=2\,\widehat{c}(\phi).
\]
Since $p_L^*k=id_{\Gamma_W(\Pi)}$, we have $p_L^*(2k-\Xi)=0$,
and so $2k-\Xi\in{G}$,
by the exactness of sequence (2) above.
Then 
\[2e(\phi)=f_\phi(2k-\Xi)=0,
\]
since $f_\phi|_G=0$.
Hence $e$ has image in the 2-torsion subgroup $_2G$.

We invoke the hypothesis on 2-torsion at this point.
Since $P\cap_2G=0$,
it follows easily that $|\mathrm {Cok}({\widehat{f}}\,)|\leq|G/P|=|H^2(\pi;\mathbb{F}_2)|$.
As $\phi$ varies in $H^2(\pi;\Pi)$ the values of $h_\phi(k)$ sweep out
a coset of $\mathrm{Im}({\widehat{f}}\,)$ in
$(p_L^*)^{-1}(id_{\Gamma_W(\Pi)})=k+G$,
and there are at most $2^\beta$ cosets, 
where $\beta=\beta_2(\pi;\mathbb{F}_2)$.

For each $v\in{H^2(\pi;\mathbb{F}_2)}$ there is a minimal $PD_4$-complex $Z$
such that $v_2(Z)=c_Z^*v$, by Theorem \ref{cd2}.
The group $Aut(\pi)$ acts on $K$ and $L$ through based self-homotopy equivalences,
and hence acts on the classifying maps $c_Z$ and $f_{Z,2}$ by composition.
These actions induce actions on $H^2(\pi;\mathbb{F}_2)$ and $\Pi$,
and hence on $H^4(L;\Gamma_W(\Pi))$.
The association $k\mapsto{v_2(Z)}$ defines a $Aut(\pi)$-equivariant surjection 
from $(p_L^*)^{-1}(id_{\Gamma_W(\Pi)})=k+G$ to $H^2(\pi;\mathbb{F}_2)$,
which  is constant on cosets of $\mathrm{Im}({\widehat{f}}\,)$,
since $E_\pi(L)$ acts trivially on $H^2(\pi;\mathbb{F}_2)$.
It follows that the refined $v_2$-type is a complete invariant 
for the homotopy types of such complexes.
\end{proof}

If $\mathbb{Z}^w\otimes_{\mathbb{Z}[\pi]}\Gamma_W(\Pi)$
is $2$-torsion free then $\widehat{f}=\widehat{c}$
(since $e=0$),
and the argument can be simplified slightly.

The hypothesis on $2$-torsion holds if $\pi$ is a $PD_2$-group,
for then $\mathbb{Z}^w\otimes_{\mathbb{Z}[\pi]}\Gamma_W(\Pi)\cong\mathbb{Z}$
if $w=1$ and has order 2 otherwise.
(Note that in this case $\Pi\cong\mathbb{Z}^u$, where $u=w+w_1(\pi)$.
We do not assume here that $w=w_1(\pi)$!)
It holds also if $\pi=Z*_m$ with $|m|>1$,
by Theorem \ref{BStf} below.
On the other hand, if $\pi=F(r)\times\mathbb{Z}$ and $w(t)=-1$,
where $t\in\pi$ generates the central $\mathbb{Z}$ factor,
then $\Pi\odot_\pi\Pi$ and $\mathbb{Z}^w\otimes_{\mathbb{Z}[\pi]}\Gamma_W(\Pi)$
have exponent 2, since $t$ acts through $\pm1$ on $\Pi$.
If $r>1$ these groups are not finitely generated, and so the
hypothesis of Theorem \ref{min} does not hold.

\begin{cor}
\label{beta=0}
If $H^2(\pi;\mathbb{F}_2)=0$ and $\Pi\odot_\pi\Pi$ is $2$-torsion free 
there is an unique minimal $PD_4$-complex realizing $(\pi,w)$.
\qed
\end{cor}

Hence two $PD_4$-complexes $X$ and $Y$ with fundamental group $\pi$
are homotopy equivalent if and only if $\lambda_X\cong\lambda_Y$
(i.e., there is an isomorphism $\theta:\pi_1(X)\cong\pi_1(Y)$ such that 
$w_1(X)=w_1(Y)\circ\theta$ and an isometry of the pairings, up to sign.)

The hypothesis $H^2(\pi;\mathbb{F}_2)=0$ holds if $\pi$ is the group 
of a link of 2-spheres in an homology 4-sphere, in particular, 
if it is a 2-knot group or is the fundamental group of an homology 4-sphere.

\begin{cor}
\label{beta=1}
If $H^2(\pi;\mathbb{F}_2)=\mathbb{F}_2$ and the image of $\Pi\odot_\pi\Pi$
in $\mathbb{Z}^w\otimes_{\mathbb{Z}[\pi]}\Gamma_W(\Pi)$ is $2$-torsion free
there are two minimal $PD_4$-complexes realizing $(\pi,w)$,
distinguished by whether $v_2(X)=0$ or not.
\qed
\end{cor}

The work of \cite{HKT} suggests that the refined $v_2$-type 
should be a complete homotopy invariant, 
without the technical hypothesis on 2-torsion 
or the restriction that $\pi$ have one end.
If, moreover, $g.d.\pi=2$ then every such minimal $PD_4$-complex should be
homotopy equivalent to a closed 4-manifold, by Theorem \ref{cd2}.
This is so if $\pi$ is a semidirect product $F(r)\rtimes\mathbb{Z}$
or a $PD_2$-group, by Theorems \ref{sdp} and \ref{pd2gp}.
Can the connection between $k_2$ and $v_2$ be made more explicit?
The canonical epimorphism $q_\Pi:\Gamma_W(\Pi)\to\Pi/2\Pi$ 
determines a change of coefficients homomorphism $q_{\Pi\#}$ 
from sequence (2) above to the parallel sequence
\begin{equation*}
\begin{CD}
0\to{H^2}(\pi;\mathbb{F}_2)
\to{H^4(L;\Pi/2\Pi)}@>p_L^*>>
Hom_{\mathbb{Z}[\pi]}(\Gamma_W(\Pi),\Pi/2\Pi)\to0.
\end{CD}
\end{equation*}
Thus $q_{\Pi\#}(k_2(Z))$ lies in the $H^2(\pi;\mathbb{F}_2)$-coset $(p_L^*)^{-1}(q_\Pi)$.

Does Theorem \ref{s2bdle} have an analogue for other 2-dimensional duality groups?
Let $X$ and $Z$ be $PD_4$-complexes with such a fundamental group $\pi$,
with $Z$ minimal, and such that  $(c_X^*)^{-1}w_1(X)=(c_Z^*)^{-1}w_1(Z)$.
Then $[X,Z]_K$ maps onto $[X,P_3(Z)]_K$, by cellular approximation,
and hence onto $\{f\in[X,L]_K\mid{f^*}k_2(Z)=0\}.$
Can the condition $f^*k_2(Z)=0$ be made more explicit?
The map $f$ corresponds to a class in $H^2(X;\Pi)$ and 
$H^4(X;\Gamma_W(\Pi))\cong\mathbb{Z}^w\otimes_{\mathbb{Z}[\pi]}\Gamma_W(\Pi))$,
by Poincar\'e duality for $X$.
Theorem \ref{s2bdle} suggests that we should consider 
the image of $f^*k_2(Z)$ in $H^2(\pi;\mathbb{F}_2)$, 
under the epimorphism of Lemma \ref{symmod2}.
Apart from this, we must determine when such a map $f$ has 
a degree-1 representative $g:X\to{Z}$.

\section{verifying the torsion condition for $Z*_{m}$}

If $\pi$ is a 2-dimensional duality group but not a $PD_2$-group then 
$\Pi=E^2\mathbb{Z}$ is finitely generated as a left $\mathbb{Z}[\pi]$-module, 
but is not finitely generated as an abelian group.
The associated groups $\Pi\odot_\pi\Pi$ 
and $\mathbb{Z}^w\otimes_{\mathbb{Z}[\pi]}\Gamma_W(\Pi))$ 
are infinitely generated abelian groups with no natural module structure.
In this section we shall investigate the 2-torsion condition.

We consider first groups which have a one-relator presentation
$\mathcal{P}=\langle{X}\mid{r}\rangle$.
It is well-known that if the relator $r$ is not conjugate to a proper power 
then the associated 2-complex $C(\mathcal{P})$ is aspherical,
and so $g.d.\pi\leq2$.
(See \S\S9-11 of Chapter III of \cite{LS}, or \cite{DV}.)

\begin{lemma}
\label{PioPi}
Let $\pi$ be a group with a finite one-relator presentation
$\langle{X}\mid{r}\rangle$ and $c.d.\pi=2$,
and let $w=1$.
Let $\Pi=E^2\mathbb{Z}$.
Then $\Pi\odot_\pi\Pi\cong\mathbb{Z}[\pi]/(U+\overline\Delta)$,
where $\Delta$ is the right ideal generated by the free derivatives $\frac{\partial{r}}{\partial{x}}$, for all $x\in{X}$, 
and $U$ is the subgroup of $\mathbb{Z}[\pi]$ generated by 
$g-g^{-1}$, for all $g\in\pi$.
\end{lemma}

\begin{proof}
On dualizing the Fox-Lyndon resolution of $\mathbb{Z}$ associated to $\langle{X}\mid{r}\rangle$
we see that ${H^2}(\pi;\mathbb{Z}[\pi])\cong\mathbb{Z}[\pi]/\Delta$,
and so $\Pi\cong\mathbb{Z}[\pi]/\overline\Delta$.

Define a function $T:\mathbb{Z}[\pi]\otimes\mathbb{Z}[\pi]\to
\mathbb{Z}[\pi]\otimes\mathbb{Z}[\pi]$ by $T(s\otimes{t})=\bar{s}\otimes{t}$,
for all $s,t\in\mathbb{Z}[\pi]$.
Then $T$ is an additive bijection and
$T(gs\otimes{gt})=\bar{s}\bar{g}\otimes{gt}$,
for all $g\in\pi$.
Hence $T$ induces an additive isomorphism from the quotient 
of $\mathbb{Z}[\pi]\otimes\mathbb{Z}[\pi]$ by the diagonal action of $\pi$ to 
$\mathbb{Z}[\pi]\otimes_{\mathbb{Z}[\pi]}\mathbb{Z}[\pi]\cong\mathbb{Z}[\pi]$,
which maps $s\otimes{t}$ to $\bar{s}t$.
The images of $\mathbb{Z}[\pi]\otimes\overline{\Delta}$ and
$\overline{\Delta}\otimes\mathbb{Z}[\pi]$ under $T$ are 
$\overline{\Delta}$ and $\Delta$, respectively.
We obtain the symmetric product $\mathbb{Z}[\pi]\odot\mathbb{Z}[\pi]$
by factoring out the tensor square $\mathbb{Z}[\pi]\otimes\mathbb{Z}[\pi]$ 
by all sums of terms of the form $s\otimes{t}-t\otimes{s}$.
The image of all such sums in $\mathbb{Z}[\pi]$ is the subgroup $U$.
(Note that $U$ is not usually an ideal!)
Since
$\mathbb{Z}[\pi]\odot_{\mathbb{Z}[\pi]}\mathbb{Z}[\pi]\cong\mathbb{Z}[\pi]/U$
and $U+\overline{\Delta}=U+\Delta$,
we see that $\Pi\odot_\pi\Pi\cong\mathbb{Z}[\pi]/(U+\overline\Delta)$.
\end{proof}

This may be extended to other 2-dimensional duality groups as follows.
Suppose that $P$ is an $a\times{b}$ presentation matrix for $\Pi$.
View $\mathbb{Z}[\pi]^b$ as a module of row vectors, 
with standard basis $\{e_1,\dots,e_b\}$.
Define a function $T:\mathbb{Z}[\pi]^b\otimes\mathbb{Z}[\pi]^b\to
{M_b(\mathbb{Z}[\pi])}$ by $T(se_i\otimes{te_j})=\bar{s}te_{ij}$,
the matrix with $(i,j)$ entry $\bar{s}t$ and all other entries 0.
Then $T(\mathbb{Z}[\pi]^b\otimes\mathrm{Im}(P))$ is $Row(P)$, 
the left ideal in $M_b(\mathbb{Z}[\pi])$ consisting of matrices 
with all rows in $\mathrm{Im}(P)$, while $T(\mathrm{Im}(P)\otimes\mathbb{Z}[\pi]^b)$
is the right ideal $Row(P)^\dagger$, 
the conjugate transpose of $Row(P)$.
Let $V$ be the subgroup generated by $M-M^\dagger$, 
for all $M$ in $M_b(\mathbb{Z}[\pi])$.
Then 
$\Pi\otimes_\pi\Pi\cong
{M_b(\mathbb{Z}[\pi])}/(V+Row(P)+Row(P)^\dagger)$.

Suppose now that $\pi$ is solvable.
Then it is a Baumslag-Solitar group $\mathbb{Z}*_m$, 
with a one-relator presentation $\langle{a,t}\mid{tat^{-1}a^{-m}}\rangle$,
for some $m\not=0$ \cite{Gi79}.
In this case we have a more explicit model for $\Pi\odot_\pi\Pi$.

\begin{theorem}
\label{BStf}
Let $\pi=\mathbb{Z}*_m$ and let ${w:\pi\to\mathbb{Z}^\times}$ be a homomorphism.
Let $\Pi=E^2\mathbb{Z}$.
If $|m|>1$ then $\Pi\odot_\pi\Pi$ is torsion free.
\end{theorem}

\begin{proof} 
We may assume that $\pi$ has the presentation $\langle{a,t}\mid{tat^{-1}a^{-m}}\rangle$.
Let $A=\langle\langle{a}\rangle\rangle$.
Then $\pi\cong{A\rtimes\mathbb{Z}}$.
Let $a_n=t^nat^{-n}$ in $A$, for all $n\in\mathbb{Z}$,
and let $a^x=a^k_{-n}$, for all $x=\frac{k}{m^n}\in\mathbb{Z}[\frac1m]$.
Then $a^0=1$, $a^1=a$ and $a^xa^y=a^{x+y}$ for all $x,y\in\mathbb{Z}[\frac1m]$,
and $x\mapsto{a^x}$ determines an isomorphism from $\mathbb{Z}[\frac1m]$ to $A$.
Every element of $\pi$ is uniquely of the form $t^pa^x$, 
for some $p\in\mathbb{Z}$ and $x\in\mathbb{Z}[\frac1m]$,
and $(t^pa^x)^{-1}=t^{-p}a^{-m^px}$.
If $m$ is even then $w(a^x)=1$ for all $x$; 
if $m$ is odd then $w(a^x)=w(a^{mx})$ for all $x$.

The function which sends $a_n$ to $a_{n+1}$ determines an automorphism
$\alpha$ of the commutative domain
$D=\mathbb{Z}[A]\cong\mathbb{Z}[a_n|n\in\mathbb{Z}]/(a_{n+1}-a_n^m)$,
and $\mathbb{Z}[\pi]$ is isomorphic to the twisted Laurent extension
$D_\alpha[t,t^{-1}]$.
(An explicit isomorphism is given by the function which sends $t^pa_n\in\oplus_{p\in\mathbb{Z}}{t^p}D$
to $t^{n+p}at^{-n}\in\mathbb{Z}[\pi]$ for all $n,p\in\mathbb{Z}$.)

We shall assume henceforth that $m$ is positive, for simplicity of notation.
Let $J_0=\{1,\dots,m-1\}$, let $J_s=\{\frac{d}{m^s}\mid0<d<m^{s+1},~(d,m)=1\}$,
for all $s\geq1$, and let $J=\cup_{s\geq0}J_s$.
Then $E={D/D(a^m-w(a)^m)}$ is freely generated as an abelian group by the image  
of $\{a^x\mid{x}\in{J}\}$.

The images of the free derivatives of the relator $r=tat^{-1}a^{-m}$ 
in $\mathbb{Z}[\pi]$ are $\frac{\partial{r}}{\partial{a}}=t-\mu_m$, 
where $\mu_m=\Sigma_{i=0}^{i=m-1}a^i$,
and $\frac{\partial{r}}{\partial{t}}=1-a^m$.
Hence 
\[
\Pi\cong\mathbb{Z}[\pi]/\mathbb{Z}[\pi](a^m-w(a)^m,t\overline{\mu_m}-w(t))
\cong(\oplus_{k\in\mathbb{Z}}t^kE)/\sim,
\]
where 
\[t^ka^x\sim{w(t)t^k}a^xt\overline{\mu_m}=w(t)t^{k+1}a^{\frac{x}m}\overline{\mu_m},\quad
\mathrm{for~all}~k\in\mathbb{Z}~\mathrm{and}~x\in{J}.
\]
As an abelian group, $\Pi\cong\varinjlim{t^pE}$,
the direct limit as $p\to+\infty$ of the family 
of $D$-linear monomorphisms $\sigma:t^pE\to{t^{p+1}E}$ 
given by $\sigma(t^pa^{x})=w(t)t^{p+1}a^{\frac{x}m}\overline{\mu_m}$,
for all $p\in\mathbb{Z}$ and $x\in{J}$.
It follows easily that 
\[
\Pi\odot\Pi\cong\varinjlim(t^kE\odot{t^kE})=(\oplus_{p\in\mathbb{Z}}t^pE\odot{t^pE})/\sim,
\]
where
$t^ka^x\odot{t^ka^y}\sim{t^{k+1}}a^{\frac{x}m}\overline{\mu_m}\odot{t^{k+1}}a^{\frac{y}m}\overline{\mu_m}$,
for all $k\in\mathbb{Z}$ and $x,y\in{J}$.

Setting $z=y-x$ gives
\[
t^ka^x(1\odot{a^z})\sim
t^{k+1}a^{\frac{x}m}(\overline{\mu_m}\odot\overline{\mu_m}{a^{\frac{z}m}}).
\]
(Here $\pi$ acts diagonally on $\Pi\odot\Pi$.)
We may expand the term in parentheses as
\[
\overline{\mu_m}\odot\overline{\mu_m}a^{\frac{z}m}
=\Sigma_{i,j=0}^{i,j=m-1}w(a)^ia^{-i}(1\odot{w(a)^{i-j}a^{i-j}a^{\frac{z}m}}).
\]
Define a function $f:E\to\Pi\odot\Pi$ by $f(e)=1\odot{e}=e\odot1$ for $e\in{E}$.
Then $f$ is additive and $f(a^x)=w(a)^ma^xf(a^{m-x})$ for all $x$,
since $a^x\odot1=a^x(1\odot{w(a)^ma^{m-x}})$.
The induced map from $E$ to $\Pi\odot_\pi\Pi$ is onto,
and 
\[
\Pi\odot_\pi\Pi\cong{E/N},
\]
where $N$ is the subgroup generated by
\[
\{a^z-w(a^{m-z})a^{m-z},~a^z-w(t)m\Sigma_{k=0}^{m-1}w(a)^ka^{k+\frac{z}m},
~\forall{z\in{J}}\}.
\]
Since $a^z-w(a^{m-z})a^{m-z}\in{N}$, the images $[a^z]$
of the elements $a^z$ with ${0\leq{z}\leq\frac{m}2}$ generate the quotient $E/N$.
Given that $[a^z]=w(a)^{m-z}[a^{m-z}]$,
the conditions $[a^z]=w(t)m\Sigma_{k=0}^{m-1}w(a)^k[a^{k+\frac{z}m}]$ and
$[a^{m-z}]=w(t)m\Sigma_{k=0}^{m-1}w(a)^k[a^{k+\frac{m-z}m}]$ 
are equivalent.

Let $F_s$ be the subgroup of $\Pi\odot_\pi\Pi$ generated by 
$\{[a^z]\mid{m^{s-1}z\in\mathbb{Z}}\}$,
for $s\geq1$.
If $|m|>1$ then the conditions $[a^z]=[w(t)m\Sigma_{k=0}^{m-1}w(a)^k[a^{k+\frac{z}m}]$ in $E/N$,
for $z\in{J}$,  
imply that $F_s$ is generated by 
$\{[a^0]\}\cup\{[a^z]\mid0<2z\leq{m},~m^{s-1}z\in\mathbb{Z}, ~m^{s-2}z\not\in\mathbb{Z}\}$,
for all $s\geq1$,
with a single relation of the form $(1-w(t)m)[a^0]=m^s\sigma$,
where $\sigma$ is a sum of the generators $[a^z]$ with 
$z\in{J_s}$ such that $0<2z<m$, 
and coefficients not divisible by $(1-m)$.
Hence $F_s$ is torsion free, for all $s\geq1$.
Since $\Pi\odot_\pi\Pi$ is the increasing union $\cup_{s\geq0}F_s$,
it  is also torsion free.
\end{proof}

If $m=\pm1$ and $w=1$ then $\Pi\odot_\pi\Pi\cong\mathbb{Z}$.
However, if $m=\pm1$ and $w\not=1$ then  $\Pi\odot_\pi\Pi=Z/2Z$,
and so the theorem does not extend to this case.

Note that the argument of the final paragraph implies that
every generator of $\Pi\odot_\pi\Pi$ is $m$-divisible,
and that $\Pi\odot_\pi\Pi$ is a free $\mathbb{Z}[\frac1m]$-module 
of infinite rank.

\begin{cor}
If $\pi=\mathbb{Z}*_m$  with $|m|>1$ then
$\mathbb{Z}\otimes_{\mathbb{Z}[\pi]}\Gamma_W(\Pi)$ is torsion free.
\end{cor}

\begin{proof}
If $m$ is even this follows immediately from the theorem
and the short exact sequence of Lemma \ref{symmod2}, 
since $H^2(\pi;\mathbb{F}_2)=0$ then.
If $m$ is odd we may apply the final part of Lemma \ref{symmod2}. 
Letting $x$ be the image of $1\in\mathbb{Z}[\pi]$,
we see that $\gamma_\Pi(x)$ generates $\Pi/(2,I_w)\Pi=H^2(\pi;\mathbb{F}_2)$,
while the image of $f(1)=x\odot{x}$ in $\Pi\otimes_\pi\Pi$ is 
not 2-divisible. 
\end{proof}

It is not immediately obvious that the models for $\Pi\odot_\pi\Pi$ in
Lemma \ref{PioPi} and Theorem \ref{BStf} agree when $\pi\cong\mathbb{Z}*_m$.
However (assuming for simplicity that $m\geq1$ and  $w=1$), the relations
\[
t^ka^x\sim_1{t^k}a^xt\mu_m=t^{k+1}a^{\frac{x}m}\mu_m\quad\mathrm{and}
\quad
{t^ka^x}\sim_2(t^ka^x)^{-1}=t^{-k}a^{-m^kx}
\]
together imply that $\Pi\odot_\pi\Pi$ is generated by the image of $E$ 
and that 
\[
a^z~\sim_1~{t}a^{\frac{z}m}\mu_m=\Sigma_{i=0}^{i=m-1}ta^ia^{\frac{z}m}
~\sim_2~\Sigma_{i=0}^{i=m-1}t^{-1}a^{-mi-z}=mt^{-1}a^{-z}
\]
\[
\sim_1~{m}a^{-\frac{z}m}\mu_m~\sim_2~{m}\Sigma_{i=0}^{i=m-1}a^{-i}a^{\frac{z}m}=
ma^{\frac{z}m}\mu_m,
\]
for all $z\in{J}$.
This is enough to see that $\mathbb{Z}[\pi]/(U+\overline\Delta)$ 
is a quotient of $E/N$, as an abelian group,
when $\Delta=(a^m-1,t-\mu_m)\mathbb{Z}[\pi]$.

Can we extend the argument of Theorem \ref{BStf}  in any way?
In particular, does the hypothesis of Theorem \ref{min} hold
for ascending HNN extensions $F*_\varphi$ with base $F$ 
a finitely generated free group and $\varphi$ an endomorphism 
such that $p\prec\varphi(p)$ for all $1\prec{p}$ with respect to some left ordering $\prec$ on $F$?
When $\varphi$ is an automorphism $\pi$ is a semidirect product ${F(r)}\rtimes_\varphi\mathbb{Z}$, 
and the result of Theorem \ref{min} holds by Theorem \ref{freebyZ}.
If $\varphi$ has odd order and $w=1$ then it can be shown that $\Pi\odot_\pi\Pi$ is 2-torsion free.
However, as we have seen, the argument of Theorem \ref{min} itself
must be changed in order to accommodate other semidirect products
${F(r)}\rtimes_\varphi\mathbb{Z}$ and orientation characters $w$. 

\section{$4$-manifolds and $2$-knots}

In this section we shall invoke surgery arguments, 
and so ``4-manifold" and ``$s$-cobordism" shall mean TOP 4-manifold
and (5-dimensional) TOP $s$-cobordism, respectively.
We continue to assume that $\pi$ is a 2-dimensional duality group.

Suppose that $\pi$ is either the fundamental group of a finite graph of groups,
with all vertex groups $\mathbb{Z}$, 
or is square root closed accessible,
or is a classical knot group.
(This includes all $PD_2$-groups, semidirect products
$F(s)\rtimes\mathbb{Z}$ and the solvable groups $\mathbb{Z}*_m$.)
Then $Wh(\pi)=0$, $L_5(\pi,w)$ acts trivially on
the $s$-cobordism structure set $S^s_{TOP}(M)$ 
and the surgery obstruction map $\sigma_4(M):[M,G/TOP]\to{L_4(\pi,w)}$
is onto, for any closed 4-manifold $M$ realizing $(\pi,w)$.
(See Lemma 6.9 and Theorem 17.8 of \cite{Hi}.)

If, moreover, $w_2(\widetilde M)=0$ then 
every 4-manifold homotopy equivalent to $M$ is $s$-cobordant to $M$, 
by Theorem 6.7, Lemma 6.5 and Lemma 6.9 of [Hi].
If $w_2(\widetilde M)\not=0$ there are at most two $s$-cobordism classes 
of homotopy equivalences.
After stabilization by connected sum with copies of $S^2\times S^2$
there are two $s$-cobordism classes, 
distinguished by their KS smoothing invariants (see \cite{KT01}).

If $\pi$ is solvable then 5-dimensional $s$-cobordisms are products 
and stabilization is unnecessary,
so homotopy equivalent 4-manifolds with fundamental group $\pi$ 
are homeomorphic if the universal cover is Spin, 
and there are two homeomorphism types otherwise,
distinguished by their KS invariants.

The Baumslag-Solitar group $\mathbb{Z}*_m$ has such 
a graph-of-groups structure and is solvable, 
so the 5-dimensional TOP $s$-cobordism theorem holds.
Thus if $m$ is even the closed orientable 4-manifold $M$
with $\pi_1(M)\cong\mathbb{Z}*_m$ and $\chi(M)=0$ is unique up to homeomorphism.
If $m$ is odd there are two such homeomorphism types,
distinguished by the second Wu class $v_2(M)$.

Let $\pi$ be a finitely presentable group with $c.d.\pi=2$.
If $H_1(\pi;\mathbb{Z})=\pi/\pi'\cong\mathbb{Z}$ and $H_2(\pi;\mathbb{Z})=0$ 
then $\mathrm{def}(\pi)=1$, by Theorem 2.8 of \cite{Hi}.
If moreover $\pi$ is the normal closure of a single element
then $\pi$ is the group of a 2-knot $K:S^2\to{S^4}$.
(If the Whitehead Conjecture is true every knot group of deficiency 1 
has cohomological dimension at most 2.)
Since $\pi$ is torsion free it is indecomposable, by a theorem of Klyachko
\cite{Kly}.
Hence $\pi$ has one end. 

Let $M=M(K)$ be the closed 4-manifold obtained by surgery on the 2-knot $K$.
Then $\pi_1(M)\cong\pi=\pi{K}$ and $\chi(M)=\chi(\pi)=0$,
and so $M$ is a minimal model for $\pi$.
If $K$ is reflexive it is determined by $M$ and the orbit of its meridian 
under the automorphisms of $\pi$ induced by self-homeomorphisms of $M$.
If $\pi=F(s)\rtimes\mathbb{Z}$ the homotopy type of $M$ is determined by $\pi$,
as explained in \S4 above.
Since $H^2(M;\mathbb{F}_2)=0$ it follows that $M$ is $s$-cobordant 
to the fibred 4-manifold with $\#^s(S^2\times{S^1})$ 
and fundamental group $\pi$.
Knots with Seifert surface a punctured sum
$\#^s(S^2\times{S^1})_o$ are reflexive.
Thus if $K$ is fibred (and $c.d.\pi=2$) it is determined (among all 2-knots) 
up to $s$-concordance and change of orientations by $\pi$ together with 
the orbit of its meridian under the automorphisms of $\pi$ induced by 
self-homeomorphisms of the corresponding fibred 4-manifold.
(This class of 2-knots includes all Artin spins of fibred 1-knots.
See \S6 of Chapter 17 of \cite{Hi} for more on 2-knots with $c.d.\pi=2$.)

A stronger result holds for the group $\pi=\mathbb{Z}*_2$.
This is the group of Fox's Example 10, which is a ribbon 2-knot \cite{Fo62}.
In this case $\pi$ determines the homotopy type of $M(K)$, 
by Theorems \ref{BStf} and \ref{min}.
Since metabelian knot groups have an unique conjugacy class
of normal generators (up to inversion) 
Fox's Example 10 is the unique 2-knot 
(up to TOP isotopy and reflection) with this group.
(If $K$ is any other nontrivial 2-knot such that $\pi{K}$
is torsion free and elementary amenable then
$M(K)$ is homeomorphic to an infrasolvmanifold. 
See Chapters 16-18 of \cite {Hi}.)

Let $\Lambda=\mathbb{Z}[\mathbb{Z}]$.
There is a hermitian pairing $B$ on a finitely generated free $\Lambda$-module 
which is not extended from the integers,
and a closed orientable 4-manifold $M_B$ with $\pi_1(M)\cong\mathbb{Z}$
and such that the intersection pairing on $\pi_2(M_B)$ is equivalent to $B$.
In particular, $M_B$ is not the connected sum of $S^1\times S^3$ with a
1-connected 4-manifold \cite{HT97}.
Let $N_B\subset M_B$ be an open regular neighbourhood of a loop representing a
generator of $\pi_1(M_B)$.
Suppose that $X$ is a closed 4-manifold with fundamental group $\pi$
and that there is an orientation preserving loop $\gamma\subset X$
whose image in $\pi/\pi'$ generates a free direct summand.
(For instance, there is such a loop if $X$ is the total space of an $S^2$-bundle
over an aspherical closed surface $F$ with $\beta_1(F)>1$).
Then $\gamma$ has a regular neighbourhood homeomorphic to $N_B$,
and we may identify these regular neighbourhoods to obtain
$N=M_B\cup_{S^1\times D^3}X$. 
The inclusion of $\langle g\rangle$ into $\pi$ and the projection of
$\pi$ onto $Z$ mapping $g$ to 1 determines a
monomorphism $\gamma:\Lambda\to\mathbb{Z}[\pi]$ and a retraction
$\rho:\mathbb{Z}[\pi]\to\Lambda$.
In particular, 
$\Lambda\otimes_{\mathbb{Z}[\rho]}
(\mathbb{Z}[\pi]\otimes_{\mathbb{Z}[\gamma]}B)
\cong B$.
It follows that as $B$ is not extended from $\mathbb{Z}$ 
neither is $\mathbb{Z}[\pi]\otimes_{\mathbb{Z}[\gamma]}B$.
Therefore $N$ is not the connected sum of $E$ with a 1-connected
4-manifold.

\section{some questions}

We shall collect here some of the questions that have arisen en route.

\begin{enumerate}

\item are strongly minimal $PD_4$-complexes always of $v_2$-type II or III?

\item if $X$ has $v_2$-type I and $c.d.\pi=2$ 
is there a minimal model $f:X\to{Z}$ with $v_2(Z)=0$?

\item must a strongly minimal $PD_4$-complex with $\pi$ 
a nontrivial free product be a connected sum?

\item can we say more about $PD_4$-complexes with $\pi$ infinitely 
ended and $\Pi=0$?

\item are there strongly minimal $PD_4$-complexes with $E^\dagger\cong{E}^3\mathbb{Z}\not=0$?

\item do strongly minimal $PD_4$-complexes always have $k_1=0$?

\item If $X$ is a $PD_4$-complex such that $\pi=\pi_1(X)$ has one end 
and $\Pi=\pi_2(X)$ is projective, 
must $\pi$ be a $PD_4$-group?

\item to what extent do $k_2$ and $v_2$ determine each other?

\item in Theorem \ref{Postcd2} must $Y$ be a $PD_4$-complex?

\item can we extend Theorems \ref{min} and \ref{BStf} to encompass 
the known results for $\pi$ a semidirect product $F(r)\rtimes\mathbb{Z}$
(at least when $w=1$)?

\item can we relax the running hypothesis that $\pi$ should have one end?

\end{enumerate}

The final four questions are of most interest for
the present work.

\end{document}